\documentclass[11pt,a4paper,twosides]{amsart}
\usepackage{amssymb,amsmath,amsthm,mathrsfs,graphicx,xcolor}
\usepackage{leftidx}
\theoremstyle{plain}
\newtheorem{theorem}{Theorem}[section]
\newtheorem{definition}[theorem]{Definition}
\newtheorem{lemma}[theorem]{Lemma}

\newtheorem{proposition}[theorem]{Proposition}

\theoremstyle{remark}
\newtheorem{remark}[theorem]{Remark}

\allowdisplaybreaks[4]

\numberwithin{equation}{section}

\newcommand{\C}{\mathbb{C}}
\newcommand{\R}{\mathbb{R}}
\newcommand{\Z}{\mathbb{Z}}

\renewcommand{\Im}{\operatorname{Im}}
\renewcommand{\Re}{\operatorname{Re}}
\newcommand{\I}{\infty}

\newcommand{\norm}[1]{\left\lVert #1\right\rVert}

\def\({\left(}
\def\){\right)}
\def\<{\left\langle}
\def\>{\right\rangle}
\def\le{\leqslant}
\def\ge{\geqslant}

\def \l{\lambda}

\newcommand{\eps}{\varepsilon}

\newcommand{\la}{\lambda}

\newcommand{\pt}{\partial}

\DeclareMathOperator{\sign}{sign}

\DeclareMathOperator{\rank}{rank}
\DeclareMathOperator{\tr}{tr}

\newcommand{\todayd}{\the\year/\the\month/\the\day}

\theoremstyle{definition}

\newcommand{\ol}{\overline}

\begin{document}
\title[System of cubic NLS euation in 1d]
{Asymptotic behavior in time of solution to system\\
of cubic nonlinear Schr\"odinger equations\\
in one space dimension}

\author{Satoshi Masaki}
\address{Department of systems innovation, 
Graduate school of Engineering Science, 
Osaka University, Toyonaka Osaka, 560-8531, Japan}
\email{masaki@sigmath.es.osaka-u.ac.jp}

\author{Jun-ichi Segata}
\address{Faculty of Mathematics, Kyushu University, 
Fukuoka, 819-0395, Japan}
\email{segata@math.kyushu-u.ac.jp}

\author{Kota Uriya}
\address{Department of Applied Mathematics, Faculty of Science, 
Okayama University of Science, Okayama, 700-0005, Japan}
\email{uriya@xmath.ous.ac.jp}

\keywords{Nonlinear Schr\"{o}dinger equation, Asymptotic behavior of solutions, 
Long-range scattering, Normalization of systems, Matrix representation}
\subjclass[2010]{Primary 35Q55, Secondary 35A22, 35B40}

%\dedicatory{Dedicated to Professor Tohru Ozawa on the occasion of his sixtieth birthday}

\begin{abstract}
In this paper, we consider the large time asymptotic behavior of solutions to 
systems of two cubic nonlinear Schr\"{o}dinger equations in one space dimension.
It turns out that for a system there exists a small solution of which asymptotic profile is a sum of two parts oscillating in a different way. 
This kind of behavior seems new.
Further, several examples of systems which admit solution with several types of behavior such as modified scattering, nonlinear amplification, and nonlinear dissipation, are given.
We also extend our previous classification result of nonlinear cubic systems.
\end{abstract}

\maketitle

\section{Introduction}
The paper is devoted to the study of the asymptotic behavior in time 
of solutions to the Cauchy problem of the following two systems of 
cubic nonlinear Schr\"odinger (NLS) equations in one space dimension:
The first one is
\begin{equation}\label{E:sysnew1}
\left\{
\begin{aligned}
&i\partial_t u_1 + \frac12\partial_x^2 u_1
= 3\l_1 |u_1|^2u_1, &&t\in\R,\ x\in\R,\\
&i\partial_t u_2 + \frac12\partial_x^2 u_2
= \l_6 (2|u_1|^2u_2+u_1^2\overline{u_2}), &&t\in\R,\ x\in\R,\\
& u_1(0,x)=u_{1,0}(x),\qquad u_2(0,x)=u_{2,0}(x), &&x\in\R
\end{aligned}
\right.
\end{equation}
and 
\begin{equation}\label{E:d21}
\left\{
\begin{aligned}
&i\partial_t u_1 + \frac12\partial_x^2 u_1
= 0, &&t\in\R,\ x\in\R,\\
&i\partial_t u_2 + \frac12\partial_x^2 u_2
=  3|u_1|^2u_1, &&t\in\R,\ x\in\R,\\
& u_1(0,x)=u_{1,0}(x),\quad u_2(0,x)=u_{2,0}(x), &&x\in\R
\end{aligned}
\right.
\end{equation}
is the second,
where $u_j:\R\times\R\to\C$ ($j=1,2$) are unknown functions, 
$u_{j,0}:\R\to\C$ ($j=1,2$) are given functions, and  $\l_1$ and  $\l_6$  
are real constants satisfying $(\l_1,\l_6)\neq(0,0)$ and
$	(\lambda_6 - \lambda_1)(\lambda_6 - 3\lambda_1) \ge 0$.
%This is a unified notation of \eqref{E:sysnewa} and \eqref{E:sysnewb}, which
%appear in Theorems \ref{T:main2} and \ref{T:main3}, below.
%where $ \lambda_5 $ is a real constant.
%In this paper, we are interested in time global existence and the asymptotic
%behavior of solutions to \eqref{E:sysnew1}. 

The systems \eqref{E:sysnew1} and \eqref{E:d21} are particular cases of  
\begin{equation}\label{eq:NLS1}
	\left\{
	\begin{aligned}
		i\partial_t u_1 + \frac12 \partial_x^2 u_1 
		&= 3\lambda_1|u_1|^2u_1 + \lambda_2(2|u_1|^2u_2+u_1^2\overline{u_2}) 
		+ \lambda_3(2u_1|u_2|^2+\overline{u_1}u_2^2) +3 \lambda_4|u_2|^2u_2,\\
		i\partial_t u_2 + \frac12 \partial_x^2 u_2 
		&= 3\lambda_5|u_1|^2u_1 + \lambda_6(2|u_1|^2u_2+u_1^2\overline{u_2}) 
		+ \lambda_7(2u_1|u_2|^2+\overline{u_1}u_2^2) + 3\lambda_8|u_2|^2u_2,\\
	\end{aligned}
	\right.
\end{equation}
where $t\in \R$, $x\in \R$ and
$ \lambda_j\ (j = 1, \cdots, 8)$ are real constants. 
%Under the special choice of the coefficients
%\begin{equation}
%	c_1 = 3\lambda_1, \ c_8 = 2\lambda_6, \ c_9 = \lambda_6, \ 
%	c_{j} = 0\quad(j\neq1,8,9)%c_2 = c_3 = c_4 = c_5 = c_6 = c_7 = c_{10} = c_{11} = c_{12} = 0, 
%\end{equation}
%we obtain the system \eqref{E:sysnew1}. 
The system \eqref{eq:NLS1} includes several important physical models such as 
Manakov system  \cite{M} or a system describing spinor Bose-Einstein condensate 
\cite{IMW}. 
Due to a classification result in our previous study \cite{MSU}, 
systems of the form \eqref{eq:NLS1} are classified according to the number of mass-like conserved quantities,
which is connected to the complexity of the behavior of solution.
In view of the classification theory, the study of the peculiar systems \eqref{E:sysnew1} and \eqref{E:d21} have
an importance.
We discuss it in detail below.

It is well known that the cubic nonlinearity is critical in one dimension from the view point of the 
asymptotic behavior of solutions to the NLS equations and systems.
As for the single %NLS %nonlinear Schr\"odinger 
equation
\begin{equation}\label{eq:sNLS}
	i\partial_t u + \frac{1}{2}\partial_x^2 u = \lambda |u|^{p-1}u, 
	\qquad t \in \R,\ x \in \R^d,
\end{equation}
where $\la\in\R$, $p=1+2/d$ is known to be the critical exponent 
(see \cite{B,St,TY}). 
In the critical case $p=1+2/d$, the long-range scattering occurs, 
namely, a class of solutions to \eqref{eq:sNLS} satisfies
\begin{equation}
	%\left\|
	u(t) \to t^{-\frac{d}{2}}W\left(\frac{\cdot}{t}\right)e^{\frac{i|\cdot|^2}{2t}
	-i\lambda |W(\frac{\cdot}{t})|^{\frac{2}{d}}\log t-i\frac{\pi}{4}d}%\right\|_{L^2} 
	\quad \text{as}\ t \to \infty,
\end{equation}
for some function $W \in L^\infty$ in a suitable topology. This kind of asymptotic behavior is also called 
the modified scattering because it involves a phase correction 
(see Ozawa~\cite{O} and Ginibre-Ozawa~\cite{GO} for the final value problem 
and Hayashi-Naumkin~\cite{HN} for the initial value problem). 

The system \eqref{eq:NLS1} is cubic and the asymptotic behavior of solutions depends on the coefficients of the nonlinearities. 
%It is now understood that 
Our underlying motivation of the study in the present paper is 
to find the all possible behavior to the system of the form \eqref{eq:NLS1} and 
%It is also motivated to study 
more general system
\begin{equation}
\label{eq:NLS}
\left\{
\begin{aligned}
&i\partial_t u_1 + \frac12\partial_x^2 u_1 = 
c_1|u_1|^2u_1+c_2|u_1|^2u_2+c_3u_1^2\overline{u_2}\\
&\qquad\qquad\qquad\qquad\qquad+c_4u_1|u_2|^2+c_5\overline{u_1}u_2^2+c_6|u_2|^2u_2, \\
%&&t \in \R,\ x \in \R,\\
&i\partial_t u_2 + \frac12\partial_x^2 u_2 =
c_7|u_1|^2u_1+c_8|u_1|^2u_2+c_9u_1^2\overline{u_2}\\
&\qquad\qquad\qquad\qquad\qquad+c_{10}u_1|u_2|^2
+c_{11}\overline{u_1}u_2^2+c_{12}|u_2|^2u_2,
%&&t \in \R,\ x \in \R, 
%\\
%&u_1(0,x)=u_{1,0}(x),\qquad u_2(0,x)=u_{2,0}(x),
%&&x \in \R,
\end{aligned}
\right.
\end{equation}
where %$u_j:\R\times\R\to\C$ ($j=1,2$) are unknown functions, 
%$u_{j,0}:\R\to\C$ ($j=1,2$) are given functions, and 
$c_j\ (j=1,\cdots,12)$ are real constants. 
Even though the systems of  the form \eqref{eq:NLS1} or \eqref{eq:NLS} 
are somewhat restricted in the sense that the nonlinearities do not contain 
derivatives of unknowns and that the coefficients are real,
the variety of behavior of solutions to these systems is still richer than the 
single equations has. 
\smallbreak

Now,
let us recall previous results on the large time behavior of solutions to systems 
of %nonlinear Schr\"odinger 
the NLS equations.
As for the cubic system in one dimension with nonlinearities with/without derivatives, 
the \emph{null condition}, a sufficient condition on nonlinearity 
for the existence of a non-trivial solution which asymptotically behaves like a free solution, is obtained 
in \cite{KSa} (see \cite{Tsutsumi} for the single equation).
In \cite{KN}, the long-range scattering is obtained for a matrix-valued equation.
As mentioned above, a quadratic nonlinearity is critical in two dimensions and 
the asymptotic behavior of solutions to the two dimensional quadratic systems 
is also extensively studied.
In this case, the ratio of the masses of two components matters.
This phenomenon is called mass-resonance 
(see \cite{HLN,HLN1,HLN2,KLS} for systems with non-derivative nonlinearities and \cite{KS,SaSu} for those with derivative nonlinearities).
The phenomenon is studied also for the one dimensional cubic systems (see \cite{NST,U}).

In this paper, we restrict ourselves to the case where the coefficients of the nonlinearities are real numbers, as mentioned above.
It is known that the NLS %nonlinear Schr\"odinger 
equations/systems with imaginary coefficients admit solution with
different kinds of behavior.
A typical example is the single NLS equation \eqref{eq:sNLS}. If the coefficient $\la$ is an imaginary number,
a nonlinear amplification/dissipation phenomenon takes place.
More precisely, if $\la \in \C \setminus \R$ then \eqref{eq:sNLS} is dissipative in one time direction and amplifying in the other direction.
The sign of $\Im \la$ decides which direction is the dissipative direction.
In the dissipative direction, it has shown that the small solution decays faster than 
the free Schr\"odinger evolution (see \cite{HLN3,Ho,K2,OgSato,Sato,S}).
On the other hand, Kita \cite{K} showed that, in the amplifying direction, there exists an arbitrarily small data which gives a blowing-up solution.
Systems with the dissipative structure is also intensively studied.
See \cite{Kim} for systems with non-derivative nonlinearities and \cite{LSu} for those with derivative nonlinearities.
Recently, 
new type of behavior of solution is found in a certain system of the cubic NLS equations 
in one dimension in \cite{LNSS1,LNSS2}.

It turns out that there exists a system of the form \eqref{eq:NLS}  such that the amplification/dissipation phenomenon takes place,
although the coefficients of the nonlinearities are real.
Further, the system admits the following three types of solutions;
(i) blowup forward in time and dissipative decay backward in time;
(ii) blowup backward in time and dissipative decay forward in time;
(iii) blowup for both time directions.
This system is an evidence for the richness of the variety of behaviors for systems.
We discuss this system in Appendix B.

%It is important to know that how nonlinear term affect the asymptotic behavior of the solution. 
In \cite{MSU}, the authors considered the system of 
cubic nonlinear Klein-Gordon equations in one space dimension:
\begin{equation}\label{E:sys}
\left\{
\begin{aligned}
&(\square + 1)u_1 
= \lambda_1u_1^3 + \l_2u_1^2u_2 + \l_3u_1u_2^2 + \l_4u_2^3, &&\quad t\in\R,\ x\in\R,\\
&(\square + 1)u_2 
= \lambda_5u_1^3 + \l_6u_1^2u_2 + \l_7u_1u_2^2 + \l_8u_2^3, &&\quad t\in\R,\ x\in\R,
\end{aligned}
\right.
\end{equation}
where $u_j: \R\times \R \to \R$ ($j=1,2$) are real-valued unknowns, $\square = \partial_t^2 - \partial_x^2$ is
the d'Alembertian, and $\l_1,\dots, \l_8$  are real constants. 
%We revealed that 
%certain matrix defined by the coefficients of the nonlinearity is useful to classify the asymptotic 
%behavior of the solutions. 
%Since the system (\ref{E:sys}) is closed under the linear transformation of the unknowns, 
%we reach the classification of ``the set of systems'' by the equivalence relation 
%naturally induced by the linear transformation of the unknowns. 
%Namely, we reduce the analysis of ``the set of systems'' into the analysis of the quotient sets 
%or their representatives. 
%In addition, it is important that we notice this equivalence relation is well described by an 
%identification with a matrix. 
%This enables us to analyze the corresponding matrix instead of considering the system itself. 
%In particular, we characterize several systems whose asymptotic behavior of the solution is known 
%in terms of the matrix and specify all systems equivalent to them with an explicit reduction 
%procedure. 
%In the previous study, both the solution and the coefficients are real. 
%Here, we consider the case where the solution is complex-valued and 
%the coefficients are real number, so that we have more variety of 
%nonlinear interactions. Thus it is much difficult to classify the asymptotic behavior of the solution to 
%the system \eqref{eq:NLS}. However, we notice that our
%
%このシステムは未知変数の線形変換について閉じている.
%この未知変数の線形変換により、自然にシステムの間の同値関係が定義される。
This system is closed under the linear transformation of unknowns.
An equivalence relation between two systems is then naturally introduced by the linear transformation of unknowns.
%システムの集合をこの同値関係で割った商集合を考察することによって、上のシステムの分類を\cite{MSU}で行った。
They give a classification result by considering a quotient set of the systems with respect to the equivalence relation.
The classification result is applicable to \eqref{eq:NLS1}.
%
%\begin{equation}\label{eq:NLS1}
%	\left\{
%	\begin{aligned}
%		i\partial_t u_1 + \frac12 \partial_x^2 u_1 
%		&= \lambda_1|u_1|^2u_1 + \lambda_2(2|u_1|^2u_2+u_1^2\overline{u_2}) 
%		+ \lambda_3(2u_1|u_2|^2+\overline{u_1}u_2^2) + \lambda_4|u_2|^2u_2,\\
%		i\partial_t u_2 + \frac12 \partial_x^2 u_2 
%		&= \lambda_5|u_1|^2u_1 + \lambda_6(2|u_1|^2u_2+u_1^2\overline{u_2}) 
%		+ \lambda_7(2u_1|u_2|^2+\overline{u_1}u_2^2) + \lambda_8|u_2|^2u_2,\\
%	\end{aligned}
%	\right.
%\end{equation}
%where $ \lambda_j\ (j = 1, \cdots, 8)$ are real constants. 
%これは変数変換に対する係数の変化がKGの場合のそれと同じだからである。
This is because the change of coefficients caused by the linear transformation of unknown is identical to that for \eqref{E:sys}.
This agrees with the fact that the asymptotic profile for a solution of \eqref{E:sys}
and that for a solution of \eqref{eq:NLS1} are, at least formally, described by the same ODE system
\begin{equation}\label{eq:limitODE}
	\left\{
	\begin{aligned}
		i\partial_t \alpha_1 
		&= 3\lambda_1|\alpha_1|^2\alpha_1 + \lambda_2(2|\alpha_1|^2\alpha_2+\alpha_1^2\overline{\alpha_2}) 
		+ \lambda_3(2\alpha_1|\alpha_2|^2+\overline{\alpha_1}\alpha_2^2) +3 \lambda_4|\alpha_2|^2\alpha_2,\\
		i\partial_t \alpha_2  
		&= 3\lambda_5|\alpha_1|^2\alpha_1 + \lambda_6(2|\alpha_1|^2\alpha_2+\alpha_1^2\overline{\alpha_2}) 
		+ \lambda_7(2\alpha_1|\alpha_2|^2+\overline{\alpha_1}\alpha_2^2) + 3\lambda_8|\alpha_2|^2\alpha_2.
	\end{aligned}
	\right.
\end{equation}
Hereinafter, we refer to the system as \emph{limit ODE system}. 
The classification result is also applicable to the system \eqref{eq:limitODE}.

Let us quickly review the classification result in \cite{MSU} by taking \eqref{eq:NLS1} as an example.
The key ingredient is the introduction of a matrix representation of a system:
A system \eqref{eq:NLS1} is identified with a matrix
\begin{equation}\label{eq:defA}
	A = \begin{pmatrix}
	\l_2  & -3\l_1 + \l_6& -3\l_5 \\
	\l_3  & -\l_2 + \l_7  & -\l_6 \\
	3\l_4 & 3\l_8 - \l_3 &-\l_7
	\end{pmatrix}.
\end{equation}
It then turns out that the change caused by the linear transformation of unknowns is clearly formulated as a matrix manipulation
and, moreover, the characteristic properties such as conservation laws are well described by the matrix (see Section A.4).
In particular, $\rank A$ is an invariant quantity which
indicates the number of mass-like conserved quantities.
Roughly speaking, the behavior of solution becomes complicated as $\rank A$ increases.
In \cite{MSU}, the authors classify two subsets of systems. One is the set of systems such that $\rank A = 1$ and the other is the set of
systems such that $B=O$, where $B$ is the matrix defined from the coefficients of the system as
\begin{equation}\label{eq:defB}
	B:= \begin{pmatrix}
	-12 \la_5 & 3 (\la_1 -\la_6 ) & 2(\la_2-\la_7) \\
	3 (\la_1 -\la_6 ) & 2(\la_2-\la_7) & 3(\la_3 -\la_8 ) \\
	2(\la_2-\la_7) & 3(\la_3 -\la_8 ) & 12 \la_4
	\end{pmatrix}
\end{equation}
and $O \in M_3(\R)$ is the zero matrix.
The quotient set  of the first subset contains 9 equivalent classes. 
And the quotient set of the second does 5 equivalent classes.

In Appendix A, we review the classification result in more detail and extend it to the set of systems of the form \eqref{eq:NLS}.
Note that the system of the form \eqref{eq:NLS} is identified with a vector $(c_j)_{1 \le j \le 12} \in \R^{12}$.
Thus, it is difficult to correspond it with a $3\times3$ matrix like \eqref{eq:defA}. 
%\eqref{eq:NLS}は$\R^{12}$の元と一対一に対応する. したがって, \eqref{eq:defA}のようにある$M_3(\R)$の元と同一視することは難しい.
%そのために, matrix representation を拡張した新しい representation を導入する.
%これにより,  同値関係を簡明に理解できることと, 方程式の特徴的な性質の情報を引き出すことができる.
Hence, we introduce a new way to represent a system. This is an extension of the matrix representation of \eqref{eq:NLS1}.
This enables us to formulate the equivalence relation in a clear way and describe the validity of conservation laws for the generalized system \eqref{eq:NLS}.
As an application, a global existence result for \eqref{eq:NLS} is shown in Proposition \ref{P:massconservation}.
%拡張された表示の一つの応用が次の\eqref{eq:NLS}大域存在の結果である.
Note that the 
local well-posedness of the Cauchy problem of the system \eqref{eq:NLS} is obtained by a standard theory in several function spaces such as the Lebesgue space $ L^2(\R) $, the Sobolev space $ H^1(\R) $ and so on, see \cite{Caz} for instance. 

%最後に, 冒頭で述べたシステムについて分類理論から考える.
Let us now discuss systems \eqref{E:sysnew1} and \eqref{E:d21} in view of the classification result.
As for the systems \eqref{E:sys} and \eqref{eq:NLS1}, the behavior of solutions is previously studied in the case that
$\rank A  \le 1$ or $B=O$.
%においては, 挙動が分かっていたのはいくつかの例外的な場合を除くと $\rank A \le 1$ の場合に限られている.
%その例外的な場合とは,
%となる場合である, 
%この条件は別の形の保存則の存在と関係している.
The conditions are connected to the existence of conserved quantities for theses systems and also for the corresponding limit ODE system \eqref{eq:limitODE}.
%我々は, \eqref{E:sysnew1}を$(\la_6-\la_1)(\la_6-3\l_1) \ge 0$ という条件のもとで考えている.
Notice that the system \eqref{E:d21} satisfies  $\rank A=1$.
The end-point case $\l_6=3\l_1$ of \eqref{E:sysnew1} also corresponds to the case $\rank A = 1$.
In the other end-point case $\l_6=\l_1$, one has $B=O$. 
We here remark that these cases are previously studied for the corresponding Klein-Gordon system \eqref{E:d21} in \cite{Su,MSU}.
%\cite{MSU}において, に対応する場合と
%\eqref{E:sysnew1} における$(\la_6-\la_1)(\la_6-3\l_1) = 0$の場合の挙動が対応するKG sys に対して与えられた.
%\eqref{E:sysnew1} における  は $\rank A = 1$となる場合である. 
%The other end point case $\l_6=\l_1$ は $\rank A = 2$となるが, 上で述べた例外的な場合, $B=O$の場合である. 
When
$(\la_6-\la_1)(\la_6-3\l_1) > 0$, we have $\rank A = 2 $ and $B \neq O$. 
As far as the authors know, this case is new.
%The classification suggests that the behavior of solutions in this case is complicated, compared to the previously studied cases.
It is revealed that the asymptotic profile for the second component involves two parts which oscillate in a different way.
This is the main result of this paper.

\subsection{Main results}
To state main results, 
we define the weighted $L^{2}$ space $ H^{0,1}(\R) $ by 
\begin{equation}\notag
	H^{0,1}(\R) = \left\{ f \in L^2(\R) \mid \|f\|_{H^{0,1}} = \|\langle x \rangle f\|_{L^2} < \infty\right\},
\end{equation}
where $ \langle x \rangle = \sqrt{1+|x|^2}$. 

Let us first consider the case $ (\lambda_6-\lambda_1)(\lambda_6-3\lambda_1) > 0$ of \eqref{E:sysnew1}.
In this case, one has $\rank A = 2$ and $B \neq O$, where $A$ and $B$ are defined in \eqref{eq:defA} and \eqref{eq:defB}, respectively.
We have 
the following result on asymptotic behavior of solution to (\ref{E:sysnew1}).

\begin{theorem}\label{T:main_add}
Suppose $ (\lambda_6-\lambda_1)(\lambda_6-3\lambda_1) > 0$.
%$\l_6=3\l_1 \neq 0$ or $\l_6=\l_1 \neq 0$ in \eqref{E:sysnew1}. 
Let $0<\gamma<\delta<1/100$. Then there exists $\varepsilon_0>0$ such 
that for any $u_{j,0}\in H^{1}(\R)\cap H^{0,1}(\R)$ satisfying 
$%\displaystyle{
\varepsilon:=\sum_{j=1}^2(\|u_{j,0}\|_{H^{1}}+\|u_{j,0}\|_{H^{0,1}})\le\varepsilon_0$, %}$, 
there exists a unique global solution $u_j\in C(\R,H^{1}(\R)\cap H^{0,1}(\R))$ 
of (\ref{E:sysnew1}) %with $u_j(0,x)=u_{j,0}(x)$ 
satisfying 
\begin{align*}
%\|u_1(t)\|_{H_x^1}+\|Ju_1(t)\|_{L_x^2}&\lesssim&\varepsilon\langle t\rangle^{C_1\varepsilon^2},\\
\|u_1(t)\|_{H_x^{0,1}}&\lesssim \varepsilon\langle t\rangle^{\gamma},\quad
\|u_1(t)\|_{L_x^{\infty}}\lesssim \varepsilon\langle t\rangle^{-\frac12},\\
\|u_2(t)\|_{H_x^{0,1}}&\lesssim \varepsilon\langle t\rangle^{\delta},\quad
\|u_2(t)\|_{L_x^{\infty}}\lesssim \varepsilon\langle t\rangle^{-\frac12}
\end{align*}
for any $t\in\R$. %where $J=x+it\pt_x$. 
Furthermore, there exist two functions $W_1, W_2 \in L^{\infty}$ such that 
\begin{align}
u_1(t)
&=t^{-\frac12}W_1\left(\frac{x}{t}\right)e^{\frac{ix^2}{2t}
-i3\lambda_1\left|W_1\left(\frac{x}{t}\right)\right|^2\log t-i\frac{\pi}{4}}
+O(t^{-\frac34+\gamma}),\label{iasym1}\\
		u_2(t) &= t^{-\frac12}{\bf 1}_{\{W_1 \neq 0\}}\left(\frac{x}{t}\right)\bigg[\lambda_6\left(\frac{W_1^2}{|W_{1}|^{2}}W_2\right)\left( \frac{x}{t} \right)
		e^{i(-3\lambda_1-\lambda_c)|W_1(\frac{x}{t})|^2\log t} \label{iasym2}\\
		&\qquad \qquad \qquad  \qquad + (3\la_{1}-2\la_{6}+\la_c)
		\overline{W_{2}}\left( \frac{x}{t} \right)
		e^{i(-3\lambda_1+\lambda_c)|W_1(\frac{x}{t})|^2\log t}\bigg]e^{\frac{ix^2}{2t}-i\frac{\pi}{4}}
		\nonumber\\
		&\quad+ O(t^{-\frac34+\delta}),\nonumber
		\end{align}
in $L^{\infty}(\R)$ as $t\to\infty$, where $\lambda_c = \sqrt{3(\lambda_6-\lambda_1)(\lambda_6-3\lambda_1)}$. 
Similar asymptotic formulas for $u_j$ hold 
for $t\to-\infty$.
\end{theorem}

\begin{remark}
If we assume the opposite inequality $ (\lambda_6-\lambda_1)(\lambda_6-3\lambda_1) < 0 $, 
we formally obtain the same asymptotic profile of the solution. 
However, since $\lambda_c$ is an imaginary number in this case,
we see that one part of the asymptotic profile grows 
and the other part decays, which is problematic in obtaining rigorous asymptotics.
%Therefore it is difficult to obtain the asymptotic behavior of the solution. 
\end{remark}

Let us move to the limiting case, $ \lambda_6 =  \lambda_1$ or $ \lambda_6 = 3\lambda_1 $, of \eqref{E:sysnew1}.
Remark that $\rank A =1$ if $\la_6 = 3 \la_1$ and $B=O$ if $\l_6=\la_1$.
In this case, we have the following result. % on asymptotic behavior of solution. 

\begin{theorem}\label{T:main4} 
Suppose $\l_6=3\l_1 \neq 0$ or $\l_6=\l_1 \neq 0$.  
Let $0<\gamma<\delta<1/100$. Then there exists $\varepsilon_0>0$ such 
that for any $u_{j,0}\in H^{1}(\R)\cap H^{0,1}(\R)$ satisfying 
$%\displaystyle{
\varepsilon:=\sum_{j=1}^2(\|u_{j,0}\|_{H^{1}}+\|u_{j,0}\|_{H^{0,1}})\le\varepsilon_0$, %}$, 
there exists a unique global solution $u_j\in C(\R,H^{1}(\R)\cap H^{0,1}(\R))$ 
of (\ref{E:sysnew1}) %with $u_j(0,x)=u_{j,0}(x)$ 
satisfying 
\begin{align*}
%\|u_1(t)\|_{H_x^1}+\|Ju_1(t)\|_{L_x^2}&\lesssim&\varepsilon\langle t\rangle^{C_1\varepsilon^2},\\
\|u_1(t)\|_{H_x^{0,1}}&\lesssim \varepsilon\langle t\rangle^{\gamma},\quad
\|u_1(t)\|_{L_x^{\infty}}\lesssim \varepsilon\langle t\rangle^{-\frac12},\\
\|u_2(t)\|_{H_x^{0,1}}&\lesssim \varepsilon\langle t\rangle^{\delta},\quad
\|u_2(t)\|_{L_x^{\infty}}\lesssim \varepsilon\langle t\rangle^{-\frac12}\log\langle t\rangle
\end{align*}
for any $t\in\R$. %where $J=x+it\pt_x$. 
Furthermore, there exist two functions $W_j\in L^{\infty}$ such that 
\begin{align}
u_1(t)
&=t^{-\frac12}W_1\left(\frac{x}{t}\right)e^{\frac{ix^2}{2t}
-i3\lambda_1\left|W_1\left(\frac{x}{t}\right)\right|^2\log t-i\frac{\pi}{4}}
+O(t^{-\frac34+\gamma}),\label{asym1}\\
u_2(t)&=
t^{-\frac12}\left\{W\left(\frac{x}{t}\right)\log t+W_2\left(\frac{x}{t}\right)\right\}
e^{\frac{ix^2}{2t}-i3\la_1\left|W_1\left(\frac{x}{t}\right)\right|^2\log t-i\frac{\pi}{4}} +O(t^{-\frac34+\delta})
\label{asym2}
\end{align}
in $L^{\infty}(\R)$ as $t\to\infty$, where $W$ is given by 
\begin{align*} 
W(y)=
\left\{
\begin{aligned}
&-6i\la_1W_1(y)\Re\left[W_1(y)\overline{W_2(y)}\right] &\ \text{if}\ \la_6=3\la_1,\\
&2\la_1W_1(y)\Im\left[W_1(y)\overline{W_2(y)}\right]      &\text{if}\ \la_6=\la_1.\ 
\end{aligned}
\right.
\end{align*} 
Similar asymptotic formulas for $u_j$ hold 
for $t\to-\infty$.
\end{theorem}
\begin{remark}
An intuitive summary of Theorems \ref{T:main_add} and \ref{T:main4} is as follows.
In the non-limiting case $ (\lambda_6-\lambda_1)(\lambda_6-3\lambda_1) > 0$, we have $\la_c>0$ and hence the asymptotic profile of the second component $u_2$ has two parts which have different phase modifications, i.e., which oscillate in a different way.
However, in the limiting case $ (\lambda_6-\lambda_1)(\lambda_6-3\lambda_1) = 0$, we have $\la_c=0$ and the oscillation of the two parts become the same.
The coincidence causes a logarithmic amplitude correction together with the fact that the coefficients of the two parts becomes the same when
$\la_6=\la_1$ and the same modulus with the different signs when $\la_6=3\la_1$.
\end{remark}

\begin{remark}
In the same way as in the proof of Theorem \ref{T:main4}, 
it is not hard to see that 
\begin{eqnarray*}
\|u_2(t)\|_{L_x^2}\lesssim \varepsilon\log t 
\end{eqnarray*}
for any $t\ge2$. Furthermore, we see that $W_{j}\in L^{2}\cap L^{{\infty}}$ 
and 
\begin{align*} 
u_2(t)=
t^{-\frac12}\left\{W\left(\frac{x}{t}\right)\log t+W_2\left(\frac{x}{t}\right)\right\}
e^{\frac{ix^2}{2t}-i3\la_1\left|W_1\left(\frac{x}{t}\right)\right|^2\log t-i\frac{\pi}{4}} +O(t^{-\frac14+\delta})
\end{align*}
in $L^2(\R)$ as $t\to\infty$, where $W$ is given in 
Theorem \ref{T:main4}. From this asymptotic formula, we have 
the lower bound of $L^{2}$ norm of $u_{2}$:
\begin{eqnarray*}
\|u_2(t)\|_{L_{x}^2}\ge\|W\|_{L_{x}^{2}}\log t-C\varepsilon.
\end{eqnarray*}
for $t\ge2$.
\end{remark}

%We also consider the system of cubic nonlinear Sch\"odinger equations, 
%whose corresponding 
We turn to \eqref{E:d21}.
The asymptotic behavior of solutions to the corresponding
system of nonlinear Klein-Gordon equations is studied in Sunagawa~\cite{Su}.
%\begin{equation}\label{E:d21}
%\left\{
%\begin{aligned}
%&i\partial_t u_1 + \frac12\partial_x^2 u_1
%= 0, &&t\in\R,\ x\in\R,\\
%&i\partial_t u_2 + \frac12\partial_x^2 u_2
%= \l_5 |u_1|^2u_1, &&t\in\R, x\in\R,\\
%& u_1(0,x)=u_{1,0}(x),\quad u_2(0,x)=u_{2,0}(x), &&x\in\R,
%\end{aligned}
%\right.
%\end{equation}
%where $ \lambda_5 $ is a real constant.
As for \eqref{E:d21},
merely the growth of the $ L^2 $-norm 
of the solution in a logarithmic rate is given (see \cite{KSa}). Here, we obtain the explicit asymptotic formula of the solution. 

\begin{theorem}\label{T:main51} 
Let $0<\gamma<1/100$. 
Then there exists $\varepsilon_0>0$ such 
that for any $u_{j,0}\in H^{1}(\R)\cap H^{0,1}(\R)$ satisfying 
$%\displaystyle{
\varepsilon:=\sum_{j=1}^2\|u_{j,0}\|_{H^{1}}+\|u_{j,0}\|_{H^{0,1}}\le\varepsilon_0$, %}$, 
there exists a unique global solution $u_j\in C(\R,H^{1}(\R)\cap H^{0,1}(\R))$ 
of (\ref{E:d21}) satisfying 
\begin{align*}
%\|u_1(t)\|_{H_x^1}+\|Ju_1(t)\|_{L_x^2}&\lesssim&\varepsilon\langle t\rangle^{C_1\varepsilon^2},\\
\|u_1(t)\|_{H_x^{0,1}}&\lesssim \varepsilon,\quad
\|u_1(t)\|_{L_x^{\infty}}\lesssim \varepsilon\langle t\rangle^{-\frac12},\\
\|u_2(t)\|_{H_x^{0,1}}&\lesssim \varepsilon\langle t\rangle^{\gamma},\quad
\|u_2(t)\|_{L_x^{\infty}}\lesssim \varepsilon\langle t\rangle^{-\frac12}\log\langle t\rangle
\end{align*}
for any $t\in\R$. %where $J=x+it\pt_x$. 
Furthermore, let $W_{1}:=\hat{u}_{1,0}$. Then there exists a function $W_2\in L^{\infty}$ such that 
\begin{align}
u_2(t)=-i 
t^{-\frac12}\left\{\left|W_1\left(\frac{x}{t}\right)\right|^2W_1\left(\frac{x}{t}\right)
\log t+W_2\left(\frac{x}{t}\right)\right\}
e^{\frac{ix^2}{2t}-i\frac{\pi}{4}}+O(t^{-\frac34+\delta})\label{asym3}
\end{align}
in $L^{\infty}(\R)$ as $t\to\infty$.
Similar asymptotic formula for $u_2$ holds for $t\to-\infty$.
\end{theorem}

\begin{remark}
%By Theorems \ref{T:main4} and \ref{T:main51}, one sees that the system \eqref{E:d21} 
Our classification argument suggests that the system \eqref{E:d21} 
can be regarded as a limiting case of \eqref{E:sysnew1}.
One sees that the behavior of solution is similar in these two cases. Especially,
it is common that the second component has a \emph{logarithmic amplitude correction}.
The difference is as follows:
The behavior of solutions to \eqref{E:sysnew1} involves a logarithmic phase correction term.
Further, the logarithmic amplitude correction depends not only on $W_1$ but also on $W_2$.
This reflects the difference of the mechanism of appearance of logarithmic amplitude correction,
which is, at least formally, easily verified by the analysis of the corresponding limit ODE systems.
\end{remark}

\begin{remark}
	Remark that the above theorems do not follow from the argument of Katayama-Sakoda~\cite{KS} since  
	\eqref{E:sysnew1} and \eqref{E:d21} do not satisfy their assumption. 
\end{remark} 
%In \eqref{E:sysnew1}, the two equations have the same linear part. 
%It is known that if the linear parts have difference masses, then certain resonance phenomenon appears 
%which is known as the mass resonance phenomenon. 
%In \cite{NST, U}, the asymptotic behavior of the solution to several system of cubic nonlinear Schr\"odinger equations
%is considered from the view point of the final state problem. 
%However, we do not follow this issue in this paper. 

%In the proof of Lemma \ref{Lem:Long1} and Theorem \ref{T:main_add}, 
%we focus on analyzing the system of ordinary differential equations:
%\begin{equation}
%	\label{eq:ode_T1.1}
%	\left\{
%	\begin{aligned}
%		i\partial_t w_1 &= 3\lambda_1|w_1|^2w_1, \\
%    i\partial_tw_2 &= \lambda_6(2|w_1|^2w_2+w_1^2\ol{w_2}).
%	\end{aligned}
%	\right.
%\end{equation}
%In Appendix, we give a classification of the general ODE system including 
%the system \eqref{eq:ode_T1.1}. 

The rest of the paper is organized as follows.
In Section 2, we first prove our main result (Theorem \ref{T:main_add}).
Then, we turn to the proofs of Theorems \ref{T:main4} and \ref{T:main51} in Section 3.
Appendix A is devoted to the  classification result of \eqref{eq:NLS}.
A global well-posdeness result for \eqref{eq:NLS} is given as an application in Proposition \ref{P:massconservation}.
Finally, we exhibit an interesting example of system in Appendix B.

%%%%%%%%%%%%%%%%%%%%%
%
%  Proof of Theorem 1.1
%
%%%%%%%%%%%%%%%%%%%%%

\section{Proof of Theorem \ref{T:main_add}.} 

In this section, we prove Theorem \ref{T:main_add}. 
Let us recall that $(u_{1},u_{2})$ satisfies 
\begin{equation}\label{E:sysnewa11}
\left\{
\begin{aligned}
&i\partial_t u_1 + \frac12\partial_x^2 u_1
=  3\la_1 |u_1|^2u_1, &&t\in\R,\ x\in\R,\\
&i\partial_t u_2 + \frac12\partial_x^2 u_2
= \la_6 (2|u_1|^2u_2+u_1^2\overline{u_2}),&&t\in\R,\ x\in\R,\\
&u_1(0,x)=u_{1,0}(x),\qquad u_2(0,x)=u_{2,0}(x),
&&x \in \R,
\end{aligned}
\right.
\end{equation}
where $\la_{1}$ and $\la_{6}$ satisfy $(\la_{6}-\la_{1})(\la_{6}-3\la_{1})>0$. 

To analyze the solution to (\ref{E:sysnewa11}), we introduce several linear operators. 
Let $\{U(t)\}_{t\in\R}$ be a unitary group generated by $i\pt_x^2/2$, i.e., 
\begin{align*}
U(t):={{\mathcal F}}^{-1}e^{-\frac{it\xi^2}{2}}{{\mathcal F}},
\end{align*}
where ${{\mathcal F}}$ and ${{\mathcal F}}^{-1}$ are usual Fourier transform 
and its inverse transform. 

We define multiplication operator $M(t)$ and 
dilation operator $D(t)$ by 
\begin{align*}
(M(t)f)(x)=e^{\frac{ix^2}{2t}}f(x),
\quad (D(t)f)(x)=t^{-\frac12}f\left(\frac{x}{t}\right)e^{-\frac{i\pi}{4}},
\quad t\in\R\backslash\{0\}.
\end{align*}
Then we have well-known Dollard decomposition
for free Schr\"odinger group:
\begin{align}
U(t)=M(t)D(t){{\mathcal F}}M(t).\label{Do}
\end{align}

Let $w_j:={{\mathcal F}}U(-t)u_j$, $j=1,2$. 
Then, applying ${{\mathcal F}}U(-t)$ to (\ref{E:sysnewa11}), we obtain 
\begin{align}
i\pt_tw_1=3\la_1{{\mathcal F}}U(-t)|U(t){{\mathcal F}}^{-1}w_1|^2U(t){{\mathcal F}}^{-1}w_1.
\label{a11}
%i\pt_tw_2&=&\la{{\mathcal F}}U(-t)(2|U(t){{\mathcal F}}^{-1}w_1|^2U(t){{\mathcal F}}^{-1}w_2
%\nonumber\\
%& &\qquad+(U(t){{\mathcal F}}^{-1}w_1)^2\overline{U(t){{\mathcal F}}^{-1}w}_2).
%\label{a2}
\end{align}
By using the Dollard decomposition (\ref{Do}), we easily see that  
\begin{align}
{{\mathcal F}}U(-t)&=U(1/t)D^{-1}(t)M^{-1}(t),\label{1.3}\\
U(t){{\mathcal F}}^{-1}&=M(t)D(t)U(-1/t).\label{1.4}
\end{align}
We summarize several estimates for the operator $U(\pm1/t)$. 
%We first derive the asymptotic formula for $U(t)$ in $L^{\infty}$ as $t\to\infty$. 
\begin{lemma}\label{Lem1}
(i) There exists a positive constant $C$ such that 
for any $0<\alpha<1/4$ and $\varphi\in H_{\xi}^1(\R)$, we have
\begin{align}
\|U(\pm1/t)\varphi-\varphi\|_{L_{\xi}^{\infty}}\lesssim t^{-\alpha}\|\varphi\|_{H_{\xi}^1}.
\label{c1}
\end{align}
(ii) There exists a positive constant $C$ such that 
for any $\varphi\in H_{\xi}^1(\R)$, we have
\begin{align}
\|U(\pm1/t)\varphi\|_{H_{\xi}^1}\lesssim\|\varphi\|_{H_{\xi}^1}.
\label{c3}
\end{align}
\end{lemma}

\noindent
{\bf Proof of Lemma \ref{Lem1}.} The proof easily follows from the explicit representation (\ref{Do})
of the unitary group $U(t)$. $\qed$

\vskip4mm

By (\ref{1.3}) and (\ref{1.4}), eq. (\ref{a11}) can be rewritten as
\begin{align}
i\pt_tw_1
&=
3\la_1 U(1/t)D^{-1}(t)M^{-1}(t)
|M(t)D(t)U(-1/t)w_1|^2M(t)D(t)U(-1/t)w_1\label{a31}\\
&=3\la_1 t^{-1}U(1/t)|U(-1/t)w_1|^2U(-1/t)w_1.\nonumber
\end{align}
In a similar way, we have
\begin{align}
i\pt_tw_2=
\la_6 t^{-1}U(1/t)
\left\{2|U(-1/t)w_1|^2U(-1/t)w_2%\right.\nonumber\\
%& &\qquad\qquad\qquad\left.
+(U(-1/t)w_1)^2\overline{U(-1/t)w_2}\right\}.\label{a41}
\end{align}

We first obtain short time bounds of $w_j$.

\begin{lemma}[Short time bounds]\label{Lem:Short1}
There exists $\varepsilon_0>0$ such 
that for any $u_{j,0}\in H^{1}(\R)\cap H^{0,1}(\R)$ satisfying 
$%\displaystyle{
\varepsilon:=\sum_{j=1}^2(\|u_{j,0}\|_{H^{1}}+\|u_{j,0}\|_{H^{0,1}})\le\varepsilon_0$, %}$, 
there exists a unique solution $u_j\in C([0,1],H^{1}(\R)\cap H^{0,1}(\R))$ 
of (\ref{E:sysnewa11}) satisfying 
\begin{align*}
\sup_{t\in[0,1]}\sum_{j=1}^{2}\|w_j(t)\|_{H_{\xi}^1}\lesssim\varepsilon.
\end{align*}
\end{lemma}

\noindent
{\bf Proof of Lemma \ref{Lem:Short1}.} The proof follows from 
a standard well-posedness theory, see \cite{Caz} for instance. 
Hence we omit the proof. $\qed$ 

\vskip2mm

Next we derive a long time bounds of $w_j$. %which is a crucial 
%part of the proof of Theorem \ref{thm:main}.
We fix $0<\gamma<\delta<1/100$ and introduce 
\begin{align*}
\|(w_1,w_2)\|_{X_T}&:=
\sup_{t\in[1,T]}
\left\{\|w_1(t)\|_{L_{\xi}^{\infty}}+\langle t\rangle^{-\gamma}\|w_1(t)\|_{H_{\xi}^1}\right.\\
& \qquad\quad+
\left.(\log\langle t\rangle)^{-1}\|w_2(t)\|_{L_{\xi}^{\infty}}
+\langle t\rangle^{-\delta}\|w_2(t)\|_{H_{\xi}^1}\right\}.
\end{align*}
%We fix $0<\gamma<1/100$ and introduce 
%\begin{align*}
%\|(w_1,w_2)\|_{X_T}&:=
%\sup_{t\in[1,T]}
%\sum_{j=1}^{2}\left\{\|w_j(t)\|_{L_{\xi}^{\infty}}+\langle t\rangle^{-\gamma}\|w_j(t)\|_{H_{\xi}^1}\right\}.
%\end{align*}

\begin{lemma}[Long time bounds]\label{Lem:Long1}
There exists $\varepsilon_0>0$ such 
that for any $u_{j,0}\in H^{1}(\R)\cap H^{0,1}(\R)$ satisfying 
$%\displaystyle{
\varepsilon:=\sum_{j=1}^2(\|u_{j,0}\|_{H^{1}}+\|u_{j,0}\|_{H^{0,1}})\le\varepsilon_0$, %}$, 
there exists a unique global solution $u_j\in C([0,\infty),H^{1}(\R)\cap H^{0,1}(\R))$ 
of (\ref{E:sysnewa11}) satisfying 
\begin{align}
\|(w_1,w_2)\|_{X_{\infty}}\lesssim\varepsilon.\label{longbound1}
\end{align}
\end{lemma}

\begin{remark}\label{Rem:Long} If $(w_1,w_2)$ satisfy (\ref{longbound1}), then 
we obtain $L^{\infty}$ decay estimates for solution $u_{1}$ 
of (\ref{E:sysnewa11}). Indeed, by Lemma \ref{Lem1}, and (\ref{longbound1}), 
we see,
\begin{align*}
\|u_1(t)\|_{L_x^{\infty}}
&=\|U(t){{\mathcal F}}^{-1}w_1(t)\|_{L_x^{\infty}}\\
&=\|M(t)D(t)U(-1/t)w_1(t)\|_{L_{\xi}^{\infty}}\\
&\lesssim t^{-\frac12}\|U(-1/t)w_1(t)\|_{L_{\xi}^{\infty}}\\
&\lesssim t^{-\frac12}(\|w_1(t)\|_{L_{\xi}^{\infty}}+t^{-\alpha}\|w_1(t)\|_{H_{\xi}^1})\\
&\lesssim \varepsilon(t^{-\frac12}+t^{-\frac12-\alpha+\gamma}),
\end{align*}
for any $t\ge1$, where $0<\alpha<1/4$. Hence choosing $\alpha$ so that 
$0<\gamma<\alpha$, we obtain 
\begin{align*}
\|u_1(t)\|_{L_x^{\infty}}\lesssim\varepsilon t^{-\frac12}. 
\end{align*}
In a similar way, we have (non-optimal) decay estimate for $u_{2}$:
\begin{align*}
\|u_2(t)\|_{L_x^{\infty}}\lesssim\varepsilon t^{-\frac12}\log t
\end{align*}
for any $t\ge1$. We will show later that $u_{2}(t)=O(t^{-1/2})$ by analyzing 
the large time asymptotics of $u_{2}$. 
\end{remark}

\begin{proof}
[{\bf Proof of Lemma \ref{Lem:Long1}.}] 
We first evaluate $H^1$ norm of $w_1$. By (\ref{a31}), we have
\begin{align*}
w_1(t)
=w_1(1)-3i\la_1\int_1^t
\tau^{-1}U(1/\tau)\left[|U(-1/\tau)w_1|^2U(-1/\tau)w_1\right](\tau)d\tau.
\end{align*}
Then, we see
\begin{align*}
\|w_1(t)\|_{H_{\xi}^1}
\lesssim \|w_1(1)\|_{H_{\xi}^1}+\int_1^t
\tau^{-1}\|U(1/\tau)\left[|U(-1/\tau)w_1|^2U(-1/\tau)w_1\right]\|_{H_{\xi}^1}d\tau.
\end{align*}
By Lemma \ref{Lem1},
\begin{align*}
\lefteqn{\|U(1/\tau)\left[|U(-1/\tau)w_1|^2U(-1/\tau)w_1\right]\|_{H_{\xi}^1}}\\
&\lesssim \||U(-1/\tau)w_1|^2U(-1/\tau)w_1\|_{H_{\xi}^1}\\
&\lesssim \|U(-1/\tau)w_1\|_{L_{\xi}^{\infty}}^2\|U(-1/\tau)w_1\|_{H_{\xi}^1}\\
&\lesssim (\|w_1\|_{L_{\xi}^{\infty}}+\tau^{-\alpha}\|w_1\|_{H_{\xi}^1})^2\|w_1\|_{H_{\xi}^1}\\
&\lesssim (1+\tau^{-\alpha+\gamma})^2\tau^{\gamma}\|(w_1,w_2)\|_{X_T}^3,
\end{align*}
where $0<\alpha<1/4$. Choosing $\alpha$ so that $\gamma<\alpha<1/4$ 
and using Lemma \ref{Lem:Short1}, we find
\begin{align}
\langle t\rangle^{-\gamma}\|w_1(t)\|_{H_{\xi}^1}
&\lesssim \varepsilon+\langle t\rangle^{-\gamma}\|(w_1,w_2)\|_{X_T}^3
\int_1^t\tau^{-1+\gamma}d\tau\label{n11}\\
&\lesssim 
\varepsilon+\frac{1}{\gamma}\|(w_1,w_2)\|_{X_T}^3.\nonumber
\end{align}
Next we evaluate $H^1$ norm of $w_2$. By (\ref{a41}), we have
\begin{align*}
\lefteqn{w_2(t)=w_2(1)}\\
&\quad -i\la_6\int_1^t
\tau^{-1}U(1/\tau)\left[2|U(-1/\tau)w_1|^2U(-1/\tau)w_2
+(U(-1/\tau)w_1)^2\overline{U(-1/\tau)w_2}\right](\tau)d\tau.
\end{align*}
Then, we see
\begin{align*}
\lefteqn{\|w_2(t)\|_{H_{\xi}^1}\lesssim\|w_2(1)\|_{H_{\xi}^1}}\\
&\quad +\int_1^t
\tau^{-1}
\left\|U(1/\tau)\left[2|U(-1/\tau)w_1|^2U(-1/\tau)w_2
+(U(-1/\tau)w_1)^2\overline{U(-1/\tau)w_2}\right]\right\|_{H_{\xi}^1}d\tau.
\end{align*}
By Lemma \ref{Lem1},
\begin{align*}
\lefteqn{\left\|U(1/\tau)\left[2|U(-1/\tau)w_1|^2U(-1/\tau)w_2
+(U(-1/\tau)w_1)^2\overline{U(-1/\tau)w_2}\right]\right\|_{H_{\xi}^1}}\\
&\lesssim \||U(-1/\tau)w_1|^2U(-1/\tau)w_2\|_{H_{\xi}^1}\\
&\lesssim \|U(-1/\tau)w_1\|_{L_{\xi}^{\infty}}\|U(-1/\tau)w_1\|_{H_{\xi}^1}
\|U(-1/\tau)w_2\|_{L_{\xi}^{\infty}}\\
&\quad +\|U(-1/\tau)w_1\|_{L_{\xi}^{\infty}}^2\|U(-1/\tau)w_2\|_{H_{\xi}^1}\\
&\lesssim (\|w_1\|_{L_{\xi}^{\infty}}+\tau^{-\alpha}\|w_1\|_{H_{\xi}^1})\|w_1\|_{H_{\xi}^1}
(\|w_2\|_{L_{\xi}^{\infty}}+\tau^{-\alpha}\|w_2\|_{H_{\xi}^1})\\
&\quad +(\|w_1\|_{L_{\xi}^{\infty}}+\tau^{-\alpha}\|w_1\|_{H_{\xi}^1})^2\|w_2\|_{H_{\xi}^1}\\
&\lesssim
(1+\tau^{-\alpha+\gamma})\tau^{\gamma}(\log \tau+\tau^{-\alpha+\delta})\|(w_1,w_2)\|_{X_T}^3\\
&\quad+(1+\tau^{-\alpha+\gamma})^{2}\tau^{\delta}\|(w_1,w_2)\|_{X_T}^3,
\end{align*}
where $0<\alpha<1/4$. Choosing $\alpha$ so that $\delta<\alpha<1/4$ 
and using Lemma \ref{Lem:Short1}, we find
\begin{align}
\langle t\rangle^{-\delta}\|w_2(t)\|_{H_{\xi}^1}
&\lesssim \varepsilon+\langle t\rangle^{-\delta}\|(w_1,w_2)\|_{X_T}^3
\int_1^t\tau^{-1+\delta}d\tau\label{n21}\\
&\lesssim 
\varepsilon+\frac{1}{\delta}\|(w_1,w_2)\|_{X_T}^3.\nonumber
\end{align}

Next we derive $L^{\infty}$ estimates for $w_j$.
From viewpoint of the asymptotic formulas for $U(\pm1/t)$ 
(Lemma \ref{Lem1}), we decompose the nonlinear term as follows:
\begin{align}
i\pt_tw_1&=3\la_1 t^{-1}|w_1|^2w_1+R_1,\label{3.91}\\
i\pt_tw_2&=\la_6 t^{-1}(2|w_1|^2w_2+w_1^2\overline{w_2})+R_2,\label{3.101}
\end{align}
where $R_1$ and $R_2$ are given by 
\begin{align*}
R_1&=3\la_1 t^{-1}\Large[U(1/t)|U(-1/t)w_1|^2U(-1/t)w_1-|w_1|^2w_1\Large],\\%\label{b1}\\
R_2&=\la_6 t^{-1}\Large[U(1/t)\{2|U(-1/t)w_1|^2U(-1/t)w_2
+(U(-1/t)w_1)^2\overline{U(-1/t)w_2}\}\nonumber\\
&\qquad\ \ -(2|w_1|^2w_2+w_1^2\overline{w_2})\Large].%\label{b2}
\end{align*}
Since
\begin{align*}
R_1&=3\la_1 t^{-1}\Large[|U(-1/t)w_1|^2U(-1/t)w_1-|w_1|^2w_1\Large]\\
&\quad +3\la_1 t^{-1}(U(1/t)-1)|U(-1/t)w_1|^2U(-1/t)w_1,
\end{align*}
by Lemma \ref{Lem1}, we have
\begin{align}
\|R_1\|_{L_{\xi}^{\infty}}&\lesssim
t^{-1}(\|U(-1/t)w_1\|_{L_{\xi}^{\infty}}+\|w_1\|_{L_{\xi}^{\infty}})^2\|U(-1/t)w_1-w_1\|_{L_{\xi}^{\infty}}
\label{3.111}\\
&\quad +t^{-1-\alpha}\||U(-1/t)w_1|^2U(-1/t)w_1\|_{H_{\xi}^1}
\nonumber\\
&\lesssim
t^{-1-\alpha}(\|w_1\|_{L_{\xi}^{\infty}}+t^{-\alpha}\|w_1\|_{H_{\xi}^1})^2\|w_1\|_{H_{\xi}^1}
\nonumber\\
&\quad +t^{-1-\alpha}\|U(-1/t)w_1\|_{L_{\xi}^{\infty}}^2\|U(-1/t)w_1\|_{H_{\xi}^1}
\nonumber\\
&\lesssim
t^{-1-\alpha}(\|w_1\|_{L_{\xi}^{\infty}}+t^{-\alpha}\|w_1\|_{H_{\xi}^1})^2\|w_1\|_{H_{\xi}^1}
\nonumber\\
&\lesssim
t^{-1-\alpha}(1+t^{-\alpha+\gamma})^2t^{\gamma}\|(w_1,w_2)\|_{X_T}^3\nonumber\\
&\lesssim
t^{-1-\alpha+\gamma}\|(w_1,w_2)\|_{X_T}^3.\nonumber
\end{align}
Since
\begin{align*}
R_2&=3\la_1 t^{-1}\Large[\{2|U(-1/t)w_1|^2U(-1/t)w_2
+(U(-1/t)w_1)^2\overline{U(-1/t)w_2}\}\\
&\qquad\qquad\qquad\qquad\qquad\qquad\qquad\qquad\qquad
-(2|w_1|^2w_2+w_1^2\overline{w_2})\Large]\\
&\quad +3\la_1 t^{-1}(U(1/t)-1)\{2|U(-1/t)w_1|^2U(-1/t)w_2
+(U(-1/t)w_1)^2\overline{U(-1/t)w_2}\}
\end{align*}
by Lemma \ref{Lem1}, we have
\begin{align}
\lefteqn{\|R_2\|_{L_{\xi}^{\infty}}}\label{3.121}\\
&\lesssim
t^{-1}(\|U(-1/t)w_1\|_{L_{\xi}^{\infty}}+\|w_1\|_{L_{\xi}^{\infty}})\|U(-1/t)w_1-w_1\|_{L_{\xi}^{\infty}}
\|U(-1/t)w_2\|_{L_{\xi}^{\infty}}
\nonumber\\
&\quad +t^{-1}\|w_1\|_{L_{\xi}^{\infty}}^2\|U(-1/t)w_2-w_2\|_{L_{\xi}^{\infty}}
\nonumber\\
&\quad +t^{-1-\alpha}\||U(-1/t)w_1|^2U(-1/t)w_2\|_{H_{\xi}^1}
\nonumber\\
&\lesssim
t^{-1-\alpha}(\|w_1\|_{L_{\xi}^{\infty}}+t^{-\alpha}\|w_1\|_{H_{\xi}^1})\|w_1\|_{H_{\xi}^1}
(\|w_2\|_{L_{\xi}^{\infty}}+t^{-\alpha}\|w_2\|_{H_{\xi}^1})
\nonumber\\
&\quad +t^{-1-\alpha}\|w_1\|_{L_{\xi}^{\infty}}^2\|w_2\|_{H_{\xi}^1}
\nonumber\\
&\quad +t^{-1-\alpha}\|U(-1/t)w_1\|_{L_{\xi}^{\infty}}^2\|U(-1/t)w_1\|_{H_{\xi}^1}
\nonumber\\
&\lesssim
t^{-1-\alpha}(\|w_1\|_{L_{\xi}^{\infty}}+t^{-\alpha}\|w_1\|_{H_{\xi}^1})\|w_1\|_{H_{\xi}^1}
(\|w_2\|_{L_{\xi}^{\infty}}+t^{-\alpha}\|w_2\|_{H_{\xi}^1})
\nonumber\\
&\quad +t^{-1-\alpha}(\|w_1\|_{L_{\xi}^{\infty}}+t^{-\alpha}\|w_1\|_{H_{\xi}^1})^2\|w_2\|_{H_{\xi}^1}
\nonumber\\
&\lesssim
t^{-1-\alpha+\gamma}(1+t^{-\alpha+\gamma})(\log t+t^{-\alpha+\delta})\|(w_1,w_2)\|_{X_T}^3
\nonumber\\
&\quad+t^{-1-\alpha+\delta}(1+t^{-\alpha+\gamma})^{2}\|(w_1,w_2)\|_{X_T}^3
\nonumber\\
&\lesssim t^{-1-\alpha+\delta}\|(w_1,w_2)\|_{X_T}^3.\nonumber
\end{align}
By (\ref{3.91}), 
\begin{align*}
\pt_t|w_1|^2=2\Im(R_1\overline{w}_1)\lesssim\|R_1(t)\|_{L_{\xi}^{\infty}}|w_1|.
\end{align*}
Hence (\ref{3.111}) yields
\begin{align}
\pt_t|w_1|\lesssim\|R_1(t)\|_{L_{\xi}^{\infty}}
\lesssim t^{-1-\alpha+\gamma}\|(w_1,w_2)\|_{X_T}^3.
\label{3.1311}
\end{align}
Therefore
\begin{align}
|w_1(t,\xi)|\lesssim\varepsilon+\|(w_1,w_2)\|_{X_T}^3.
\label{3.13110}
\end{align}
By (\ref{3.91}) and (\ref{3.101}),
\begin{align}
i\pt_{t}w_{1}\overline{w_{2}}
&=3\la_{1}t^{-1}|w_{1}|^{2}w_{1}\overline{w_{2}}+R_{1}\overline{w_{2}},\label{d11}\\
i\overline{w_{1}}\pt_{t}w_{2}
&=\la_{6}t^{-1}(2|w_{1}|^{2}\overline{w_{1}}w_{2}+|w_{1}|^{2}w_{1}\overline{w_{2}})
+R_{2}\overline{w_{1}}.\label{d21}
\end{align}
From (\ref{d11}) and (\ref{d21}), we find 
\begin{align}
\lefteqn{\pt_{t}\{w_{1}\overline{w_{2}}+\tilde{\la}_{\pm}\overline{w_{1}}w_{2}\}
e^{\pm\la_c|w_{1}|^{2}\log t}}\\
&=\left(-iR_{1}\overline{w_{2}}+i\overline{R_{2}}w_{1}\pm 2i\la_{c}\log tw_{1}\overline{w_{2}}
\Im(R_{1}\overline{w_{1}})\right)e^{\pm\la_c|w_{1}|^{2}\log t}\nonumber\\
& \qquad +\tilde{\la}_{\pm}\left(i\overline{R_{1}}w_{2}-iR_{2}\overline{w_{1}}
\pm 2i\la_{c}\log t\overline{w_{1}}w_{2}
\Im(R_{1}\overline{w_{1}})\right)e^{\pm\la_c|w_{1}|^{2}\log t}\nonumber\\
&\lesssim|R_{1}|(1+|w_{1}|^{2}\log t)|w_{2}|+|R_{2}||w_{1}|,\nonumber
\end{align}
where 
\begin{eqnarray}
\tilde{\la}_{\pm}=\frac{-3\la_{1}+2\la_{6}\pm\la_c}{\la_{6}},
\quad\la_c=\sqrt{3(\la_{6}-\la_{1})(\la_{6}-3\la_{1})}.\label{lam}
\end{eqnarray}
Hence
\begin{eqnarray*}
|w_1\overline{w_2}|\lesssim\varepsilon^{2}+t^{-\alpha+\delta}
(\|(w_1,w_2)\|_{X_T}^4+\|(w_1,w_2)\|_{X_T}^6). 
\end{eqnarray*}
Therefore (\ref{3.101}) and (\ref{3.121}) imply 
\begin{align*}
\pt_t|w_2|
&\lesssim t^{-1}|w_1\overline{w_2}||w_1|+|R_2|\\
&\lesssim 
\varepsilon^2t^{-1}\|(w_1,w_2)\|_{X_T}
+t^{-1}\|(w_1,w_2)\|_{X_T}^5+t^{-1}\|(w_1,w_2)\|_{X_T}^7\\
&\quad+t^{-1-\alpha+\delta}\|(w_1,w_2)\|_{X_T}^3.
\end{align*}
Integrating this in $t$, we have
\begin{align}
|w_2(t,\xi)|
&\lesssim \varepsilon+\varepsilon^2\log t\|(w_1,w_2)\|_{X_T}+\log t\|(w_1,w_2)\|_{X_T}^5\label{3.151}\\
&\quad +\log t\|(w_1,w_2)\|_{X_T}^7+\|(w_1,w_2)\|_{X_T}^3.\nonumber
\end{align}
Collecting (\ref{n11}), (\ref{n21}), (\ref{3.13110}) and (\ref{3.151}), we
obtain
\begin{align*}
\|(w_1,w_2)\|_{X_T}
\lesssim \varepsilon+\varepsilon^2\|(w_1,w_2)\|_{X_T}+\|(w_1,w_2)\|_{X_T}^3
+\|(w_1,w_2)\|_{X_T}^7.
\end{align*}
The the standard continuity argument yields that 
if $\varepsilon=\varepsilon(\gamma,\delta)$ is sufficiently small, 
then we have
\begin{align*}
\|(w_1,w_2)\|_{X_T}\lesssim\varepsilon
\end{align*}
for any $T\ge1$. This completes the proof of Lemma \ref{Lem:Long1}. 
\end{proof}
%%%%%%%%%%%%%%%%%%%%%%%%%%%%%%%%%%%%%%%%%

\begin{remark}
One sees from \eqref{3.91} and \eqref{3.101} that $(w_1,w_2)$ solves \eqref{eq:limitODE} with an error (by regarding $\log t $ as a ``time'' variable).
\end{remark}

\vskip2mm
\begin{proof}[{\bf Proof of Theorem \ref{T:main_add}.} ]
The global existence of solution to 
(\ref{E:sysnewa11}) and decay estimate for $u_{1}$ 
follow from Lemma \ref{Lem:Long1} and Remark \ref{Rem:Long}. 
We now derive the asymptotic formulas (\ref{iasym1}) and (\ref{iasym2}) 
for solution $(u_1,u_2)$ to (\ref{E:sysnewa11}) as $t\to\infty$. 
As byproduct of the asymptotic formula (\ref{iasym2}), we obtain 
decay estimate for $u_{2}$.

By (\ref{3.1311}), we see that there exists a real-valued positive function 
$\widetilde{W}_1\in L^{\infty}$ such that 
\begin{align}
\||w_1|-\widetilde{W}_1\|_{L_{\xi}^{\infty}}\lesssim\varepsilon^3 t^{-\alpha+\gamma},
\label{e11}
\end{align}
where $\gamma<\alpha<1/4$. Substituting this into (\ref{3.91}), we have
\begin{align}
i\pt_tw_1=3\la_1 t^{-1}\widetilde{W}_1^2w_1+R_3,\label{e21}
\end{align}
where $R_3$ is given by
\begin{align*}
R_3=3\la_1 t^{-1}(|w_1|^2-\widetilde{W}_1^2)w_1+R_1.
\end{align*}
Hence by (\ref{3.111}) and (\ref{e11}), 
\begin{align*}
\|R_3(t)\|_{L_{\xi}^{\infty}}\lesssim\varepsilon^3 t^{-1-\alpha+\gamma}.
\end{align*}
Multiplying (\ref{e21}) by $\exp(3i\la_1\tilde{W}_1^2\log t)$, we have
\begin{align*}
i\pt_t(w_1e^{3i\la_1\widetilde{W}_1^2\log t})=R_3e^{3i\la_1\widetilde{W}_1^2\log t}.
\end{align*}
This implies that there exists $W_1\in L^{\infty}$ such that 
\begin{align*}
\|w_1e^{3i\la_1\widetilde{W}_1^2\log t}-W_1\|_{L_{\xi}^{\infty}}\lesssim\varepsilon^3 t^{-\alpha+\gamma}.
%\label{e61}
\end{align*}
We easily see that $|W_1|=\widetilde{W}_1$ and obtain %(\ref{iasym1}).
\begin{align}
\|w_1-W_1e^{-3i\la_1|W_1|^2\log t}\|_{L_{\xi}^{\infty}}\lesssim\varepsilon^3 t^{-\alpha+\gamma}.
\label{e61}
\end{align}

To derive the asymptotic behavior of $w_{2}$, we introduce a new unknown function 
\begin{equation}\notag
	\beta(t,\xi) = w_2(t,\xi)e^{3i\lambda_1|W_1(\xi)|^2\log t}.
\end{equation}
Then by (\ref{3.101}), we see
\begin{equation}\label{eq:beta}
	\begin{aligned}
		i\partial_t \beta &= 
		\left\{\frac{\lambda_6}{t}(2|w_1|^2w_2+w_1^2\overline{w_2})
		+R_{2}\right\}e^{3i\lambda_1|W_1|^2\log t}
		- \frac{3\lambda_1}{t}|W_1|^2w_2e^{3i\lambda_1|W_1|^2\log t}\\
		&= \frac{1}{t}\left( (-3\lambda_1+2\lambda_6)|W_1|^2\beta
		+\lambda_6W_1^2\overline{\beta} \right)  + R_4,
	\end{aligned}
\end{equation}
where 
\begin{eqnarray*}\notag
	R_4 &=& \frac{2\la_{6}}{t}(|w_1|^2-|W_1|^2)\beta 
	+ \frac{\lambda_6}{t}\left\{\left( w_{1}e^{3i\lambda_1|W_1|^2\log t} \right)^2 - W_1^2\right\}\overline{\beta}\\
	& &+R_2e^{3i\lambda_1|W_1(\xi)|^2\log t}.  
\end{eqnarray*}
%Taking the imaginary part of \eqref{eq:beta}, we also have 
%\begin{equation}\label{eq:betabar}
%	\begin{aligned}
%		i\partial_t \overline{\beta}
%			= \frac{1}{t}\left\{(3\lambda_1-2\lambda_6)|W_1|^2\overline{\beta}
%			-\lambda_6\overline{W}_1^2\beta \right\} - \overline{R_4}.
%	\end{aligned}
%\end{equation}
It holds that 
\begin{equation}\notag
	\begin{aligned}
		\|R_4\|_{L^\infty} 
		&\lesssim t^{-1}(\|w_1\|_{L^\infty}+\|W_1\|_{L^\infty})
		\left\|w_1-W_1e^{-3i\lambda_1|W_1(\xi)|^2\log t}\right\|_{L^\infty}\|w_2\|_{L^\infty}\\
		&\quad+\|R_2\|_{L^\infty}\\
		&\lesssim \varepsilon^3 t^{-1-\alpha+\delta}.
	\end{aligned}
\end{equation}
From \eqref{eq:beta}, % and \eqref{eq:betabar}, 
we have
\begin{equation}\label{eq:mbeta}
	i\partial_t 
	\begin{pmatrix}
		\beta \\[1mm]
		\overline{\beta}
	\end{pmatrix}
	= \frac{1}{t}\begin{pmatrix}
		(-3\lambda_1+2\lambda_6)|W_1|^2 & \lambda_6 W_1^2 \\[1mm]
		-\lambda_6\overline{W_1}^2 & (3\lambda_1-2\lambda_6)|W_1|^2
	\end{pmatrix}
	\begin{pmatrix}
		\beta \\[1mm]
		\overline{\beta}
	\end{pmatrix}
	+ 
	\begin{pmatrix}
		 R_4 \\[1mm]
		 - \overline{R_4}
	\end{pmatrix}.
\end{equation}
Let $N:=\{\xi\in\R\ ;\ W_{1}(\xi)=0\}$. Since $i\pt_{t}\beta=R_{4}$ on $N$, 
we easily see that $\beta=O(t^{-\alpha+\gamma})$. Therefore we concentrate 
our attention to the case $\xi\in N^{c}$. If $\xi\in N^{c}$, then we see 
that the matrix 
\begin{equation}\notag
	A %= 
%	\begin{pmatrix}
%		X & Y \\[1mm]
%		-\overline{Y} & -X
%	\end{pmatrix}
	=
	\begin{pmatrix}
		(-3\lambda_1+2\lambda_6)|W_1|^2 & \lambda_6 W_1^2 \\[1mm]
		-\lambda_6\overline{W_1}^2 & (3\lambda_1-2\lambda_6)|W_1|^2
	\end{pmatrix}
\end{equation}
can be diagonalized by the matrix 
\begin{equation}\notag
	P= 
	\begin{pmatrix}
		\la_{6}\frac{W_{1}^{2}}{|W_{1}|^{2}} & 3\la_{1}-2\la_{6}+\la_c  \\[1mm]
		3\lambda_{1}-2\la_{6}+\la_c & \la_{6}\frac{\overline{W}_{1}^{2}}{|W_{1}|^{2}}
	\end{pmatrix},
%	= 
%	\begin{pmatrix}
%		\lambda_6W_1^2 & \lambda - (2\lambda_6-3\lambda_1)|W_1|^2 \\[1mm]
%		\lambda - (2\lambda_6-3\lambda_1)|W_1|^2 & \lambda_6\overline{W}_1^2
%	\end{pmatrix}, 
\end{equation}
where $ \lambda $ is given by (\ref{lam}). %= \sqrt{3(\lambda_6-\lambda_1)(\lambda_6-3\lambda_1)}$. %|W_1|^2 $. 
%Namely, it holds that 
%\begin{equation}\notag
%	P^{-1}AP = 
%	\begin{pmatrix}
%		\lambda|W_{1}|^{2} & 0 \\[1mm]
%    0 & -\lambda|W_{1}|^{2}
%	\end{pmatrix}.
%\end{equation}
Diagonalizing the equation \eqref{eq:mbeta} by the matrix $ P $, we have
\begin{equation}\notag
	i\partial_t
	\begin{pmatrix}
		\gamma \\[1mm]
		\overline{\gamma}
	\end{pmatrix}
	= 
	\frac{1}{t}
	\begin{pmatrix}
		\lambda_c|W_{1}|^{2} & 0 \\[1mm]
		0 & -\lambda_c|W_{1}|^{2}
	\end{pmatrix}
	\begin{pmatrix}
		\gamma \\[1mm]
		\overline{\gamma}
	\end{pmatrix}
	+
	P^{-1}
	\begin{pmatrix}
		R_4 \\[1mm]
		- \overline{R_4}
	\end{pmatrix},
\end{equation}
where $\gamma$ is given by 
\begin{equation}\notag
	\begin{pmatrix}
		\gamma \\[1mm]
		\overline{\gamma}
	\end{pmatrix}
	= P^{-1}
	\begin{pmatrix}
		\beta \\[1mm]
		\overline{\beta}
	\end{pmatrix}.
\end{equation}
For the first component, we have 
\begin{equation}\notag
	i\partial_t\gamma = 
	\frac{1}{t}\lambda_c|W_{1}|^{2}\gamma 
	-\frac{\lambda_c\overline{W}_{1}^{2}R_4
	 +\tilde{\la}_{-}|W_{1}|^{2}\overline{R_4}}{2\la_c\tilde{\la}|W_{1}|^{2}},
\end{equation}
where $\tilde{\la}_{-}$ is given by (\ref{lam}). This implies that 
\begin{equation}\notag
	i\partial_t\left( 
e^{i\lambda_c|W_{1}|^{2} \log t}\gamma_1
	 \right) = -e^{i\lambda_c|W_{1}|^{2}  \log t} 
	 \frac{\lambda_c\overline{W}_{1}^{2}R_4
	 +\tilde{\la}_{-}|W_{1}|^{2}\overline{R_4}}{2\la_c\tilde{\la}|W_{1}|^{2}}.
\end{equation}
By using the estimates for $ R_4 $, we see that 
there exists a function $ W_2 \in L^\infty$ such that 
\begin{equation}\notag
	\left\| e^{i\lambda_c|W_{1}|^{2} \log t}\gamma_1 - W_2\right\|_{L^\infty} \lesssim \varepsilon t^{-\alpha+\gamma}.
\end{equation} 
Since 
\begin{equation}\notag
	\begin{pmatrix}
		\beta \\[1mm]
		\overline{\beta}
	\end{pmatrix}
	= 
	P\begin{pmatrix}
		\gamma \\[1mm]
		\overline{\gamma}
	\end{pmatrix}
	= \begin{pmatrix}
		\la_{6}\frac{W_{1}^{2}}{|W_{1}|^{2}}\gamma+(3\la_{1}-2\la_{6}+\la_c)\overline{\gamma}\\[1mm]
    (3\la_{1}-2\la_{6}+\la_c)\gamma+\la_{6}\frac{\overline{W}_{1}^{2}}{|W_{1}|^{2}}\overline{\gamma}
	\end{pmatrix},
\end{equation}
it follows that
\begin{equation}\notag
\begin{aligned}
		\beta(t) %& \la_{6}\frac{W_{1}^{2}}{|W_{1}|^{2}}\gamma+(3\la_{1}-2\la_{6}+\la)\overline{\gamma}\\
		=& \la_{6}\frac{W_{1}^{2}}{|W_{1}|^{2}}W_2e^{-i\lambda_c|W_{1}|^{2}\log t}
		+(3\la_{1}-2\la_{6}+\la_c)\overline{W}_{2}e^{i\lambda_c|W_{1}|^{2}\log t} + O(t^{-\alpha+\gamma}).
\end{aligned}
\end{equation}
Hence we see that
\begin{equation}\notag
\begin{aligned}
	w_2(t, \xi) 
		=& \la_{6}\frac{W_{1}^{2}}{|W_{1}|^{2}}W_2e^{i(-3\la_{1}-\lambda_c)|W_{1}|^{2}\log t}\\
		&+(3\la_{1}-2\la_{6}+\la_c)\overline{W}_{2}e^{i(-3\la_{1}+\lambda_c)\log t} + O(t^{-\alpha+\gamma}).
	\end{aligned}
\end{equation}
By Lemma \ref{Lem1} and the asymptotic formula (\ref{e61}) for $w_1$,  
we have
\begin{align}
u_1(t)
&=U(t){{\mathcal F}}^{-1}w_1\label{qq1}\\
&=M(t)D(t)U(-1/t)w_1\nonumber\\
&=M(t)D(t)w_1+O(t^{-\frac12-\alpha})\nonumber\\
&=t^{-\frac12}W_1\left(\frac{x}{t}\right)
e^{\frac{ix^2}{2t}-3i\la_1\left|W_1\left(\frac{x}{t}\right)\right|^2\log t-i\frac{\pi}{4}}
+O(t^{-\frac12-\alpha+\gamma}),\nonumber
\end{align}
in $L_x^{\infty}$ as $t\to\infty$. Hence we have (\ref{iasym1}). In a similar way, 
we obtain (\ref{iasym2}). 
%\begin{equation}\notag
%	\begin{aligned}
%		u_2(t,x) &= \lambda_6t^{-\frac12}%\left(
%		\frac{W_{1}^{2}}{|W_{1}|^{2}}W_2%\right)
%		\left( \frac{x}{t} \right)
%		e^{\frac{ix^2}{2t}+i(3\lambda_1-\la)|W_1(\frac{x}{t})|^2\log t-i\frac{\pi}{4}} \\
%		&\quad +(3\la_{1}-2\la_{6}+\la)\overline{W}_{2}\left( \frac{x}{t} \right)
%		e^{\frac{ix^2}{2t} + i(3\lambda_1+\la)|W_1(\frac{x}{t})|^2\log t-i\frac{\pi}{4}}\\
%		&\quad+ O(t^{-\frac12-\alpha+\gamma}).
%	\end{aligned}
%\end{equation}
This completes the proof. 
%Therefore, we see that 
%\begin{equation}\left\{
%	\begin{aligned}
%		\gamma_1(t) &=
%		\frac{(\lambda+X)\beta(1)+Y\overline{\beta}(1)}{2\lambda Y}e^{-i\lambda \log t} 
%		+ \frac{(\lambda+X)(R_2+R_4)+Y(\overline{R_2}+\overline{R_4})}{2\lambda Y},\\
%		\gamma_2(t) &=
%		\frac{(\lambda-X)\beta(1)-Y\overline{\beta}(1)}{2\lambda Y}e^{-i\lambda \log t}
%		+ \frac{(\lambda-X)(R_2+R_4)-Y(\overline{R_2}+\overline{R_4})}{2\lambda Y}.
%	\end{aligned}
%\right.
%\end{equation}
\end{proof}

%%%%%%%%%%%%%%%%%%%%%
%
%  Proof of Theorem 1.3
%
%%%%%%%%%%%%%%%%%%%%%

\section{Proofs of Theorems \ref{T:main4} and \ref{T:main51}.} 

In this section, we prove Theorems \ref{T:main4} and \ref{T:main51}. 
We give the proof of Theorem \ref{T:main4} only since 
the proof of Theorem \ref{T:main51} is similar and simpler. 

We first consider the case $\la_6=3\la_1$, i.e., 
\begin{equation}\label{E:sysnewa1}
\left\{
\begin{aligned}
&i\partial_t u_1 + \frac12\partial_x^2 u_1
=  3\la_1 |u_1|^2u_1, &&t\in\R,\ x\in\R,\\
&i\partial_t u_2 + \frac12\partial_x^2 u_2
= 3\la_1 (2|u_1|^2u_2+u_1^2\overline{u_2}),&&t\in\R,\ x\in\R,\\
&u_1(0,x)=u_{1,0}(x),\qquad u_2(0,x)=u_{2,0}(x),
&&x \in \R,
\end{aligned}
\right.
\end{equation}
Let $w_j:={{\mathcal F}}U(-t)u_j$, $j=1,2$. 
As in the proof of Theorem \ref{T:main_add}, by applying ${{\mathcal F}}U(-t)$ to (\ref{E:sysnewa1}), we obtain 
\begin{align}
i\pt_tw_1&=3\la_1 t^{-1}U(1/t)|U(-1/t)w_1|^2U(-1/t)w_1,\label{a3}\\
i\pt_tw_2&=
3\la_1 t^{-1}U(1/t)
\left\{2|U(-1/t)w_1|^2U(-1/t)w_2%\right.\nonumber\\
%& &\qquad\qquad\qquad\left.
+(U(-1/t)w_1)^2\overline{U(-1/t)w_2}\right\}.\label{a4}
\end{align}
We first obtain short time bounds of $w_j$.

\begin{lemma}[Short time bounds]\label{Lem:Short}
There exists $\varepsilon_0>0$ such 
that for any $u_{j,0}\in H^{1}(\R)\cap H^{0,1}(\R)$ satisfying 
$%\displaystyle{
\varepsilon:=\sum_{j=1}^2(\|u_{j,0}\|_{H^{1}}+\|u_{j,0}\|_{H^{0,1}})\le\varepsilon_0$, %}$, 
there exists a unique solution $u_j\in C([0,1],H^{1}(\R)\cap H^{0,1}(\R))$ 
of (\ref{E:sysnewa1}) satisfying 
\begin{align*}
\sup_{t\in[0,1]}\sum_{j=1}^{2}\|w_j(t)\|_{H_{\xi}^1}\lesssim\varepsilon.
\end{align*}
\end{lemma}

\noindent
{\bf Proof of Lemma \ref{Lem:Short}.} The proof follows from 
a standard well-posedness theory, see \cite{Caz} for instance. 
Hence we omit the proof. $\qed$ 

\vskip2mm

Next we derive a long time bounds of $w_j$. %which is a crucial 
%part of the proof of Theorem \ref{thm:main}.
We fix $0<\gamma<\delta<1/100$ and introduce 
\begin{align*}
\|(w_1,w_2)\|_{X_T}&:=
\sup_{t\in[1,T]}
\left\{\|w_1(t)\|_{L_{\xi}^{\infty}}+\langle t\rangle^{-\gamma}\|w_1(t)\|_{H_{\xi}^1}\right.\\
& \qquad\quad+
\left.(\log\langle t\rangle)^{-1}\|w_2(t)\|_{L_{\xi}^{\infty}}
+\langle t\rangle^{-\delta}\|w_2(t)\|_{H_{\xi}^1}\right\}.
\end{align*}

\begin{lemma}[Long time bounds]\label{Lem:Long}
There exists $\varepsilon_0>0$ such 
that for any $u_{j,0}\in H^{1}(\R)\cap H^{0,1}(\R)$ satisfying 
$%\displaystyle{
\varepsilon:=\sum_{j=1}^2(\|u_{j,0}\|_{H^{1}}+\|u_{j,0}\|_{H^{0,1}})\le\varepsilon_0$, %}$, 
there exists a unique global solution $u_j\in C([0,\infty),H^{1}(\R)\cap H^{0,1}(\R))$ 
of (\ref{E:sysnewa1}) satisfying 
\begin{align}
\|(w_1,w_2)\|_{X_{\infty}}\lesssim\varepsilon.\label{longbound}
\end{align}
\end{lemma}

\begin{remark}\label{Rem:Long11} As in Remark \ref{Rem:Long}, we 
see that if $(w_1,w_2)$ satisfy (\ref{longbound}), then 
we obtain 
\begin{align*}
\|u_1(t)\|_{L_x^{\infty}}\lesssim\varepsilon t^{-\frac12},
\quad
\|u_2(t)\|_{L_x^{\infty}}\lesssim\varepsilon t^{-\frac12}\log t
\end{align*}
for any $t\ge1$.
\end{remark}

\noindent
{\bf Proof of Lemma \ref{Lem:Long}.} 
As in the proof of Lemma \ref{Lem:Long1}, we have
%\begin{align*}
%w_1(t)
%=w_1(1)-3i\la_1\int_1^t
%\tau^{-1}U(1/\tau)\left[|U(-1/\tau)w_1|^2U(-1/\tau)w_1\right](\tau)d\tau.
%\end{align*}
%Then, we see
%\begin{align*}
%\lefteqn{\|w_1(t)\|_{H_{\xi}^1}}\\
%&\lesssim \|w_1(1)\|_{H_{\xi}^1}+\int_1^t
%\tau^{-1}\|U(1/\tau)\left[|U(-1/\tau)w_1|^2U(-1/\tau)w_1\right]\|_{H_{\xi}^1}d\tau.
%\end{align*}
%By Lemma \ref{Lem1},
%\begin{align*}
%\lefteqn{\|U(1/\tau)\left[|U(-1/\tau)w_1|^2U(-1/\tau)w_1\right]\|_{H_{\xi}^1}}\\
%&\lesssim \||U(-1/\tau)w_1|^2U(-1/\tau)w_1\|_{H_{\xi}^1}\\
%&\lesssim \|U(-1/\tau)w_1\|_{L_{\xi}^{\infty}}^2\|U(-1/\tau)w_1\|_{H_{\xi}^1}\\
%&\lesssim (\|w_1\|_{L_{\xi}^{\infty}}+\tau^{-\alpha}\|w_1\|_{H_{\xi}^1})^2\|w_1\|_{H_{\xi}^1}\\
%&\lesssim (1+\tau^{-\alpha+\gamma})^2\tau^{\gamma}\|(w_1,w_2)\|_{X_T}^3,
%\end{align*}
%where $0<\alpha<1/4$. Choosing $\alpha$ so that $\gamma<\alpha<1/4$ 
%and using Lemma \ref{Lem:Short}, we find
\begin{align}
\langle t\rangle^{-\gamma}\|w_1(t)\|_{H_{\xi}^1}
%&\lesssim \varepsilon+\langle t\rangle^{-\gamma}\|(w_1,w_2)\|_{X_T}^3
%\int_1^t\tau^{-1+\gamma}d\tau\label{n1}\\
&\lesssim 
\varepsilon+\frac{1}{\gamma}\|(w_1,w_2)\|_{X_T}^3,\label{n1}\\
%\end{align}
%Next we evaluate $H^1$ norm of $w_2$. By (\ref{a4}), we have
%\begin{align*}
%\lefteqn{w_2(t)=w_2(1)}\\
%\quad -3i\la_1\int_1^t
%\tau^{-1}U(1/\tau)\left[2|U(-1/\tau)w_1|^2U(-1/\tau)w_2
%+(U(-1/\tau)w_1)^2\overline{U(-1/\tau)w_2}\right](\tau)d\tau.
%\end{align*}
%Then, we see
%\begin{align*}
%\lefteqn{\|w_2(t)\|_{H_{\xi}^1}\lesssim\|w_2(1)\|_{H_{\xi}^1}}\\
%&\quad +\int_1^t
%\tau^{-1}
%\left\|U(1/\tau)\left[2|U(-1/\tau)w_1|^2U(-1/\tau)w_2
%+(U(-1/\tau)w_1)^2\overline{U(-1/\tau)w_2}\right]\right\|_{H_{\xi}^1}d\tau.
%\end{align*}
%By Lemma \ref{Lem1},
%\begin{align*}
%\lefteqn{\left\|U(1/\tau)\left[2|U(-1/\tau)w_1|^2U(-1/\tau)w_2
%+(U(-1/\tau)w_1)^2\overline{U(-1/\tau)w_2}\right]\right\|_{H_{\xi}^1}}\\
%&\lesssim \||U(-1/\tau)w_1|^2U(-1/\tau)w_2\|_{H_{\xi}^1}\\
%&\lesssim \|U(-1/\tau)w_1\|_{L_{\xi}^{\infty}}\|U(-1/\tau)w_1\|_{H_{\xi}^1}
%\|U(-1/\tau)w_2\|_{L_{\xi}^{\infty}}\\
%&\quad +\|U(-1/\tau)w_1\|_{L_{\xi}^{\infty}}^2\|U(-1/\tau)w_2\|_{H_{\xi}^1}\\
%&\lesssim (\|w_1\|_{L_{\xi}^{\infty}}+\tau^{-\alpha}\|w_1\|_{H_{\xi}^1})\|w_1\|_{H_{\xi}^1}
%(\|w_2\|_{L_{\xi}^{\infty}}+\tau^{-\alpha}\|w_2\|_{H_{\xi}^1})\\
%&\quad +(\|w_1\|_{L_{\xi}^{\infty}}+\tau^{-\alpha}\|w_1\|_{H_{\xi}^1})^2\|w_2\|_{H_{\xi}^1}\\
%&\lesssim
%\left\{(1+\tau^{-\alpha+\gamma})\tau^{\gamma}(\log\tau+\tau^{-\alpha+\delta})
%+(1+\tau^{-\alpha+\gamma})^2\tau^{\delta}\right\}\|(w_1,w_2)\|_{X_T}^3,
%\end{align*}
%where $0<\alpha<1/4$. Choosing $\alpha$ so that $\delta<\alpha<1/4$ 
%and using Lemma \ref{Lem:Short}, we find
%\begin{align}
\langle t\rangle^{-\delta}\|w_2(t)\|_{H_{\xi}^1}
%&\lesssim \varepsilon+\langle t\rangle^{-\delta}\|(w_1,w_2)\|_{X_T}^3
%\int_1^t\tau^{-1+\delta}d\tau\label{n2}\\
&\lesssim 
\varepsilon+\frac{1}{\delta}\|(w_1,w_2)\|_{X_T}^3.\label{n2}
\end{align}
Next we derive $L^{\infty}$ estimates for $w_j$.
Form viewpoint of the asymptotic formulas for $U(\pm1/t)$ 
(Lemma \ref{Lem1}), we decompose the nonlinear term as follows:
\begin{align}
i\pt_tw_1&=3\la_1 t^{-1}|w_1|^2w_1+R_1,\label{3.9}\\
i\pt_tw_2&=3\la_1 t^{-1}(2|w_1|^2w_2+w_1^2\overline{w_2})+R_2,\label{3.10}
\end{align}
where $R_1$ and $R_2$ are given by 
\begin{align*}
R_1&=3\la_1 t^{-1}\Large[U(1/t)|U(-1/t)w_1|^2U(-1/t)w_1-|w_1|^2w_1\Large],\\%\label{b1}\\
R_2&=3\la_1 t^{-1}\Large[U(1/t)\{2|U(-1/t)w_1|^2U(-1/t)w_2
+(U(-1/t)w_1)^2\overline{U(-1/t)w_2}\}\nonumber\\
&\qquad\ \ -(2|w_1|^2w_2+w_1^2\overline{w_2})\Large].%\label{b2}
\end{align*}
%Since
%\begin{align*}
%R_1&=3\la_1 t^{-1}\Large[|U(-1/t)w_1|^2U(-1/t)w_1-|w_1|^2w_1\Large]\\
%&\quad +3\la_1 t^{-1}(U(1/t)-1)|U(-1/t)w_1|^2U(-1/t)w_1,
%\end{align*}
As in the proof of Theorem \ref{T:main_add}, Lemma \ref{Lem1} yields
\begin{align}
\|R_1\|_{L_{\xi}^{\infty}}%&\lesssim
%t^{-1}(\|U(-1/t)w_1\|_{L_{\xi}^{\infty}}+\|w_1\|_{L_{\xi}^{\infty}})^2\|U(-1/t)w_1-w_1\|_{L_{\xi}^{\infty}}
%\label{3.11}\\
%&\quad +t^{-1-\alpha}\|U(-1/t)w_1|^2U(-1/t)w_1\|_{H_{\xi}^1}
%\nonumber\\
%&\lesssim
%t^{-1-\alpha}(\|w_1\|_{L_{\xi}^{\infty}}+t^{-\alpha}\|w_1\|_{H_{\xi}^1})^2\|w_1\|_{H_{\xi}^1}
%\nonumber\\
%&\quad +t^{-1-\alpha}\|U(-1/t)w_1\|_{L_{\xi}^{\infty}}^2\|U(-1/t)w_1\|_{H_{\xi}^1}
%\nonumber\\
%&\lesssim
%t^{-1-\alpha}(\|w_1\|_{L_{\xi}^{\infty}}+t^{-\alpha}\|w_1\|_{H_{\xi}^1})^2\|w_1\|_{H_{\xi}^1}
%\nonumber\\
%&\lesssim
%t^{-1-\alpha}(1+t^{-\alpha+\gamma})^2t^{\gamma}\|(w_1,w_2)\|_{X_T}^3\nonumber\\
\lesssim
t^{-1-\alpha+\gamma}\|(w_1,w_2)\|_{X_T}^3,\label{3.11}%\nonumber
\end{align}
%Since
%\begin{align*}
%R_2&=3\la_1 t^{-1}\Large[\{2|U(-1/t)w_1|^2U(-1/t)w_2
%+(U(-1/t)w_1)^2\overline{U(-1/t)w_2}\}\\
%&\qquad\qquad\qquad\qquad\qquad\qquad\qquad\qquad\qquad
%-(2|w_1|^2w_2+w_1^2\overline{w_2})\Large]\\
%&\quad +3\la_1 t^{-1}(U(1/t)-1)\{2|U(-1/t)w_1|^2U(-1/t)w_2
%+(U(-1/t)w_1)^2\overline{U(-1/t)w_2}\}
%\end{align*}
%by Lemma \ref{Lem1}, we have
\begin{align}
\|R_2\|_{L_{\xi}^{\infty}}%}\label{3.12}\\
%&\lesssim
%t^{-1}(\|U(-1/t)w_1\|_{L_{\xi}^{\infty}}+\|w_1\|_{L_{\xi}^{\infty}})\|U(-1/t)w_1-w_1\|_{L_{\xi}^{\infty}}
%\|U(-1/t)w_2\|_{L_{\xi}^{\infty}}
%\nonumber\\
%&\quad +t^{-1}\|w_1\|_{L_{\xi}^{\infty}}^2\|U(-1/t)w_2-w_2\|_{L_{\xi}^{\infty}}
%\nonumber\\
%&\quad +t^{-1-\alpha}\|U(-1/t)w_1|^2U(-1/t)w_2\|_{H_{\xi}^1}
%\nonumber\\
%&\lesssim
%t^{-1-\alpha}(\|w_1\|_{L_{\xi}^{\infty}}+t^{-\alpha}\|w_1\|_{H_{\xi}^1})\|w_1\|_{H_{\xi}^1}
%(\|w_2\|_{L_{\xi}^{\infty}}+t^{-\alpha}\|w_2\|_{H_{\xi}^1})
%\nonumber\\
%&\quad +t^{-1-\alpha}\|w_1\|_{L_{\xi}^{\infty}}^2\|w_2\|_{H_{\xi}^1}
%\nonumber\\
%&\quad +t^{-1-\alpha}\|U(-1/t)w_1\|_{L_{\xi}^{\infty}}^2\|U(-1/t)w_1\|_{H_{\xi}^1}
%\nonumber\\
%&\lesssim
%t^{-1-\alpha}(\|w_1\|_{L_{\xi}^{\infty}}+t^{-\alpha}\|w_1\|_{H_{\xi}^1})\|w_1\|_{H_{\xi}^1}
%(\|w_2\|_{L_{\xi}^{\infty}}+t^{-\alpha}\|w_2\|_{H_{\xi}^1})
%\nonumber\\
%&\quad +t^{-1-\alpha}(\|w_1\|_{L_{\xi}^{\infty}}+t^{-\alpha}\|w_1\|_{H_{\xi}^1})^2\|w_2\|_{H_{\xi}^1}
%\nonumber\\
%&\lesssim
%\left\{t^{-1-\alpha+\gamma}(1+t^{-\alpha+\gamma})(1+t^{-\alpha+\delta})
%+t^{-1-\alpha+\delta}(1+t^{-\alpha+\gamma})^2\right\}\|(w_1,w_2)\|_{X_T}^3
%\nonumber\\
\lesssim t^{-1-\alpha+\delta}\|(w_1,w_2)\|_{X_T}^3.\label{3.12}%\nonumber
\end{align}
By (\ref{3.9}), 
\begin{align*}
\pt_t|w_1|^2=2\Im(R_1\overline{w}_1)\lesssim\|R_1(t)\|_{L_{\xi}^{\infty}}|w_1|.
\end{align*}
Hence (\ref{3.11}) yields
\begin{align}
\pt_t|w_1|\lesssim\|R_1(t)\|_{L_{\xi}^{\infty}}
\lesssim t^{-1-\alpha+\gamma}\|(w_1,w_2)\|_{X_T}^3.
\label{3.131}
\end{align}
Therefore
\begin{align}
|w_1(t,\xi)|\lesssim\varepsilon+\|(w_1,w_2)\|_{X_T}^3.
\label{3.13}
\end{align}
By (\ref{3.9}), (\ref{3.10}), (\ref{3.11}) and (\ref{3.12}),
\begin{align}
\pt_t\Re(w_1\overline{w}_2)
&=
\Im(R_1\overline{w}_2)+\Im(R_2\overline{w}_1)
\label{3.141}\\
&\lesssim
\|R_1\|_{L_{\xi}^{\infty}}\|w_2\|_{L_{\xi}^{\infty}}
+\|R_2\|_{L_{\xi}^{\infty}}\|w_1\|_{L_{\xi}^{\infty}}
\nonumber\\
&\lesssim
(t^{-1-\alpha+\gamma}\log t+t^{-1-\alpha+\delta})\|(w_1,w_2)\|_{X_T}^4
\nonumber\\
&\lesssim
t^{-1-\alpha+\delta}\|(w_1,w_2)\|_{X_T}^4.
\nonumber
\end{align}
Integrating this in $t$, we obtain
\begin{align}
|\Re(w_1\overline{w}_2)|
\lesssim\varepsilon^2+\|(w_1,w_2)\|_{X_T}^4.\label{3.14}
\end{align}
On the other hand, by (\ref{3.10}), 
\begin{align*}
i\pt_tw_2=3\la_1 t^{-1}|w_1|^2w_2+6\la_1 t^{-1}w_1\Re(w_1\overline{w_2})+R_2.
\end{align*}
Hence,
\begin{align*}
\pt_t|w_2|^2
&=12\la_1 t^{-1}\Re(w_1\overline{w_2})\Im(w_1\overline{w_2})+2\Im(R_2\overline{w_2})\\
&\lesssim t^{-1}|\Re(w_1\overline{w_2})||w_1||w_2|+|R_2||w_2|.
\end{align*}
Therefore (\ref{3.12}) and (\ref{3.14}) imply 
\begin{align*}
\pt_t|w_2|
&\lesssim t^{-1}|\Re(w_1\overline{w_2})||w_1|+|R_2|\\
&\lesssim 
\varepsilon^2t^{-1}\|(w_1,w_2)\|_{X_T}
+t^{-1}\|(w_1,w_2)\|_{X_T}^5+t^{-1-\alpha+\delta}\|(w_1,w_2)\|_{X_T}^3.
\end{align*}
Integrating this in $t$, we have
\begin{align}
|w_2(t,\xi)|
&\lesssim \varepsilon+\varepsilon^2\log t\|(w_1,w_2)\|_{X_T}\label{3.15}\\
&\quad +\log t\|(w_1,w_2)\|_{X_T}^5+\|(w_1,w_2)\|_{X_T}^3.\nonumber
\end{align}
Collecting (\ref{n1}), (\ref{n2}), (\ref{3.13}) and (\ref{3.15}), we
obtain
\begin{align*}
\|(w_1,w_2)\|_{X_T}
\lesssim \varepsilon+\varepsilon^2\|(w_1,w_2)\|_{X_T}+\|(w_1,w_2)\|_{X_T}^3
+\|(w_1,w_2)\|_{X_T}^5.
\end{align*}
The the standard continuity argument yields that 
if $\varepsilon=\varepsilon(\gamma,\delta)$ is sufficiently small, 
then we have
\begin{align*}
\|(w_1,w_2)\|_{X_T}\lesssim\varepsilon
\end{align*}
for any $T\ge1$. This completes the proof of Lemma \ref{Lem:Long}. $\qed$

%%%%%%%%%%%%%%%%%%%%%%%%%%%%%%%%%%%%%%%%%
\vskip2mm
\noindent
{\bf Proof of Theorem \ref{T:main4}.} 
The global existence and decay estimates for solution to 
(\ref{E:sysnewa1}) follow from Lemma \ref{Lem:Long} and 
Remark \ref{Rem:Long}. We now derive the asymptotic formulas 
(\ref{asym1}) and (\ref{asym2}) for solution $(u_1,u_2)$ to 
(\ref{E:sysnewa1}) as $t\to\infty$.

As in the proof of Theorem \ref{T:main_add}, 
%By (\ref{3.131}), we see that there exists a real-valued positive function 
%$\widetilde{W}_1\in L^{\infty}$ such that 
%\begin{align}
%\||w_1|-\widetilde{W}_1\|_{L_{\xi}^{\infty}}\lesssim\varepsilon^3 t^{-\alpha+\gamma},
%\label{e1}
%\end{align}
%where $\delta<\alpha<1/4$. Substituting this into (\ref{3.9}), we have
%\begin{align}
%i\pt_tw_1=3\la_1 t^{-1}\widetilde{W}_1^2w_1+R_3,\label{e2}
%\end{align}
%where $R_3$ is given by
%\begin{align*}
%R_3=3\la_1 t^{-1}(|w_1|^2-\widetilde{W}_1^2)w_1+R_1.
%\end{align*}
%Hence by (\ref{3.11}) and (\ref{e1}), 
%\begin{align*}
%\|R_3(t)\|_{L_{\xi}^{\infty}}\lesssim\varepsilon^3 t^{-1-\alpha+\gamma}.
%\end{align*}
%Multiplying (\ref{e2}) by $\exp(3i\la_1\tilde{W}_1^2\log t)$, we have
%\begin{align*}
%i\pt_t(w_1e^{3i\la_1\widetilde{W}_1^2\log t})=R_3e^{3i\la_1\widetilde{W}_1^2\log t}.
%\end{align*}
we find that there exists $W_1\in L^{\infty}$ such that 
\begin{align}
\|w_1e^{3i\la_1|{W}_1|^2\log t}-W_1\|_{L_{\xi}^{\infty}}\lesssim\varepsilon^3 t^{-\alpha+\gamma}.
\label{e6}
\end{align}
%We easily see that $|W_1|=\widetilde{W}_1$.

Furthermore, by (\ref{3.141}), 
there exists a real-valued function $\widetilde{W}\in L^{\infty}$ such that 
\begin{align}
\|\Re(w_1\overline{w}_2)-\widetilde{W}\|_{L_{\xi}^{\infty}}
\lesssim\varepsilon^4 t^{-\alpha+\delta}.
\label{e3}
\end{align}
Substituting this into (\ref{3.10}), we have
\begin{align}
i\pt_tw_2=3\la_1 t^{-1}|W_1|^2w_2+6\la_1 t^{-1}W_1\widetilde{W}e^{-3i\la_1 |W_1|^2\log t}+R_4,
\label{e4}
\end{align}
where
\begin{align*}
R_4&=
3\la_1 t^{-1}(|w_1|^2-|W_1|^2)w_2
+6\la_1 t^{-1}(w_1\Re(w_1\overline{w}_2)-W_1\widetilde{W}e^{-3i\la_1 |W_1|^2\log t})\\
&\quad +R_2.
\end{align*}
Hence by (\ref{3.12}), (\ref{e6}) and (\ref{e3}), 
\begin{align}
\|R_4(t)\|_{L_{\xi}^{\infty}}\lesssim\varepsilon^5 t^{-1-\alpha+\delta}.
\label{e7}
\end{align}
Multiplying (\ref{e4}) by $\exp(3i\la_1 |W_1|^2\log t)$, we find 
\begin{align*}
i\pt_t\left\{w_2e^{3i\la_1 |W_1|^2\log t}\right\}=6\la_1 t^{-1}W_1\widetilde{W}+R_4e^{3i\la_1 |W_1|^2\log t}.
\end{align*}
Since $W_1$ and $\widetilde{W}$ are independent of $t$, we see
\begin{align*}
i\pt_t\left\{w_2e^{3i\la_1 |W_1|^2\log t}+6i\la_1\log tW_1\widetilde{W}\right\}
=R_4e^{3i\la_1 |W_1|^2\log t}.
\end{align*}
Hence, by (\ref{e7}) we find that 
there exists $W_2\in L^{\infty}$ such that 
\begin{align}
\|w_2e^{3i\la_1 |W_1|^2\log t}+6i\la_1\log tW_1\widetilde{W}-W_2\|_{L_{\xi}^{\infty}}
\lesssim\varepsilon^5 t^{-\alpha+\delta}.\label{e9}
\end{align}
Especially, we have $\widetilde{W}=\Re(W_1\overline{W}_2)$.

In the same way as in the proof of (\ref{qq1}), 
from the asymptotic formulas (\ref{e6}) for $w_1$ and (\ref{e9}) 
for $w_{2}$, we obtain (\ref{asym1}) and (\ref{asym2}). 
%Lemma \ref{Lem1} and the asymptotics for $w_2$ (\ref{e9}) 
%yield
%\begin{align*}
%u_2(t)
%&=U(t){{\mathcal F}}^{-1}w_2\\
%&=M(t)D(t)U(-1/t)w_2\\
%&=M(t)D(t)w_2+O(t^{-\frac12-\alpha})\\
%&=-2it^{-\frac12}\log tW_1\left(\frac{x}{t}\right)
%\Re\left[W_1\left(\frac{x}{t}\right)\overline{W_2\left(\frac{x}{t}\right)}\right]
%e^{\frac{ix^2}{2t}-i\la\left|W_1\left(\frac{x}{t}\right)\right|^2\log t-i\frac{\pi}{4}}\\
%&\quad +t^{-\frac12}W_2\left(\frac{x}{t}\right)
%e^{\frac{ix^2}{2t}-i\la\left|W_1\left(\frac{x}{t}\right)\right|^2\log t-i\frac{\pi}{4}}
%+O(t^{-\frac12-\alpha+\delta})
%\end{align*}
%in $L_x^{\infty}$ as $t\to\infty$. 

For the case $\la_6=\la_1$, it suffices to replace $\Re(w_1\overline{w}_2)$ 
by $\Im(w_1\overline{w}_2)$ in the proof for the case $\la_6=3\la_1$. 
This completes the proof of Theorem \ref{T:main4}. $\qed$

%%%%%%%%%%%%%%%%%%%%%
%
%  Appendix 1
%
%%%%%%%%%%%%%%%%%%%%%

\appendix

\section{Classification for system of cubic nonlinear system}
\subsection{Set of cubic nonlinear systems}\label{subsec:a1}
In this appendix, we consider classification of abstract cubic nonlinear systems  of the form
\begin{equation}\label{eq:cubicsys}
\left\{
\begin{aligned}
& \mathcal{L} u_1 = 
c_1|u_1|^2u_1+c_2|u_1|^2u_2+c_3u_1^2\overline{u_2}+c_4u_1|u_2|^2+c_5\overline{u_1}u_2^2+c_6|u_2|^2u_2, \\
& \mathcal{L} u_2 =
c_7|u_1|^2u_1+c_8|u_1|^2u_2+c_9u_1^2\overline{u_2}+c_{10}u_1|u_2|^2
+c_{11}\overline{u_1}u_2^2+c_{12}|u_2|^2u_2,
\end{aligned}
\right.
\end{equation}
where $u_j: \Omega \to \C$ ($j=1,2$) are complex-valued unknown functions defined on a set $\Omega$, $\mathcal{L}$ is a $\C$-linear operator,
and $c_j\ (j=1,\cdots,12)$ are real constants. 
Obviously, our system \eqref{eq:NLS} is an example of \eqref{eq:cubicsys}.
Another example of a system of the form \eqref{eq:cubicsys} in our mind is the limit ODE system:
\begin{equation}
\label{eq:ode}
\left\{
\begin{aligned}
i\partial_t \alpha_1&=
c_1|\alpha_1|^2\alpha_1+c_2|\alpha_1|^2\alpha_2
+c_3\alpha_1^2\overline{\alpha_2}+c_4\alpha_1|\alpha_2|^2
+c_5\overline{\alpha_1}\alpha_2^2+c_6|\alpha_2|^2\alpha_2\\
i\partial_t \alpha_2 &=
c_7|\alpha_1|^2\alpha_1+c_8|\alpha_1|^2\alpha_2
+c_9\alpha_1^2\overline{\alpha_2}+c_{10}\alpha_1|\alpha_2|^2
+c_{11}\overline{\alpha_1}\alpha_2^2+c_{12}|\alpha_2|^2\alpha_2,
\end{aligned}
\right.
\end{equation}
which is a generalized version of \eqref{eq:limitODE}.
The system arises in the study of asymptotic behavior of solutions to \eqref{eq:NLS}.
Indeed, the main part of the systems \eqref{3.91}-\eqref{3.101}, \eqref{3.9}-\eqref{3.10} and \eqref{e4}
are examples of \eqref{eq:ode}.
Although we discuss in a quite abstract setting, our main targets are \eqref{eq:NLS} and \eqref{eq:ode}.

For each fixed $\Omega$ and $\mathcal{L}$,
a system of the form \eqref{eq:cubicsys} is naturally regarded as a vector $(c_j)_{1 \le j \le 12}\in \R^{12}$ by identifying the system with a combination of constants of the nonlinearity.
Let $\mathrm{CNS}=\mathrm{CNS}(\Omega, \mathcal{L})$ be a set of cubic nonlinear systems of the form \eqref{eq:cubicsys}.

We introduce a subclass of systems of the form
\begin{equation}\label{eq:cubicsys1}
	\left\{
	\begin{aligned}
		\mathcal{L} u_1 
		&= 3\lambda_1 |u_1|^2u_1 + \lambda_2(2 |u_1|^2 u_2 + u_1^2 \overline{u_2}) + \lambda_3 (2 u_1 |u_2|^2 + \overline{u_1} u_2^2) + 3 \lambda_4 |u_2|^2u_2,
		\\
		\mathcal{L} u_2 
		&= 3\lambda_5 |u_1|^2u_1 + \lambda_6 (2 |u_1|^2 u_2 + u_1^2 \overline{u_2}) + \lambda_7 ( 2 u_1 |u_2|^2 + \overline{u_1} u_2^2) + 3\lambda_8 |u_2|^2u_2,
	\end{aligned}
	\right.
\end{equation}
where $ \lambda_j\ (j = 1, \cdots, 8)$ are real constants.
It is a particular case 
\begin{align*}
	c_1 ={}& 3 \l_1, & c_2 ={}& 2c_3 = 2\l_2, & c_4 ={}& 2c_5 = 2\l_3, & c_6 ={}& 3 \l_4, \\
	c_7 ={}& 3 \l_5, & c_8 ={}& 2c_9 = 2\l_6, & c_{10} ={}& 2c_{11} = 2\l_7, & c_{12} ={}& 3 \l_8
\end{align*}
of \eqref{eq:cubicsys}.
Let $\mathrm{CNS}_A \subset \mathrm{CNS}$ be the set of systems of the form \eqref{eq:cubicsys1}.
Notice that if we take \eqref{eq:NLS} as an example of $\mathrm{CNS}$, then $\mathrm{CNS}_A$ is the set of systems of the form \eqref{eq:NLS1}.

The system \eqref{eq:cubicsys} is closed under the change of variables of the form
\begin{equation}\label{E:linearchange}
	\begin{pmatrix} v_1 \\ v_2 \end{pmatrix} = M\begin{pmatrix} u_1 \\ u_2 \end{pmatrix},
	\quad M \in GL_2(\R).
\end{equation}
Namely,
if we define a new unknown $(v_1,v_2)$ by the above transform then $(v_1,v_2)$ solves a system of the same form, with a possibly different combination of $c_j$.
Then, a natural equivalence relation ${\sim}$ on $\mathrm{CNS}$ is induced by the change of variables of this form.
Our interest is in the study on the quotient set $\mathrm{CNS}/{\sim}$.

\subsection{Matrix representation of systems}
In the previous study in \cite{MSU}, 
we introduced a matrix representation of systems in $\mathrm{CNS}_A$.
The representation enables us to describe the equivalence relation by matrix multiplications.
We briefly recall the representation.
Then, we consider a similar representation of systems in $\mathrm{CNS}$, which is our first main result in this appendix.
It turns out that our representation is a natural extension of that for $\mathrm{CNS}_A$.

\subsubsection{Representation of systems in $\mathrm{CNS}_A$}

To begin with let us recall the matrix representation for $\mathrm{CNS}_A$.
Let
\[
	\mathcal{Z} := \{  A \in M_3(\R) \ |\ \tr A = 0 \}.
\]
A system $\sigma = (\lambda_j)_{1 \le j \le 8} \in \mathrm{CNS}_A$ parameterized as in \eqref{eq:cubicsys1} is identified with a matrix $A=A(\sigma) \in \mathcal{Z}$ defined by 
\begin{equation}\tag{\ref{eq:defA}}
	A = \begin{pmatrix}
	\l_2  & -3\l_1 + \l_6& -3\l_5 \\
	\l_3  & -\l_2 + \l_7  & -\l_6 \\
	3\l_4 & 3\l_8 - \l_3 &-\l_7
	\end{pmatrix}.
\end{equation}
Now, we introduce a notation. For a given matrix $A= (a_{ij})_{1 \le i,j \le 3} \in \mathcal{Z}$, we denote 
the corresponding system $\sigma \in \mathrm{CNS}_A $ by
\begin{equation}\label{eq:cubicsys1p}
	\left\{
	\begin{aligned}
		\mathcal{L} u_1 
		&= Q_{1,A} (u_1,u_2),
		\\
		\mathcal{L} u_2 
		&= Q_{2,A} (u_1,u_2),
	\end{aligned}
	\right.
\end{equation}
where the nonlinearities $Q_{1,A}$ and $Q_{2,A}$ are given by
\begin{equation}\label{eq:QjA}
	Q_{k,A} (u,v) := 3\lambda_{4k-3} |u|^2u + \lambda_{4k-2}(2|u|^2v+u^2\overline{v}) 
		+ \lambda_{4k-1} (2u|v|^2+\overline{u}v^2) + 3\lambda_{4k} |v|^2v
\end{equation}
($k=1,2$) with the parameters
$ \lambda_j\ (j = 1, \cdots, 8)$ given from $A$ by the relation \eqref{eq:defA}.

\subsubsection{Representation of systems in $\mathrm{CNS}$}

Given system $\sigma \in \mathrm{CNS}$,
we introduce a matrix $C=C(\sigma)$ by
\begin{equation}\label{eq:matrixC}
C = 
\begin{pmatrix}
c_2-c_3 & -c_1+c_8-c_9 & -c_7 \\
c_5 & -c_3+c_{11} & -c_9 \\
c_6 & -c_4+c_5+c_{12} & -c_{10}+c_{11}
\end{pmatrix} \in M_3(\R), 
\end{equation}
where $M_3(\R)$ is a set of $3\times 3$ matrices with real entries.
Remark that the correspondence $\mathrm{CNS}\ni \sigma \mapsto C(\sigma) \in M_3(\R)$ is not injective. 
To obtain a unique representation with $C(\sigma)$, let us take the kernel of the map $\mathfrak{G}: \sigma \mapsto C(\sigma)$ into account.
\begin{lemma}[Kernel of $\mathfrak{G}$]
$\ker \mathfrak{G} = \mathrm{Span} ({\bf c}^{(1)},{\bf c}^{(2)},{\bf c}^{(3)}) $, where ${\bf c}^{(k)} = (c_j^{(k)})_{1 \le j \le 12} \in \R^{12}$ ($k=1,2,3$)
is given by
\[
		c_j^{(1)} = \delta_{1,j} + \delta_{8,j}, \quad
		c_j^{(2)} =  \delta_{2,j} + \delta_{3,j}+ \delta_{10,j} + \delta_{11,j}, \quad
		c_j^{(3)} = \delta_{4,j} + \delta_{12,j},
\]
respectively, with the Kronecker delta $\delta_{j,k}$. 
\end{lemma}

\begin{theorem}[A matrix-kernel representation of $\mathrm{CNS}$]\label{thm:mkrep}
A system $\sigma \in \mathrm{CNS}$ is identified with $(C,p,q,r) \in M_3(\R) \times \R^3$ as follows:
If $\sigma = (c_j)_{1 \le j \le 12}\in \mathrm{CNS}$ then
$C(\sigma)$ is given by \eqref{eq:matrixC} and $(p(\sigma),q(\sigma),r(\sigma))$ is given by 
\[
	p (\sigma) = c_8-2c_9, \quad  q (\sigma) = \tfrac12 (-c_2+2c_3-c_{10}+ 2c_{11})  ,\quad r (\sigma) = c_4-2c_5
\]
respectively.
On the other hand, if $((a_{ij})_{1 \le i,j \le 3},p,q,r)  \in M_3(\R) \times \R^3$ is given then 
$\sigma = (c_j)_{1 \le j \le 12} \in \mathrm{CNS}$ is given by
$c_5  = a_{21}$,  $c_6  = a_{31}$,  $c_7  = -a_{13}$,  $c_9  = -a_{23}$, 
\begin{align*}
	c_1 &{} = p-a_{12} - a_{23}, &
	c_8 &{} = p -2 a_{23}, \\
	c_2 &{} = q+ \tfrac12 (3a_{11} -  a_{22}-a_{33}),&
	c_3 &{} = q+ \tfrac12 (a_{11} -a_{22}-a_{33}), \\
	c_{10} &{} = q+ \tfrac12 (a_{11} +  a_{22}-3a_{33}), &
	c_{11} &{} = q+ \tfrac12 (a_{11} +  a_{22}-a_{33}),\\
	c_4 &{} = r+ 2 a_{21} , &
	c_{12} &{} =r+a_{21} + a_{32} .
\end{align*}
Further, any system $\sigma \in \mathrm{CNS}$ is uniquely written as
\begin{equation}\label{eq:cubicsys2}
	\left\{
	\begin{aligned}
		\mathcal{L} u_1 
		&= Q_{1,C(\sigma)} (u_1,u_2)  + V_{p(\sigma),q(\sigma),r(\sigma)}(u_1,u_2) u_1 ,
		\\
		\mathcal{L} u_2
		&= Q_{2, C(\sigma)} (u_1,u_2) + V_{p(\sigma),q(\sigma),r(\sigma)}(u_1,u_2)u_2,
	\end{aligned}
	\right.
\end{equation}
where 
\begin{equation}\label{eq:QjC}
	Q_{j,C}(u,v) := Q_{j, C- (\tr C)I_{22}} (u,v) + (-1)^{j} (\tr{C}) \Re (\overline{u} v) (\delta_{j,1}u + \delta_{j,2} v)
\end{equation}
is an extension of $Q_{j,A}$ define in \eqref{eq:QjA},
$I_{22} = (\delta_{i,2}\delta_{j,2})_{1 \le i,j \le 3} \in M_3(\R)$,
and
\begin{equation}\label{eq:defVpqr}
	V_{p,q,r}(u,v) := p |u|^2 + 2q \Re (\overline{u}v)  + r |v|^2 = 
	\begin{pmatrix} \overline{u} &\overline{v}  \end{pmatrix}
	\begin{pmatrix} p & q \\ q & r \end{pmatrix}
	\begin{pmatrix} u \\ v  \end{pmatrix}.
\end{equation}

\end{theorem}
The proof is straightforward.
In the sequel, we freely use the identification $\sigma$ with $(C,p,q,r)$.

\subsubsection{Characterization of $\mathrm{CNS}_A$}
To observe that the matrix-kernel expression for $ \mathrm{CNS}$ given in Theorem  \ref{thm:mkrep}
is an extension of the matrix representation for $ \mathrm{CNS}_A$
introduced in \cite{MSU},
we give a characterization for being a system in $ \mathrm{CNS}_A$ in the matrix-kernel representation.

\begin{proposition}
A system
$\sigma \in  \mathrm{CNS}$ belongs to $\mathrm{CNS}_A$ if and only if $\tr C(\sigma)=0$ and $p(\sigma)=q(\sigma)=r(\sigma)=0$.
Further, if $\sigma \in \mathrm{CNS}_A$ then $C(\sigma)$
coincides with the matrix $A$ given in \eqref{eq:defA}.
\end{proposition}
The proof is due to simple computation.
By means of the form \eqref{eq:cubicsys2} and  \eqref{eq:QjC}, it is also possible to identify $\sigma \in \mathrm{CNS}$ with $(A, \tau, p,q,r) \in \mathcal{Z} \times \R^4$,
where the parameter $\tau$ corresponds to $\tr C$. 
%The map from $M_3(\R)$ to $\mathcal{Z} $ is given by $C \mapsto C - (\tr C)I_{22} $.

\subsection{Equivalence relation in terms of matrices}

We have introduced an equivalence relation between two systems in Section \ref{subsec:a1}.
By the matrix-kernel representation of systems in $\mathrm{CNS}$ given in Theorem \ref{thm:mkrep}, an equivalence relation on
$M_3 (\R) \times \R^3$ is naturally induced.
One benefit of the introduction of the matrix-kernel representation is that the equivalence relation can be explicitly described as matrix multiplications.

\begin{theorem}\label{thm:equivalence}
Let $(C,p,q,r)$ be a matrix-kernel representation of a system $\sigma \in 
\mathrm{CNS}$ with unknowns $(u_1,u_2)$.
Define a new unknown $(v_1,v_2)$ via
\[
	\begin{pmatrix} v_1 \\ v_2 \end{pmatrix} = M \begin{pmatrix} u_1 \\ u_2 \end{pmatrix},\quad
	M = \begin{pmatrix} a & b \\ c & d \end{pmatrix} \in \text{GL}_2(\R).
\]
Let $(C',p',q',r')$ be the matrix-kernel representation of the system for $(v_1,v_2)$. Then,
\begin{equation}\label{eq:Cprime}
	C' = \frac{1}{\det M}  D(M) C D(M)^{-1}
\end{equation}
and
\begin{equation}\label{eq:Vprime}
	\begin{pmatrix} p' \\ q' \\ r' \end{pmatrix}
	= \frac1{\det M}  D(M) 	\begin{pmatrix} p \\ q \\ r \end{pmatrix}
\end{equation}
hold, where
\[
	D(M) := \frac1{\det M}
	\begin{pmatrix} d^2 & -2dc & c^2 \\ -bd & ad+ bc & -ac \\ b^2 & -2ab & a^2 \end{pmatrix}
	\in {SL}_3 (\R),
\]
\begin{equation*}%\label{E:DM-1form}
	D(M)^{-1} =\frac1{\det M}\begin{pmatrix} a^2 & 2ac & c^2 \\ ab & ad+ bc & cd \\ b^2 & 2bd & d^2 \end{pmatrix} \in SL_3(\R).
\end{equation*}
In particular, $\rank C = \rank C'$. Further, $\sign ((q')^2 - p'r') = \sign (q^2-pr)$.
\end{theorem}
This can be shown by a direct computation.
One sees from \eqref{eq:Cprime} that $\tr C' = (\det M)^{-1} \tr C$. 
Hence the property $\tr C = 0$ is conserved by any change of unknowns.
%Notice that the formula \eqref{eq:Cprime} is the same as that for the change of $A \in Z$.

\subsection{Matrix representation and conserved quantities}
Remark that there are several options on manners
to represent a system with an explicit description of the equivalence relation.
The simplest one would be the natural identification $\mathrm{CNS}=\R^{12}$, that is, identification with a column vector in $\R^{12}$.
Then, there is a $12 \times 12$-matrix-valued function on $GL_2(\R)$ such that the equivalence relation on $\mathrm{CNS}$ is understood as a multiplication of the matrix from the left.
The good point of our choice is not only the capability of a clear description of the equivalence relation but also 
the fact that it reflects the characteristic property of systems.
In particular, as for the concrete examples \eqref{eq:NLS} and \eqref{eq:ode}, our representation also clarifies
the existence of  conserved quantities.

\subsubsection{Conserved quantities for \eqref{eq:ode}}
Let us consider \eqref{eq:ode}. In this model, the kernel part of the nonlinearity is 
removed by a gauge transform:
Let $I \subset \R$ be an interval such that $0 \in I$.
Let $(u_1,u_2) \in C^1(I,\C^2)$ be a solution to
\begin{equation}\tag{\ref{eq:ode}'}
	\left\{
	\begin{aligned}
		i \partial_t u_1 
		&= Q_{1,C} (u_1,u_2) + V_{p,q,r}(u_1,u_2)  u_1 ,
		\\
		i \partial_t u_2
		&= Q_{2, C} (u_1,u_2) + V_{p,q,r}(u_1,u_2)u_2
	\end{aligned}
	\right.
\end{equation}
for some $(C,p,q,r) \in \mathrm{CNS}$.
Then, $\tilde{u}_j (t):= u_j(t) e^{ i\int^t_0 V_{p,q,r}(u_1(s),u_2(s)) ds}$ solves
\[
	\left\{
	\begin{aligned}
		i \partial_t \tilde{u}_1 
		&= Q_{1,A(\sigma)} (\tilde{u}_1,\tilde{u}_2) ,
		\\
		i \partial_t \tilde{u}_2
		&= Q_{2, A(\sigma)} (\tilde{u}_1,\tilde{u}_2) ,
	\end{aligned}
	\right.
\]
which is the system $(C,0,0,0)\in \mathrm{CNS}$. This can be seen, for instance,
by the fact that the nonlinearity is gauge-invariant and that $V_{p,q,r}$ is $\R$-valued.
Since $V_{p,q,r} (u_1,u_2) = V_{p,q,r} (\tilde{u}_1,\tilde{u}_2)$, the inverse transform is 
given by $u_j (t):= \tilde{u}_j(t) e^{- i\int^t_0 V_{p,q,r}(u_1(s),u_2(s)) ds}$.
Thus, as for this model it suffices to consider the case $p=q=r=0$.
Further, we have the following:
\begin{proposition}\label{P:ODEconservation}
Let $I \subset \R$ be an interval.
Let $(u_1,u_2) \in C^1 (I;\C^2)$ be a solution to \eqref{eq:ode}. Then, for any $\ltrans{(a,b,c)} \in \R^3$,
\[
	\frac{d}{dt} \begin{pmatrix} |u_1|^2 & 2\Re (\overline{u_1} u_2) & |u_2|^2 \end{pmatrix} 
	\begin{pmatrix} a \\ b \\ c \end{pmatrix}
	= \Im (\overline{u_1}u_2) \begin{pmatrix} |u_1|^2 & 2\Re (\overline{u_1} u_2) & |u_2|^2 \end{pmatrix} C
	\begin{pmatrix} a \\ b \\ c \end{pmatrix}.
\]
In particular, if $\ltrans{(a,b,c)} \in \ker C$ then $a |u_1|^2 + 2 b \Re (\overline{u_1} u_2) +c |u_2|^2$
is conserved. Further, $3-\rank C$ is the number of linearly dependent conserved quantities of this form.
\end{proposition}

Let us consider another type of conserved quantity.
The following proposition makes the role of the matrix $B$ defined in \eqref{eq:defB} clear.

\begin{proposition}
\begin{enumerate}
\item
Let $I \subset \R$ be an interval.
Let $(u_1,u_2) \in C^1 (I;\C^2)$ be a solution to \eqref{eq:ode}. Then,
\[
	\frac{d}{dt} \Im (\overline{u_1}u_2)
	= \frac14 
	B[
	 |u_1|^2 , 2\Re (\overline{u_1} u_2) , |u_2|^2 ],
\]
where $B[x] := \ltrans{x} B x$ is the quadratic form  for $x \in \R^3$ given by the matrix
\begin{equation}\label{eq:defBp}
	B:= \begin{pmatrix}
	-4c_7 & c_1 - c_8 - c_9 & 2 (c_2-c_3-c_{10}+c_{11}) \\
	c_1 - c_8 - c_9 & 2(c_3 - c_{11})& c_4+c_5 - c_{12} \\
	 2 (c_2-c_3-c_{10}+c_{11}) & c_4+c_5 - c_{12} & 4 c_6
	\end{pmatrix}.
\end{equation}
In particular, if $B=O$ then $2 \Im (\overline{u_1} u_2) $ is conserved. 
\item
If a system $\sigma \in \mathrm{CNS}$ satisfies $B(\sigma)=O$ and $p(\sigma)=q(\sigma)=r(\sigma)=0$
then $\sigma \in \mathrm{CNS}_A$ and $C(\sigma)$ takes the form
\[
	\begin{pmatrix}
	\l_2 & -2\l_6 & 0 \\
	\l_3 & 0 & -\l_6  \\
	 0 & 2 \l_3 &-\l_2
	\end{pmatrix} \in \mathcal{Z},
\]
i.e., $C(\sigma)$ is a matrix of the form \eqref{eq:defA} and satisfies $\l_1=\l_6$, $\l_4=\l_5=0$, $\l_7=\l_2$, and $\l_8=\l_3$.
In particular, $\rank C(\sigma)\in \{0,2\}$.
\end{enumerate}
\end{proposition}

The proof is due to a direct computation.
Remark that the matrix $B$ defined in \eqref{eq:defBp} is an extension of the matrix given by \eqref{eq:defB}.
Namely if $\sigma \in \mathrm{CNS}_A$
then the two definitions give the same matrix.

\subsubsection{Conserved quantities for \eqref{eq:NLS}}
As for the nonlinear Schr\"odinger system \eqref{eq:NLS}, the mass-type conservation is well-described by the matrix $C$.
\begin{proposition}\label{P:massconservation}
Let $I \subset \R$ be an interval.
Let $(u_1,u_2) \in (C (I; L^2(\R)) \cap L^4 (I; L^\I (\R)))^2$ be a local $L^2$-solution to \eqref{eq:NLS}. Then, one has\footnote{Rigorously speaking, it holds in the corresponding integral form.}
\[
	\partial_t 
	\begin{pmatrix} \norm{ u_1 }_{L^2}^2 & 2 \Re (u_1,u_2)_{L^2} & \norm{ u_2}_{L^2}^2 \end{pmatrix} 
	\begin{pmatrix} a \\ b \\ c \end{pmatrix}
	= \int_\R 2 \Im (\overline{u_1}u_2) \begin{pmatrix} |u_1|^2 & 2\Re (\overline{u_1} u_2) & |u_2|^2 \end{pmatrix} C
	\begin{pmatrix} a \\ b \\ c \end{pmatrix}dx.
\]
In particular,
\[
	a \norm{ u_1 }_{L^2}^2 + 2b \Re (u_1,u_2)_{L^2} + c \norm{ u_2}_{L^2}^2
\]
is a conserved quantity if and only if $\ltrans{(a,b,c)} \in \ker C$.
Moreover, $3-\rank C$ is the number of linearly dependent conserved quantities of this form.
Furthermore, if there exists $(a_0,b_0,c_0)\in \ker C$ such that $b_0^2 < a_0 c_0$ then any $L^2$-solution to \eqref{eq:NLS} is global in time.
\end{proposition}
Let us consider the energy conservation. We seek an energy of the form
\begin{equation}\label{eq:NLSenergy}
	a \norm{ \nabla u_1 }_{L^2}^2 + 2b \Re (\nabla u_1,\nabla u_2)_{L^2} + c \norm{ \nabla u_2}_{L^2}^2 + (\text{quartic term}).
\end{equation}
It turns out that the kernel part must be taken into account.
\begin{proposition}
Let $I \subset \R$ be an interval.
Let $(u_1,u_2) \in (C (I; H^1(\R)) \cap L^8 (I; W^{1,4}(\R)))^2$ be a local $H^1$-solution to \eqref{eq:NLS}.
There exists a conserved energy of the form \eqref{eq:NLSenergy} if and only if
\[
	C \begin{pmatrix} a \\ b \\ c \end{pmatrix} = 0, \quad
	\begin{pmatrix} \tr C - 2q & 2p & 0 \\ -r & \tr C & p  \\ 0 & - 2r & \tr C + 2q \end{pmatrix} \begin{pmatrix} a \\ b \\ c \end{pmatrix}=0.
\]
Further, if the condition is satisfied then the quartic term of the energy is given by
\begin{align*}
	\int &\big[ (ac_1 + b c_7) |u_1|^4 +4 (ac_3 + bc_9) |u_1|^2 \Re (\overline{u_1}u_2)
	+ 2(ac_4 + bc_{10}) |u_1|^2 |u_2|^2 \\
	&+ 2(ac_5 + b c_{11}) \Re (\overline{u_1}^2 u_2^2) 
	+  4 (bc_{5} + cc_{11}) |u_2|^2 \Re (\overline{u_1}u_2) + (bc_6 + cc_{12}) |u_2|^4 \big] dx .
\end{align*}
In particular, if the system is of the form \eqref{eq:NLS1}${}\in \mathrm{CNS}_A$ then the condition is simplified as $\ltrans{(a,b,c)} \in \ker C$.
\end{proposition}

\subsection{Classification of  systems}

We now turn to the main issue of the appendix, i.e., the classification of the systems. 
To this end, we introduce a notion.
For a system $\sigma \in \mathrm{CNS}$, we let $[\sigma] := \{ \sigma' \in \mathrm{CNS} \ |\ \sigma' \sim \sigma\} \in \mathrm{CNS}/{\sim}$ be an equivalent class contains $\sigma$, or an orbit of $\sigma$.
For a subset $S \subset \mathrm{CNS}$, we define $[S]:= \{[\sigma] \in \mathrm{CNS}/{\sim} \ |\ \sigma \in S \} \subset \mathrm{CNS}/{\sim}$.

\begin{definition}[System of representatives (SR)]\label{def:SR}
Let $S \subset \mathrm{CNS}/{\sim}$ be a set of equivalent classes.
We say a subset $T \subset \mathrm{CNS}$ is a \emph{system of representatives (SR)} of $S$
 if $T$ satisfies the following two properties:
\begin{enumerate}
\item $[T] = S$;
\item if $\sigma_1 , \sigma_2 \in T$ satisfies $\sigma_1 \sim \sigma_2$ then $\sigma_1=\sigma_2$.
\end{enumerate}
\end{definition}
In the present paper, we find an SR of $\mathrm{CNS}_0/{\sim}$
and $\mathrm{CNS}_1/{\sim}$, where
\[
	\mathrm{CNS}_j := \{ \sigma \in \mathrm{CNS}\ |\ \rank C(\sigma) = j \}
\]
for $j=0,1,2,3$.

One easily finds an SR of $\mathrm{CNS}_0/{\sim}$.
Remark that $\rank C(\sigma)=0$ implies $C(\sigma)=O$.
Hence, the classification of $\mathrm{CNS}_0$ boils down to the classification of a quadratic form $V_{p,q,r}$. 
By a standard invariant theory of quadratic forms, we have the following:
\begin{theorem}[System of representatives of $\mathrm{CNS}_0/{\sim}$]\label{thm:sysrank0}
The set
\[
	\{\mathrm{O}\} \times \{ (1,0,1),\, (-1,0,-1), \, (1,0,-1),\, (1,0,0),\, (-1,0,0),\, (0,0,1),\, (0,0,-1),\, (0,0,0) \}
\]
is a system of representatives of $\mathrm{CNS}_0/{\sim}$,
where $\mathrm{O}\in M_3(\R)$ is the zero matrix. 
\end{theorem}

Let us turn to the SR of $\mathrm{CNS}_1/{\sim}$.
We see from Theorem \ref{thm:equivalence} 
that the matrix part $C$ and the kernel part $(p,q,r)$ are transformed separately.
According to this structure, we specify one SR in the following two steps.
At the first step, we find an SR of the matrix part.
More precisely, we define an equivalent relation on $M_3(\R)$, which again denote by ${\sim}$, as follows: We say $M_3(\R) \ni C \sim C' \in M_3(\R)$ holds if
there exists $M \in GL_2(\R)$ such that the relation \eqref{eq:Cprime} is valid.
We will find a system $T_Z$ of representatives of $\mathcal{M}/{\sim}$, where
\[
	\mathcal{M} := \{ C \in M_3 (\R) \ |\ \rank C =1 \}
\]
and the notion of an SR is defined in a similar manner as in Definition \ref{def:SR}.
This is the first step.
Then, in the second step, we obtain a desired SR of
$\mathrm{CNS}_1/{\sim}$ as follows.
Let $T_Z$ be as above.
Given system $\sigma \in S$, there exists a unique represenatative $C' \in T_Z$ such that 
$C' \sim C(\sigma)$.
For each $C' \in T_Z $, we let $G(C') \subset GL_2(\R)$ a subgroup
which leaves $C'$ invariant.
Then, one obtains a set $K(\sigma) \subset \R^3$ by dividing $\R^3$ by the action given by matrices in $G'(C')$.
On sees that the set $K(\sigma)$ is the largest set of the kernel parts so that $C' \times K(C')$ 
does not contain two elements belong to the same equivalence class.
Thus, we obtain an SR of the form
\[
	\{ (C, p,q,r) \in M_3(\R) \times \R^3 \ |\ C \in T_Z,\, (p,q,r) \in K(C) \}.
\]
The result is as follows.

\begin{theorem}[System of representatives of $\mathrm{CNS}_1/{\sim}$]\label{thm:sysrank1}
The set $T:= \cup_{j=1}^9 (Z_j \times K_j) $ is a system of representatives of $\mathrm{CNS}_1/{\sim}$,
where
\begin{align*}
	Z_1 :={}& \left\{ \begin{pmatrix} 0 & 1  & 0 \\ 0 &k & 0 \\ 0 & 1 & 0  \end{pmatrix}  \ \middle|\ 
	k> 0 \right\} 
	 \cup \left\{ \begin{pmatrix} 0 & 1  & 0 \\ 0 &1 & 0 \\ 0 & 0 & 0  \end{pmatrix}, \begin{pmatrix} 0 & 0 & 0 \\ 0 &1 & 0 \\ 0 & 1 & 0  \end{pmatrix}, \begin{pmatrix} 0 & 0  & 1 \\ 0 &0 & 0 \\ 0 & 0 & 1  \end{pmatrix}, \begin{pmatrix} 0 & 0 & -1 \\ 0 &0 & 0 \\ 0 & 0 & 1  \end{pmatrix}  \right\} \\
	&{} \cup 	 \left\{ \begin{pmatrix} \sin \theta & 0 & \sin \theta \\ \cos \theta & 0 & \cos \theta \\ \sin \theta & 0 & \sin \theta  \end{pmatrix}  \ \middle|\  \theta \in [0, \pi/2) \right\}, \\
	Z_2 := {}&\left\{ \begin{pmatrix} 0 & \sigma  & 0 \\ 0 &k & 0 \\ 0 & - \sigma & 0  \end{pmatrix}  \ \middle|\ 
	\sigma \in \{\pm1\},\, k> 0 \right\}, \\
	Z_3 := {}& \left\{ \begin{pmatrix} 0 & \sigma  & 0 \\ 0 &0 & 0 \\ 0 & - \sigma & 0  \end{pmatrix}  \ \middle|\ 
	\sigma \in \{\pm1\} \right\}, \\
	Z_4 := {}&
	\left\{ \begin{pmatrix} 0 & 0  & 0 \\ 0 &0 & \sigma \\ 0 & 0 & 0  \end{pmatrix}  \ \middle|\ 
	\sigma \in \{\pm1\} \right\} \cup \left\{ \begin{pmatrix} 0 & 0  & 0 \\ 0 &0 & 0 \\ 0 & \sigma & 0  \end{pmatrix}  \ \middle|\ 
	\sigma \in \{\pm1\} \right\}, 
\end{align*}
\begin{align*}
	Z_5 := {}&	\left\{
	\begin{pmatrix} 0 & 1  & 0 \\ 0 &0 & 0 \\ 0 & 1 & 0  \end{pmatrix}\right\} ,&
	Z_6 := {}&\left\{
	\begin{pmatrix} 0 & 0  & 0 \\ 0 & 1 & 0 \\ 0 & 0 & 0  \end{pmatrix}\right\},&
	Z_7 := {}&\left\{
	\begin{pmatrix} 0 & 0  & 1 \\ 0 & 0 & 0 \\ 0 & 0 & 0  \end{pmatrix}\right\},\\
	Z_8 := {}&\left\{
	\begin{pmatrix} 0 & 0  & 0 \\ 0 & 0 & 0 \\ 0 & 0 & 1  \end{pmatrix}\right\},&
	Z_9 := {}&\left\{
	\begin{pmatrix} 1 & 0  & 1 \\ 0 & 0 & 0 \\ 1 & 0 & 1  \end{pmatrix}\right\}
\end{align*}
and $K_1 = \R^3$,
\begin{align*}
	K_2 :={}& \left\{ \ltrans{(p,q,r)} \in \R^3 \ \middle| \ p > r \right\}
	\cup \left\{ \ltrans{(p,q,r)} \in \R^3 \ \middle| \ p=r,\,q \ge 0  \right\}, \\
	K_3 :={}& \left\{ \ltrans{(p,q,r)} \in \R^3 \ \middle| \ p \ge r,\, q \ge 0 \right\}, \\
	K_4 :={}& \left\{ \ltrans{(p,q,1)} \in \R^3 \ \middle| \ q\ge  0 \right\}
	\cup \left\{ \ltrans{(p , \tilde{\sigma} , 0 )} \in \R^3 \ \middle| \ \tilde{\sigma} \in \{0,1\} \right\},\\
	K_5 :={}& \left\{ \ltrans{(p,q,r)} \in \R^3 \ \middle| \ q\ge  0 \right\}, \\
	K_6 :={}& \left\{ \ltrans{(p,q,r)} \in \R^3 \ \middle| \ |p|=|r|\ge 0,\,  q \ge 0 \right\}
	\cup \left\{ \ltrans{(p,q,0)} \in \R^3 \ \middle| \ |p|=1>0,\, q \ge 0 \right\},\\
	K_7 :={}& \left\{ \ltrans{(p,q,r)} \in \R^3 \ \middle| \ |p|=|r|>0 \right\}
	\cup \left\{ \ltrans{(p,q,0)} \in \R^3 \ \middle| \ |p|=|q|>0 \right\}\\
	&\cup \left\{ \ltrans{(0,q,r)} \in \R^3 \ \middle| \ |q|=|r|>0 \right\} 
	 \cup \left\{ \ltrans{(1,0,0)}, \ltrans{(0,1,0)}, \ltrans{(0,0,1)}, \ltrans{(0,0,0)} \right\},\\
	K_8 :={}& \left\{ \ltrans{(p,q,r)} \in \R^3 \ \middle| \ |p|=|r|\ge 0,\, q \ge 0 \right\}
	\cup \left\{ \ltrans{(p,q,0)} \in \R^3 \ \middle| \ |p|=1>0,\, q \ge 0 \right\}\\
	&\cup \left\{ \ltrans{(0,q,r)} \in \R^3 \ \middle| \ |r|=1>0,\, q \ge 0 \right\} , \\
	K_9 :={}& \left\{ \ltrans{(p,q,r)} \in \R^3 \ \middle| \ p \ge r,\, q = 0 \right\}.
\end{align*}
Here, a row vector in $\R^3$ is denoted as $(a,b,c)$, with commas.
\end{theorem}

\begin{remark}
One obtains the asymptotic behavior of solutions to sveral representatives by applying known results. 
More precisely, let $T$ be the SR of $\mathrm{CNS}_1/{\sim}$ given by the above theorem.
If the matrix part of a representative $\sigma \in T$ satisfies $C(\sigma) \in Z_+$, where
\begin{equation}\label{e:Z_+}
\begin{aligned}
	Z_+ :={}& 	\left\{ \begin{pmatrix} 0 & 1  & 0 \\ 0 &k & 0 \\ 0 & 1 & 0  \end{pmatrix}  \ \middle|\ 
	k\ge 0 \right\} \cup 
	\left\{ \begin{pmatrix} 0 & \sigma  & 0 \\ 0 &k & 0 \\ 0 & - \sigma & 0  \end{pmatrix}  \ \middle|\ 
	\sigma \in \{\pm1\},\, k\ge 0 \right\} \\
	{}& \cup 
	\left\{ \begin{pmatrix} 0 & 0  & 0 \\ 0 &0 & 0 \\ 0 & \sigma & 0  \end{pmatrix}  \ \middle|\ 
	\sigma \in \{\pm1\}\right\}
	\cup \left\{ \begin{pmatrix} 0 & 1 & 0 \\ 0 & 1 & 0 \\ 0 & 0 & 0  \end{pmatrix} ,\,
		\begin{pmatrix} 0 & 0 & 0 \\ 0 & 1 & 0 \\ 0 & 1 & 0  \end{pmatrix} ,\,
		\begin{pmatrix} 0 & 0 & 0 \\ 0 & 1 & 0 \\ 0 & 0 & 0  \end{pmatrix} 
		\right\},
\end{aligned}
\end{equation}
then \cite[Theorem 2.1]{KS} is applicable, in view of Proposition \ref{P:ODEconservation}.
Such an example is \cite[Example 6.2]{KS}, for instance.
\end{remark}

\begin{remark}
The case $\l_6=3\l_1 \in \{\pm1 \}$ of \eqref{E:sysnew1} belongs to $Z_4 \times K_4$,
and the system \eqref{E:d21} to $Z_7 \times K_7$. Notice that \eqref{E:sysnew1} belongs to $\mathrm{CNS}_2$, and hence is out of the range of 
Theorem \ref{thm:sysrank1}, if $\l_6 \neq 3 \l_1$.
\end{remark}

\subsection{Classification of matrices of rank one}

In the rest of this section, we prove Theorem \ref{thm:sysrank1}.
As mentioned above, what we do first is to specify $\mathcal{M}/{\sim} \subset M_3 (\R)/{\sim}$.

\begin{theorem}\label{thm:rank1}
The set
$T_Z:= Z_+ \cup Z_0 \cup Z_{-} \subset \mathcal{M}$ is a system of representatives of $\mathcal{M}/{\sim}$, where
$Z_+$ is given by \eqref{e:Z_+},
\[
	Z_0 := \left\{ \begin{pmatrix} 0 & 0 & \sigma \\ 0 & 0 & 0 \\ 0 & 0 & 1  \end{pmatrix}  \ \middle|\ 
	\sigma \in \{0,\pm1\}
		\right\} \cup 
		\left\{ \begin{pmatrix} 0 & 0 & 0 \\ 0 & 0 & 1 \\ 0 & 0 & 0  \end{pmatrix} ,\,
		\begin{pmatrix} 0 & 0 & 0 \\ 0 & 0 & -1 \\ 0 & 0 & 0  \end{pmatrix} ,\,
		\begin{pmatrix} 0 & 0 & 1 \\ 0 & 0 & 0 \\ 0 & 0 & 0  \end{pmatrix} 
		\right\},
\]
and
\[
	Z_{-} := \left\{ \begin{pmatrix} \sin \theta & 0 & \sin \theta \\ \cos \theta & 0 & \cos \theta \\ \sin \theta & 0 & \sin \theta  \end{pmatrix}  \ \middle|\ 
	 \theta \in [0, \pi/2] \right\} .
\]
Further $T_Z = \cup_{j=1}^9 Z_j$ holds true.
\end{theorem}

%\subsubsection{A constructive proof}

Here, we give a constructive proof, i.e., we find a matrix $M \in GL_2(\R)$
which transforms a given matrix into the corresponding representative.
This construction is valuable when we apply the theory to concrete examples.
This part is a generalization of a result in \cite{MSU}, where an SR of $(\mathcal{M} \cap \mathcal{Z})/{\sim}$ is given.
We give a slightly different proof.

For $C \in \mathcal{M}$, let
$\nu(C) \in \R^3$ be the unit vector such that $C$ is reduced into 
$\ltrans{\begin{pmatrix}
{\nu(C)} & {\bf 0} & {\bf 0}
\end{pmatrix}}$
by row operations, where ${\bf 0}\in \R^3$ denotes the zero column vector.
Remark that $\nu(C)$ is uniquely determined as an element of the projective space $\R \mathrm{P}^2$. 

\begin{lemma}\label{L:dnuC}
Let $C \in M_3(\R)$.
$\rank C =1 $ if and only if $C$ is written as
$
	C = {\bf d} \ltrans{\nu(C)}
$
with some nonzero vector ${\bf d} \neq0$.
\end{lemma}

Let $\mathcal{B}$ be the unit sphere of $\R^3$. 
Let $\Gamma_1, \Gamma_2 \subset \mathcal{B}$ be simple closed smooth curves given by
\begin{equation}\label{E:defC}
	\begin{aligned}
	\Gamma_1:={}& \{ \ltrans{(a,b,c)} \in \mathcal{B} \ | \ b^2 = ac ,\, a\ge0,\, c\ge 0\}\\
	={}&\left\{ \ltrans{\( \tfrac{\sqrt2(1+\sin \theta)}{\sqrt{7-\cos 2\theta}} , 
	\tfrac{\sqrt2 \cos \theta}{\sqrt{7-\cos 2\theta}},
	\tfrac{\sqrt2 (1-\sin \theta)}{\sqrt{7-\cos2\theta}}\)} \in \mathcal{B} \ \middle| \ 0 \le \theta \le 2\pi \right\}
	\end{aligned}
\end{equation}
and 
\begin{equation}\label{E:defD}
	\begin{aligned}
	\Gamma_2 :={}& \{ \ltrans{(a,b,c)} \in \mathcal{B} \ | \ b^2 = 4ac ,\, a\ge0,\, c\ge 0\}\\
	={}&
	\left\{ \ltrans{\( \tfrac{1-\sin \theta}{\sqrt{5+\cos 2\theta}} , 
	-\tfrac{2 \cos \theta}{\sqrt{5+\cos 2\theta}},
	\tfrac{ 1+\sin \theta}{\sqrt{5+\cos2\theta}}\)}\in \mathcal{B} \ \middle| \ 0 \le \theta \le 2\pi \right\},
	\end{aligned}
\end{equation}
respectively.
Further, let $\mathscr{S}_1, \mathscr{S}_2\subset \mathcal{B}$ be open connected surfaces given by
\begin{equation}\label{E:defS12}
	\mathscr{S}_1 := \{ \ltrans{(a,b,c)} \in \mathcal{B} \ | \ b^2 > 4ac \}
\quad\text{and}
\quad
	\mathscr{S}_2 := \{ \ltrans{(a,b,c)} \in \mathcal{B} \ | \ b^2 < 4ac ,\, a>0 \},
\end{equation}
respectively. Remark that
\begin{equation}\label{E:cBdecomp}
	\mathcal{B}  = \mathscr{S}_2 \cup \Gamma_2 \cup \mathscr{S}_1 \cup (-\Gamma_2) \cup (-\mathscr{S}_2).
\end{equation}
Further, $\partial \mathscr{S}_1=\Gamma_2 \sqcup (-\Gamma_2)$ and $\partial \mathscr{S}_2=\Gamma_2$.
We define three subsets
\begin{align*}
	\mathcal{M}_{+} :={}& \{ C \in \mathcal{M} \ |\ \nu (C) \in \mathscr{S}_1 \},\\
	\mathcal{M}_{0} :={}& \{ C \in \mathcal{M} \ |\ \nu (C) \in \Gamma_2 \sqcup (-\Gamma_2) \},\\
	\mathcal{M}_{-} :={}& \{ C \in \mathcal{M} \ |\ \nu (C) \in \mathscr{S}_2 \sqcup (-\mathscr{S}_2) \}.
\end{align*}
By \eqref{E:cBdecomp}, we have
$\mathcal{M} =  \mathcal{M}_{+} \cup  \mathcal{M}_{0} \cup  \mathcal{M}_{-}$.

Let us now an intermediate sets.
\begin{align*}
	Y_{+} :={}& \{ C \in \mathcal{M}_{+} \ |\ \nu (C) =  \ltrans{(0,1,0)} \},\\
	Y_{0} :={}& \{ C \in \mathcal{M}_{0} \ |\ \nu (C) =  \ltrans{(0,0,1)}  \},\\
	Y_{-} :={}& \{ C \in \mathcal{M}_{-} \ |\ \nu (C) =  2^{-1/2}\ltrans{(1,0,1)} \}.
\end{align*}
Our reduction is divided into two steps.
\begin{proposition}[First reduction]\label{Prop:first}
Let $\varsigma \in \{ 0, \pm \}$.
For any $C \in \mathcal{M}_{\varsigma}$, there exists $\tilde{C} \in Y_{\varsigma}$ such that $C \sim \tilde{C}$.
\end{proposition}

\begin{proof}[{\bf Proof of Proposition \ref{Prop:first}}]
First consider the case $\varsigma=+$.
Pick $C_+ \in \mathcal{M}_{+}$. Since $\nu(C_+) \in \mathscr{S}_1$,
the trichotomy in \cite{MSU} implies that $\ker C_+$ has two intersection points with $\Gamma_1$.
We name them as
\[
	{\bf p}_j = r(\theta_j) \begin{pmatrix}1+\sin \theta_j\\ \cos \theta_j\\ 1-\sin \theta_j \end{pmatrix}\qquad (j=1,2)
\]
with the normalizing factor $r(\theta) %=\sqrt2 / \sqrt{7-\cos 2\theta}
$ and $\theta_1,\theta_2 \in [0,2\pi)$ such that $\theta_1\neq \theta_2$.
Remark that (one of) $\nu(C_+)\in \mathcal{B}$ is written as
\[
	\nu (C_+) = c r(\theta_1)^{-1}r(\theta_2)^{-1}({\bf p}_1 \times {\bf p}_2) = c
	\begin{pmatrix} \cos \theta_1(1-\sin \theta_2)-\cos \theta_2(1-\sin \theta_1)\\
	-2 (\sin \theta_1 -  \sin \theta_2)\\
	-\cos \theta_1 (1+\sin \theta_2)+\cos \theta_2 (1+\sin \theta_1)
	\end{pmatrix}
\]
with a suitable normalizing factor $c=c(\theta_1,\theta_2)\neq0$.
Let 
\[
	M_+ = \begin{pmatrix}
	\cos \eta_1 & \sin \eta_1 \\
	\cos \eta_2 & \sin \eta_2 
	\end{pmatrix}, \qquad
	\eta_j = -\tfrac{\theta_j}2 + \tfrac{\pi}4.
\]
A computation shows that
\begin{align*}
	D(M_+)^{-1} 
	&{}= \frac1{2\sin \frac{\theta_1-\theta_2}2}
	\begin{pmatrix}
	1+\sin \theta_1 &  2(\sin \frac{\theta_1+\theta_2}2 +\cos \frac{\theta_1-\theta_2}2) & 1+\sin \theta_2 \\
	\cos \theta_1 & 2 \cos \frac{\theta_1+\theta_2}2 & \cos \theta_2 \\
	1-\sin \theta_1  & 2(-\sin \frac{\theta_1+\theta_2}2 + \cos \frac{\theta_1-\theta_2}2 ) & 1-\sin \theta_2 
	\end{pmatrix}.
\end{align*}
Note that
\[
	\ltrans{D(M_+)^{-1}} \nu(C_+)
	= - 4 c (1- \cos^2 \tfrac{\theta_1-\theta_2}{2}) \begin{pmatrix}
	0 \\ 1\\ 0
	\end{pmatrix}.
\]
This shows that $\tilde{C}_+:= (\det M_+)^{-1} D(M_+) C_+ D(M_+)^{-1}$ satisfies
$\nu (\tilde{C}_+) = \ltrans{(0,1,0)}$, and hence $\tilde{C}_+ \in Y_+$.

Next consider the case $\varsigma=0$.
Pick $C_0 \in \mathcal{M}_{0}$. Since $\nu(C_0) \in \Gamma_2 \cup (-\Gamma_2)$,
we deduce from the trichotomy that there exists $\theta \in \R/2\pi \Z$ such that
\[
	\nu (C_0) =c'
	\begin{pmatrix}
		1- \sin \theta \\ -2 \cos \theta \\ 1+ \sin \theta
	\end{pmatrix}.
\]
We define
\[
	M_0 = \begin{pmatrix}
	\cos \eta & \sin \eta \\
	c& d 
	\end{pmatrix} \in GL_2 , \qquad
	\eta = -\tfrac{\theta}2 + \tfrac{\pi}4.
\]
One has
\begin{align*}
	D(M_0)^{-1}
%	 &{}= \frac1{\det M_0}
%	\begin{pmatrix}
%	\cos^2 \eta & 2 c \cos \eta  & c^2 \\
%	\sin \eta \cos \eta & d \cos \eta + c \sin \eta& cd \\
%	\sin^2 \eta  & 2 d \sin \eta & d^2 
%	\end{pmatrix}\\
	&{}=\frac1{2\det M_0}
	\begin{pmatrix}
	1+\sin \theta & 4 c \cos \eta  & 2c^2 \\
	\cos \theta & 2(d \cos \eta + c \sin \eta)& 2cd \\
	1-\sin \theta & 4 d \sin \eta & 2d^2 
	\end{pmatrix}.
\end{align*}
Hence, one has
\[
	\ltrans{D(M_0)^{-1}} \nu(C_0)
	= 2c'\det M_0 \begin{pmatrix}
		0 \\ 0\\ 1
	\end{pmatrix}.
\]
This shows that $\tilde{C}_0:= (\det M_0)^{-1} D(M_0) C_0 D(M_0)^{-1}$ satisfies
$\nu (\tilde{C}_0) = \ltrans{(0,0,1)}$, and hence $\tilde{C}_0 \in Y_0$.

Finally, let us consider the case $\varsigma=-$.
Pick $C_{-} \in \mathcal{M}_{-}$. Since $\nu(C_{-})=\ltrans{(u,v,w)} \in \mathscr{S}_2 \cup (-\mathscr{S}_2)$, we have $v^2 <4 uw $.
Take
\[
		 M_{-}=\begin{pmatrix} 2w & -v \\ 0 &  \sqrt{4uw-v^2} \end{pmatrix} \in GL_2.
\]
Remark that 
\begin{align*}
	D(M_{-})^{-1} &{}= \frac1{\det M_{-}}
	\begin{pmatrix}
	4 w^2 & 0  &0  \\
	-2wv & 2w  \sqrt{4uw-v^2} & 0 \\
	v^2 & -2 v  \sqrt{4uw-v^2} & 4uw-v^2
	\end{pmatrix}.
\end{align*}
Hence,
\[
	\ltrans{D(M_{-})^{-1}} \nu(C_{-})
	=\frac{ \sqrt{4wu - v^2} }{2} \begin{pmatrix}
		1 \\ 0\\ 1
	\end{pmatrix}.
\]
This shows that $\tilde{C}_{-}:= (\det M_{-})^{-1} D(M_{-}) C_{-} D(M_{-})^{-1}$ satisfies
$\nu (\tilde{C}_{-}) =2^{-1/2} \ltrans{(1,0,1)}$, and hence $\tilde{C}_{-} \in Y_{-1}$.
\end{proof}

\begin{proposition}[Second reduction]\label{prop:2ndred}
Let $\varsigma \in \{ 0, \pm  \}$.
For any $C \in Y_{\varsigma}$, there exists $\tilde{C} \in Z_{\varsigma}$ such that $C \sim \tilde{C}$.
\end{proposition}
To show the proposition, we use the following:
\begin{lemma}\label{lem:redstep2}
Let $M \in GL_2(\R)$. Let $\varsigma \in \{ 0,\pm \}$.
Let $C,C' \in Y_\varsigma$.
Then, $C\sim C'$ is true if and only if \eqref{eq:Cprime} holds with the following matrices:
\begin{itemize}
\item
If $\varsigma = +$ then
\[
	M = \begin{pmatrix}
	0& 1 \\ 1 & 0
	\end{pmatrix}^\ell
	\begin{pmatrix}
	p& 0 \\ 0 & s
	\end{pmatrix}
\]
for $p,s \in \R\setminus\{0\}$ and $\ell \in \{0,1\}$;
\item If $\varsigma=0$ then 
\[
	M= \begin{pmatrix}
	p & 0 \\ r & s
	\end{pmatrix}
\]
for $p,s \in \R\setminus\{0\}$ and $r\in \R$;
\item If $\varsigma=-$ then 
\[
	M= R \begin{pmatrix}
	1 & 0 \\ 0 &  -1
	\end{pmatrix}^\ell
\begin{pmatrix}
	\cos \eta & -\sin \eta \\ \sin \eta &  \cos \eta
	\end{pmatrix}
\]
for some $R>0$ and $\eta \in \R/ 2\pi\Z$ and $\ell \in \{0,1\}$.
\end{itemize}
\end{lemma}

\begin{proof}[{\bf Proof of Lemma \ref{lem:redstep2}}]
We prove the only if part.
Suppose that $C \sim C'$. Then, \eqref{eq:Cprime} is valid for some 
\[
	M=\begin{pmatrix}
	p & q \\ r & s
	\end{pmatrix} \in GL_2(\R).
\]

Let us first consider the case $\varsigma=+$.
Since $\nu(C)=\nu(C') = \ltrans{(0,1,0)}$, we see from \eqref{eq:Cprime} that $\ltrans{(0,1,0)}$ must be an eigenvector of 
\[
	\ltrans{(D(M)^{-1})} = \frac1{\det M} 
	\begin{pmatrix}
	p^2 & pq & q^2 \\ 
	2pr&ps+qr &2qs \\
	r^2 & rs & s^2
	\end{pmatrix}
\]
associated with a nonzero eigenvalue.
To this end, we need
$pq = rs =0$ and $ps + qr \neq0$.
If $p=0$ then we have $qr\neq0$ and hence $s=0$.
Similarly, if $q=0$ then we have $ps \neq0$ and hence $r=0$.

Let us next consider the case $\varsigma=0$.
By the same argument as in the case $\sigma=1$,
one sees from $\nu(C)=\nu(C') =\ltrans{ (0,0,1)}$ that $\ltrans{(0,0,1)}$ is an eigenvector of $\ltrans{(D(M)^{-1})}$.
From this, we obtain $q =0$ and  $s \neq0$.

Finally, we consider the case $\varsigma=-$.
In this case, $\ltrans{(1,0,1)}$ is  an eigenvector of $\ltrans{(D(M)^{-1})}$.
To this end, we need
$p^2 + q^2 = r^2 + s^2 >0$ and  $pr+qs=0$.
One sees from the first equation that
there exists $R>0$, $\eta$, and $\zeta$ such that
\[
	p=R \cos \eta , \quad q= -R \sin \eta ,\quad r =R \cos \zeta, \quad s=-R \sin \zeta.
\]
Plugging these formulas to the second identity, we obtain $\cos (\eta - \zeta) = 0$, which shows
$\zeta = \eta + \frac\pi2 + \ell \pi$ ($\ell \in \{0,1\}$).
Thus, we obtain the result.

The  if part is obvious from the argument above.
\end{proof}

\begin{proof}[{\bf Proof of Proposition \ref{prop:2ndred}}]
We first consider the case $\varsigma=+$.
Pick $C \in Y_+$. Then, $C$ is written as
\[
	C = \begin{pmatrix} 0 & k_1 & 0 \\ 0 & k_2 & 0 \\ 0 & k_3 & 0  \end{pmatrix} = 
	\begin{pmatrix}  k_1 \\  k_2  \\  k_3  \end{pmatrix}
	\begin{pmatrix}  0 & 1 & 0  \end{pmatrix}
\]
with a nonzero vector $\ltrans{(k_1,k_2,k_3)} \in \R^3$.
By Lemma \ref{lem:redstep2}, $C' \in Y_+$ satisfies $C' \sim C$ if and only if there exist nonzero $p,s$ and $\ell \in \{0,1\}$ such that
\[
	C' =	\begin{pmatrix}  p^{-2} k_1 \\  (ps)^{-1} k_2  \\ s^{-2}  k_3  \end{pmatrix}
	\begin{pmatrix}  0 & 1 & 0  \end{pmatrix}
\]
if $\ell=0$ and 
\[
	C' =- 	\begin{pmatrix} s^{-2}  k_3 \\ (ps)^{-1} k_2  \\ p^{-2} k_1  \end{pmatrix}
	\begin{pmatrix}  0 & 1 & 0  \end{pmatrix}
\]
if $\ell=1$.
Let us show that $C' \in Z_+$ holds for a suitable choice of $p,r$ and $\ell$.

When $k_1>0$ and $k_3 >0$ then we choose $\ell=0$, $p= \sqrt{k_1}$, and
$s=\sqrt{k_3}$ if  $k_2 \ge 0$ and $s= -\sqrt{k_3}$ if $k_2 < 0$.
Then, one sees that
\[
	C' =	\begin{pmatrix}  1 \\    |k_2| / \sqrt{k_1k_3}  \\ 1  \end{pmatrix}
	\begin{pmatrix}  0 & 1 & 0  \end{pmatrix} \in Z_+.
\]
The case $k_1<0$ and $k_3<0$ is handled similarly, we take $\ell=1$, $p= \sqrt{|k_1|}$, and
$s=-\sqrt{|k_3|}$ if  $k_2 \ge 0$ and $s= \sqrt{|k_3|}$ if $k_2 < 0$.
The choice gives us the same reduced matrix.
If $k_1>0>k_3$ then one can choose parameters so that
\[
	C' =	\begin{pmatrix}  1 \\   |k_2|/\sqrt{k_1|k_3|}  \\ -1  \end{pmatrix}
	\begin{pmatrix}  0 & 1 & 0  \end{pmatrix} \in Z_+.
\]
If $k_3>0>k_1$ then we have
\[
	C' =	\begin{pmatrix} - 1 \\   |k_2|/\sqrt{|k_1|k_3}  \\ 1  \end{pmatrix}
	\begin{pmatrix}  0 & 1 & 0  \end{pmatrix} \in Z_+.
\]
for a suitable choice of the parameters. We left the precise choice for the readers.

If $k_1 > 0=k_3$ and $k_2\neq0$ then we take $\ell=0$,
$p= \sqrt{k_1}$ and $s = p^{-1}k_2$
 to get
\[
	C' =\begin{pmatrix} 1 \\  1  \\ 0  \end{pmatrix}
	\begin{pmatrix}  0 & 1 & 0  \end{pmatrix} \in Z_+.
\] 
If $k_1 > 0=k_2=k_3$ then we take $\ell=1$,
$p= \sqrt{k_1}$ and $s=1$
 to get
\[
	C' =\begin{pmatrix} 0 \\  0  \\ -1  \end{pmatrix}
	\begin{pmatrix}  0 & 1 & 0  \end{pmatrix} \in Z_+.
\] 
If $k_1<0=k_3$ then  we let $\ell=1$,
$p = \sqrt{|k_1|}$ and $s = -p^{-1}k_2$ if $k_2\neq0$ and $s=1$ if $k_2=0$
 to get
\[
	C' =\begin{pmatrix} 0 \\  1  \\ 1  \end{pmatrix}
	\begin{pmatrix}  0 & 1 & 0  \end{pmatrix} \in Z_+,\quad \text{ or }\quad
	C' =\begin{pmatrix} 0 \\  0  \\ 1  \end{pmatrix}
	\begin{pmatrix}  0 & 1 & 0  \end{pmatrix} \in Z_+.
\] 
The case $k_1=0 \neq k_3$ are handled similarly.
Finally, in the case $k_1=k_3=0 \neq k_2$, we choose $\ell=0$, $p=1$, and $s=k_2$ to get
\[
	C' =\begin{pmatrix} 0 \\  1  \\ 0  \end{pmatrix}
	\begin{pmatrix}  0 & 1 & 0  \end{pmatrix} \in Z_+.
\]
Hence, we have the results for $\varsigma=+$.

Let us turn to the case $\varsigma=0$.
Pick $C \in Y_0$. Then, $C$ is written as
\[
	C = \begin{pmatrix} 0 &0 & k_1  \\ 0 &0 & k_2  \\ 0 & 0 &k_3   \end{pmatrix} = 
	\begin{pmatrix}  k_1 \\  k_2  \\  k_3  \end{pmatrix}
	\begin{pmatrix}  0 & 0 & 1  \end{pmatrix}
\]
with a nonzero vector $\ltrans{(k_1,k_2,k_3)} \in \R^3$.

By Lemma \ref{lem:redstep2}, if $C' \in Y_0$ satisfies $C' \sim C$ then $C'$ is written as
\[
	C' = \frac1{(ps)^3} 	\begin{pmatrix} s^2 k_1 - 2 rs k_2 + r^2 k_3 \\ p(s  k_2 - r k_3)  \\ p^2 k_3  \end{pmatrix}
	\begin{pmatrix}  0 & 0 & s^2  \end{pmatrix}
\]
for $ ps\neq0$ and $r\in \R$.
When $k_3\neq0$, we choose $p=k_3/s$ and $r=sk_2/k_3$.
Then,
\[
	C' =
	\begin{pmatrix} \frac{s^4}{k_3^3} (k_1 -  \frac{k_2^2}{k_3}) \\ 0  \\ 1  \end{pmatrix}
	\begin{pmatrix}  0 & 0 & 1  \end{pmatrix}.
\]
One can choose $s >0$ so that $\frac{s^4}{k_3^3} (k_1 -  \frac{k_2^2}{k_3}) \in \{0, \pm 1\}$.
For this choice, we have $C' \in Z_0$.
When $k_3=0$ and $k_2\neq0$, we choose $p=|k_2|^{1/2}$ and $r=sk_1/2k_2$.
Then
\[
	C' =
	\begin{pmatrix} 0 \\ \sign k_2 \\  0 \end{pmatrix}
	\begin{pmatrix}  0 & 0 & 1  \end{pmatrix} \in Z_0.
\]
When $k_3=k_2=0$ and $k_1\neq0$, we choose $p=1$ and $s=1/k_1$ to obtain
\[
	C' =
	\begin{pmatrix} 1 \\ 0 \\  0 \end{pmatrix}
	\begin{pmatrix}  0 & 0 & 1  \end{pmatrix} \in Z_0.
\]
For all cases, we reach to the desired conclusion.

Finally consider the case $\varsigma=-$.
Pick $C \in Y_{-}$. Then, $C$ is written as
\[
	C = \begin{pmatrix} k_1 &0 & k_1  \\ k_2 &0 & k_2  \\ k_3 & 0 &k_3   \end{pmatrix} = 
	\begin{pmatrix}  k_1 \\  k_2  \\  k_3  \end{pmatrix}
	\begin{pmatrix}  1 & 0 & 1  \end{pmatrix}
\]
with a nonzero vector $\ltrans{(k_1,k_2,k_3)} \in \R^3$.

By Lemma \ref{lem:redstep2}, $C' \in Y_-$ satisfies $C' \sim C$ if and only if
there exists  $ R>0$, $\eta \in \R/2\pi \Z$, and $\ell\in \{0,1\}$ such that
\[
	C' = \frac1{2R^2} 	\begin{pmatrix} (k_1+k_3) + [(k_1-k_3)\cos 2\eta - 2k_2 \sin 2\eta]  \\ (k_1-k_3)\sin 2\eta + 2k_2 \cos 2\eta  \\ (k_1+k_3) - [(k_1-k_3)\cos 2\eta - 2k_2 \sin 2\eta]  \end{pmatrix}
	\begin{pmatrix}  1 & 0 & 1  \end{pmatrix}
\]
if $\ell=0$ and
\[
	C' = \frac1{2R^2} 	\begin{pmatrix} -(k_1+k_3) + [-(k_1-k_3)\cos 2\eta - 2k_2 \sin 2\eta]  \\ -(k_1-k_3)\sin 2\eta + 2k_2 \cos 2\eta  \\ -(k_1+k_3) - [-(k_1-k_3)\cos 2\eta - 2k_2 \sin 2\eta]  \end{pmatrix}
	\begin{pmatrix}  1 & 0 & 1  \end{pmatrix}
\]
if $\ell=1$. 

We let $\ell=0$ if $k_1+k_3 \ge 0$ and $\ell=1$ if $k_1+k_3<0$.
Choose $\eta$ so that $(k_1-k_3)\cos 2\eta - 2(-1)^\ell k_2 \sin 2\eta=0$
and
\[
	(-1)^\ell (k_1-k_3)\sin 2\eta + 2k_2 \cos 2\eta
	= \sqrt{(k_1-k_3)^2 + 4k_2^2}
\]
hold. 
Then, we have
\[
	C' = \frac1{2R^2} 	\begin{pmatrix} |k_1+k_3| \\ \sqrt{(k_1-k_3)^2 + 4k_2^2} \\ |k_1+k_3| \end{pmatrix}
	\begin{pmatrix}  1 & 0 & 1  \end{pmatrix}.
\]
Now we choose $R>0$ so that $(2R^2)^2 = (k_1+k_3)^2 + (k_1-k_3)^2 + 4 k_2^2 >0$ to obtain
\[
	C' =
	\begin{pmatrix} y_1  \\ y_2 \\ y_1 \end{pmatrix}
	\begin{pmatrix}  1 & 0 & 1  \end{pmatrix}
\]
with $y_1,y_2\ge0$ and $y_1^2+y_2^2 =1$. This shows $C' \in Z_{-}$.
Thus, we reach to the desired conclusion in all cases.
\end{proof}

Now we are in a position to complete the proof of Theorem \ref{thm:rank1}.

\begin{proof}[{\bf Proof of Theorem \ref{thm:rank1}}]
We have shown that $\mathcal{M}/{\sim}=[Z_+] \cup [Z_0] \cup [Z_{-}]$.
Hence, it suffices to show that if $C,C' \in Z_+ \cup Z_0 \cup Z_{-}$ satisfies $C \sim C'$
then $C=C'$.
Suppose that $C,C' \in Z_+ \cup Z_0 \cup Z_{-}$ satisfies $C \sim C'$.
First note that if $C \in Z_\varsigma$ for $\varsigma \in \{0,\pm\}$ then $C' \in Z_ \varsigma$ with the same $\varsigma$.

Suppose that $\varsigma=1$. By Lemma \ref{lem:redstep2}, $C'$ is given by the formula \eqref{eq:Cprime} with
\[
	M = \begin{pmatrix}
	0& 1 \\ 1 & 0
	\end{pmatrix}^\ell
	\begin{pmatrix}
	p& 0 \\ 0 & s
	\end{pmatrix}
\]
for $p,s \in \R\setminus\{0\}$ and $\ell \in \{0,1\}$.
 If 
\[
	C = 
	\begin{pmatrix}  1 \\  k  \\  1  \end{pmatrix}
	\begin{pmatrix}  0 & 0 & 1  \end{pmatrix}
\]
for some $k \ge 0$ then one has
\[
	C' =	\begin{pmatrix}  p^{-2}  \\  (ps)^{-1} k  \\ s^{-2}    \end{pmatrix}
	\begin{pmatrix}  0 & 1 & 0  \end{pmatrix}
\]
if $\ell=0$ and 
\[
	C' =- 	\begin{pmatrix} s^{-2}   \\ (ps)^{-1} k  \\ p^{-2}   \end{pmatrix}
	\begin{pmatrix}  0 & 1 & 0  \end{pmatrix}
\]
if $\ell=1$.
One then sees from $C' \in Z_1$ that $\ell=0$ and $p^2=s^2=1$. Further, since $k\ge0$, we have $ps = 1$.
This shows $C'=C$.
 If 
\[
	C = 
	\begin{pmatrix}  \tilde{\sigma} \\  k  \\  -\tilde{\sigma}  \end{pmatrix}
	\begin{pmatrix}  0 & 0 & 1  \end{pmatrix}
\]
for some $\tilde\sigma \in \{\pm1\}$ and $k \ge 0$ then one sees that $p^2=s^2=1$ from the first and the third row of $C'$.
Then, looking at the second row, we have $ps= (-1)^\ell$. This shows $C=C'$. The other cases are handled similarly.

Let us proceed to the case $\varsigma=0$. Suppose that
\[
	C = 
	\begin{pmatrix}  \tilde{\sigma} \\  0  \\  1  \end{pmatrix}
	\begin{pmatrix}  0 & 0 & 1  \end{pmatrix}
\]
with $\tilde{\sigma} \in \{0, \pm1\}$. Then, any $C'\in [C]$ is written as
\[
	C' = 	\begin{pmatrix} \frac{s}{p^3} \tilde{\sigma}  + \frac{r^2}{p^3s}  \\  - \frac{r}{p^2s}   \\ \frac1{ps}   \end{pmatrix}
	\begin{pmatrix}  0 & 0 & 1  \end{pmatrix}.
\]
Suppose that $C' \in Z_0$.
By noticing that the third row of $C'$ is nonzero, we see that
\[
	C = 
	\begin{pmatrix}  \tilde{\sigma}' \\  0  \\  1  \end{pmatrix}
	\begin{pmatrix}  0 & 0 & 1  \end{pmatrix}
\]
with $\tilde{\sigma}' \in \{0, \pm1\}$. Let us show $\tilde{\sigma}'=\tilde{\sigma}$.
Comparing the two representations, one sees that $ps=1$ and $r=0$ are necessary. With the choice,
the former representation reads as
\[
	C' = 	\begin{pmatrix} \frac{s^2}{p^2} \tilde{\sigma}    \\  0   \\ 1   \end{pmatrix}
	\begin{pmatrix}  0 & 0 & 1  \end{pmatrix}.
\]
Since $\frac{s^2}{p^2} \tilde{\sigma}  = \tilde{\sigma}'$, we conclude that $\tilde{\sigma}=\tilde{\sigma}'$.
The other cases are handled similarly.

The case $\varsigma=-$.
Suppose that
\[
	C = 
	\begin{pmatrix}  \sin \theta \\  \cos \theta  \\  \sin \theta  \end{pmatrix}
	\begin{pmatrix}  1 & 0 & 1  \end{pmatrix}
\]
with $\theta \in [0,\pi/2]$. Then, any $C'\in [C]$ is written as
\[
	C' = 	\frac1{R^2}\begin{pmatrix} \sin \theta -\cos \theta \sin 2\eta  \\   \cos \theta \cos 2\eta   \\ \sin \theta + \cos \theta \sin 2 \eta   \end{pmatrix}
	\begin{pmatrix}  0 & 0 & 1  \end{pmatrix}
\]
or
\[
	C' = 	\frac1{R^2}	\begin{pmatrix} -\sin \theta-\cos \theta \sin 2\eta  \\   \cos \theta \cos 2\eta   \\ -\sin \theta+ \cos \theta \sin 2\eta   \end{pmatrix}
	\begin{pmatrix}  0 & 0 & 1  \end{pmatrix}
\]
It is easy to see that $C' \in Z_{-}$ implies that $C'$ is of the first form and $\eta=0$ and $R=1$.
This shows $C'=C$.
\end{proof}

\subsection{Reduction of the kernel part}

We next specify the possible kernel part for each element of $C \in T_Z$.
To this end, for a given $C \in T_Z$, we find a subgroup of $GL_2(\R)$
which leaves $C$ invariant.

\begin{proposition}\label{Prop:G}
Let $C \in Z_j$ for some $1\le j\le9$. Let $M \in GL_2(\R)$ and let $C'$ be the matrix defined by \eqref{eq:Cprime}.
If $C' =C $ then $M$ belongs to $G_j $, where $G_1 := \{ \pm E_2 \} $,
\begin{align*}
	G_2 :={}& \left\{ \pm E_2 , \begin{pmatrix} 0 & 1   \\ -1 & 0   \end{pmatrix} , \begin{pmatrix} 0 &- 1   \\ 1 & 0   \end{pmatrix} \right\}, \\
	G_3 :={}&  
	\left\{ \begin{pmatrix}  \sigma_1 &0  \\  0 & \sigma_2  \end{pmatrix}  \ \middle|\ 
	\sigma_1,\sigma_2 \in \{\pm1\} \right\} \cup
	\left\{ \begin{pmatrix} 0 & \sigma_1   \\   \sigma_2 & 0  \end{pmatrix}  \ \middle|\ 
	\sigma_1,\sigma_2 \in \{\pm1\} \right\} , \\
	G_4 :={}& 
	\left\{ \begin{pmatrix} \sigma & 0   \\   0 & s  \end{pmatrix}  \ \middle|\ 
	s\neq 0 ,\, \sigma \in \{\pm1\} \right\}  , \\
	G_5 :={}& \left\{ \begin{pmatrix} \sigma_1 & 0   \\   0 & \sigma_2  \end{pmatrix}  \ \middle|\ 
	\sigma_1, \sigma_2 \in \{\pm1\} \right\} , \\
	G_6 :={}& \left\{ \begin{pmatrix} p & 0   \\   0 & 1/p  \end{pmatrix}  \ \middle|\ 
	p\neq 0 \right\} \cup 	\left\{ \begin{pmatrix} 0 & p    \\   - 1/p & 0 \end{pmatrix}  \ \middle|\ 
	p\neq 0 \right\}  , \\
	G_7 :={}&  \left\{ \begin{pmatrix} p & 0    \\   0 & p^{3} \end{pmatrix}  \ \middle|\ 
	p\neq 0 \right\} , \\
	G_8 :={}& 	\left\{ \begin{pmatrix} p & 0    \\   0 & 1/p \end{pmatrix}  \ \middle|\ 
	p\neq 0 \right\}  , \\
	G_9 :={}&  \left\{ \begin{pmatrix} \cos \eta & - \sin \eta  \\  \sin \eta & \cos \eta \end{pmatrix}  \ \middle|\ 
	\eta \in \R/2\pi\Z \right\}.
\end{align*}	
\end{proposition}

\begin{proof}[{\bf Proof of Proposition \ref{Prop:G}}]
Let $C \in T$. Let $M \in GL_2(\R)$ and let $C'$ be the matrix defined by \eqref{eq:Cprime}.
Suppose that $C'=C$.

Let us begin with the case $C \in Z_+$. 
Then, $C$ is of the form
\[
	C = \begin{pmatrix} 0 & k_1 & 0 \\ 0 & k_2 & 0 \\ 0 & k_3 & 0  \end{pmatrix} = 
	\begin{pmatrix}  k_1 \\  k_2  \\  k_3  \end{pmatrix}
	\begin{pmatrix}  0 & 1 & 0  \end{pmatrix}
\]
with a nonzero vector $\ltrans{(k_1,k_2,k_3)} \in \R^3$.
By Lemma \ref{lem:redstep2}, $M$ is of the form
\[
	M = \begin{pmatrix}
	0& 1 \\ 1 & 0
	\end{pmatrix}^\ell
	\begin{pmatrix}
	p& 0 \\ 0 & s
	\end{pmatrix}
\]
for $p,s \in \R\setminus\{0\}$ and $\ell \in \{0,1\}$. Remark that
\[
	C' =	\begin{pmatrix}  p^{-2} k_1 \\  (ps)^{-1} k_2  \\ s^{-2}  k_3  \end{pmatrix}
	\begin{pmatrix}  0 & 1 & 0  \end{pmatrix}
\]
if $\ell=0$ and 
\[
	C' =- 	\begin{pmatrix} s^{-2}  k_3 \\ (ps)^{-1} k_2  \\ p^{-2} k_1  \end{pmatrix}
	\begin{pmatrix}  0 & 1 & 0  \end{pmatrix}
\]
if $\ell=1$.

If $k_1=k_3=1$ and $k_2 \ge 0$
%\[
%	C \in \left\{ \begin{pmatrix} 0 & 1  & 0 \\ 0 &k & 0 \\ 0 & 1 & 0  \end{pmatrix}  \ \middle|\ 
%	k\ge 0 \right\} 
%\]
then we see from the first row and the third row of $C'$ that $\ell=0$, $p^2=s^2=1$.
When $k_2>0$. we also have $ps=1$ by the second row. Hence, $(p,s)=\pm (1,1)$.
When $k_2=0$ there is no additional condition. We obtain the result.

If $k_1=-k_3 \in \{ \pm 1 \}$ and $k_2 \ge 0$
%\[
%	C \in \left\{ \begin{pmatrix} 0 & \sigma  & 0 \\ 0 &k & 0 \\ 0 & -\sigma & 0  \end{pmatrix}  \ \middle|\ \sigma \in \{ \pm 1\},\,
%	k\ge 0 \right\} 
%\]
then we see from the first row and the third row of $C'$ that $p^2=s^2=1$.
When $k_2>0$. we also have $ps=(-1)^\ell $ by the second row. %Hence, $(p,s)=\pm (1,1)$.
When $k_2=0$ there is no additional condition. We obtain the result.

If $(k_1,k_2,k_3)=(1,1,0)$ then we see that $\ell=0$, $p^2=1$, and $ps=1$. Hence, $(p,s)=\pm (1,1)$.
The case $(k_1,k_2,k_3)=(0,1,1)$ is handled similarly.

If $(k_1,k_2,k_3)=(0,0,\pm1)$ then we see that $\ell=0$ and $s^2=1$. Recall that $p\neq0$ follows from $M \in GL_2(\R)$.
%The case $(k_1,k_2,k_3)=(0,0,-1)$ is handled similarly.

If $(k_1,k_2,k_3)=(0,1,0)$ then we see that $ps=(-1)^\ell$. 

We turn to the case $C \in Z_0$. 
Then, $C$ is of the form
\[
	C = \begin{pmatrix} 0 & 0 & k_1  \\ 0 & 0 & k_2 \\ 0 & 0 & k_3   \end{pmatrix} = 
	\begin{pmatrix}  k_1 \\  k_2  \\  k_3  \end{pmatrix}
	\begin{pmatrix}  0 & 0 & 1  \end{pmatrix}
\]
with a nonzero vector $\ltrans{(k_1,k_2,k_3)} \in \R^3$.
By Lemma \ref{lem:redstep2}, $M$ is of the form
\[
	M= \begin{pmatrix}
	p & 0 \\ r & s
	\end{pmatrix}
\]
for $p,s, r\in \R$ with $ps\neq0$. Remark that
\[
	C' = \frac1{p^3 s} 	\begin{pmatrix} s^2 k_1 - 2 rs k_2 + r^2 k_3 \\ p(s  k_2 - r k_3)  \\ p^2 k_3  \end{pmatrix}
	\begin{pmatrix}  0 & 0 & 1  \end{pmatrix}.
\]

If $k_1 \in \{0, \pm 1 \}$, $k_2=0$, and $k_3=1$ then
we see from the second row and the third row of $S'$ that $ps=1$ and $r=0$.
When $k_1 \neq0$, we further need $s/p^3=1$. This implies $(p,s)=\pm (1,1)$.
When $k_1=0$, we need no additional condition.

If $k_1=k_3=0$ and $k_2\in \{ \pm 1 \}$ then
we see from the first row and the second row of $C'$ that $r=0$ and $p^2=1$.

If $k_1=1$ and $k_2=k_3=0$ then we see that $s/p^3=1$.

We finally consider the case $C \in Z_{-}$. 
Then, $C$ is of the form
\[
	C = 
	\begin{pmatrix}  \sin \theta \\  \cos \theta  \\  \sin \theta  \end{pmatrix}
	\begin{pmatrix}  1 & 0 & 1  \end{pmatrix}
\]
with $\theta \in [0,\pi/2]$. 
By Lemma \ref{lem:redstep2}, $M$ is of the form
\[
	M= R \begin{pmatrix}
	1 & 0 \\ 0 &  -1
	\end{pmatrix}^\ell
\begin{pmatrix}
	\cos \eta & -\sin \eta \\ \sin \eta &  \cos \eta
	\end{pmatrix}
\]
for some $R>0$ and $\eta \in \R/ 2\pi\Z$ and $\ell \in \{0,1\}$.
Remark that
\[
	C' = 	\frac1{R^2}\begin{pmatrix} \sin \theta -\cos \theta \sin 2\eta  \\   \cos \theta \cos 2\eta   \\ \sin \theta + \cos \theta \sin 2 \eta   \end{pmatrix}
	\begin{pmatrix}  0 & 0 & 1  \end{pmatrix}
\]
if $\ell=0$ and
\[
	C' = 	\frac1{R^2}	\begin{pmatrix} -\sin \theta-\cos \theta \sin 2\eta  \\   \cos \theta \cos 2\eta   \\ -\sin \theta+ \cos \theta \sin 2\eta   \end{pmatrix}
	\begin{pmatrix}  0 & 0 & 1  \end{pmatrix}
\]
if $\ell=1$.

If $\theta<\pi/2$ then we see that $\ell=0$, $R=1$, $\sin 2\eta =0$, and $\cos 2\eta =1$. Hence, $M\in \{ \pm E_2 \}$.

If $\theta=\pi/2$ then we see that $\ell=0$ and $R=1$.
Thus, we obtain the result.
\end{proof}

We are now in a position to complete the proof of Theorem \ref{thm:sysrank1}.
\begin{proposition}\label{Prop:Kj}
	For $1 \le j \le 9$, $K_j$ is a system of representatives of 
	%\[\R^3/G_j = \{  \{ y \in \R^3 \ |\ \exists M \in G_j \text{ s.t. } y=  (\det M) D(M)^{-1}x\} \ |\ x \in \R^3  \}.\]
	\[\R^3/G_j := \{  \{  (\det M)^{-1} D(M)x \ |\  M \in G_j \} \subset \R^3\ |\ x \in \R^3  \} .\]
\end{proposition}

\begin{proof}[{\bf Proof of Proposition \ref{Prop:Kj}}]
If $M\in GL_2(\R)$ is of the form
\[
	M = \begin{pmatrix} p & 0   \\ 0 & s   \end{pmatrix}
\]
with $ps\neq0$ then
\[
	\frac1{\det M} D(M) = \begin{pmatrix} 1/p^2  & 0 & 0 \\ 0 & 1/ps & 0 \\ 0 & 0 & 1/s^2  \\ \end{pmatrix}.
\]
Hence, one deduces from direct computation that $K_j = \R^3/G_j$.holds for $j=1,4,5,7,8$. 

Similarly, if $M\in GL_2(\R)$ is of the form
\[
	M = \begin{pmatrix} 0 & q  \\ r & 0   \end{pmatrix}
\]
with $qr\neq0$ then
\[
	\frac1{\det M} D(M) = \begin{pmatrix} 0 & 0 & 1/q^2 \\ 0 & 1/qr & 0 \\ 1/r^2 & 0 & 0  \\ \end{pmatrix}.
\]
Hence, together with the above formula, we see that $K_j = \R^3/G_j$ is true for $j=2,3,6$.

Finally, if $M\in GL_2(\R)$ is of the form
\[
	M= \begin{pmatrix} \cos \eta & - \sin \eta  \\  \sin \eta & \cos \eta \end{pmatrix} 
\]
for some $\eta \in \R/ 2\pi \Z$ then
\[
	\frac1{\det M} D(M) = 
	\begin{pmatrix} \cos^2 \eta & -2 \sin \eta \cos \eta & \sin^2 \eta \\ \sin \eta \cos \eta & \cos^2 \eta -\sin^2 \eta & -\sin \eta \cos \eta \\ \sin^2 \eta & 2\sin \eta \cos \eta & \cos^2\eta .
	\end{pmatrix}.
\]
This shows $K_9 = \R^3/G_9$. We leave the proof of the uniqueness part for readers.
\end{proof}

%%%%%%%%%%%%%%%%%%%%%
%
%  Appendix 2
%
%%%%%%%%%%%%%%%%%%%%%

\section{%Appendix 2: 
Examples}

In this appendix we exhibit two examples of system of the form \eqref{eq:NLS}.
%ここでは,  いくつかの system の例を紹介する.
We see that the nonlinear amplification/dissipation phenomenon takes place for the systems although the coefficients of the nonlinearities are real.
Further, the system admits several types of behaviors.

\subsection{Typical types  of behavior for the single NLS equation}
To begin with, we quickly recall several results on the single cubic NLS equation in one dimension.

\subsubsection{Modified scattering}  Let us consider 
\begin{equation}
i\partial_t u + \frac12\partial_x^2 u= \la |u|^2u, \label{NLS11}
\end{equation}
where $\lambda \in \R$. 
Due to Hayashi-Naumkin \cite{HN}, if $u(0)$ is sufficiently small 
in $H^{1}\cap H^{0,1}$ and $0<\gamma<1/100$, then there exists a global 
in time solution to (\ref{NLS11}) satisfying
\begin{align*}
u(t)
=t^{-\frac12}W\left(\frac{x}{t}\right)e^{\frac{ix^2}{2t}
-\l i\left|W\left(\frac{x}{t}\right)\right|^2\log t-i\frac{\pi}{4}}
+O(t^{-\frac34+\gamma})
\end{align*}
in $L^{\infty}$ as $t\to\infty$ for some $W\in L^{\infty}$. 
The asymptotic profile is that of the free solution 
with a phase correction by nonlinear effect. 

\subsubsection{Nonlinear dissipation}
%Next, we consider the case $u_{2,0}=e^{-i\pi/4}u_{1,0}$. 
%By the uniqueness of solution to (\ref{eq:2NLS}), we obtain that $u_2=e^{-i\pi/4}u_1$ 
%and $u_1$satisfies 
Let us consider
\begin{equation}
i\partial_t u + \frac12\partial_x^2 u=  - i|u|^2u. \label{NLS12}
\end{equation}
By Shimomura \cite{S}, if $u(0)$ is sufficiently small in $H^{1}\cap H^{0,1}$ 
and $0<\gamma<1/100$, then there exists a global in time solution to (\ref{NLS12}) satisfying
\begin{align}\label{eq:d-profile}
u(t)
=\frac{t^{-\frac12}}{
\sqrt{1+\left|W\left(\frac{x}{t}\right)\right|^2\log t}}
W\left(\frac{x}{t}\right)e^{\frac{ix^2}{2t}-i\frac{\pi}{4}}
+O(t^{-\frac34+\gamma})
\end{align}
in $L^{\infty}$ as $t\to\infty$ for some $W\in L^{\infty}$. 
Note that the rate of time decay of $u$ is faster than that of free solution
due to nonlinear dissipation. 

\subsubsection{Nonlinear amplification}
%We consider the case $u_{2,0}\equiv e^{i\pi/4}u_{1,0}$. 
%In this case, we find that $u_2=e^{i\pi/4}u_1$ and $u_1$ satisfies 
Let us consider
\begin{equation}
i\partial_t u + \frac12\partial_x^2 u= i|u|^2u. \label{NLS13}
\end{equation}
By Kita \cite{K}, for any $\varepsilon>0$ there exists 
%an initial data $u_{1,0}\in L^2$ such that $\|u_{1,0}\|_{L^2}<\varepsilon$ 
%such that the corresponding solution 
a solution
$u_\eps \in C([0,T^{\ast}), H^1 \cap H^{0,1})$ to (\ref{NLS13})  which satisfies $\norm{u_\eps(0)}_{L^2} < \eps$ and
blows up in finite time, i.e.,
\begin{equation*}
\lim_{t\nearrow T^{\ast}}\|u_\eps(t)\|_{L^2}=\infty
\end{equation*}
for some $T^{\ast}<\I$. 
%Note that this result is based on the nonlinear amplification. 
This is due to nonlinear amplification.
Note that for any solution to \eqref{NLS13}, $\ol{u(-t,x)}$ is a solution to \eqref{NLS12}.
Hence, solutions to \eqref{NLS13} have a dissipative structure for the  negative time direction.
Here we remark that the above forward-blowing-up solution $u_\eps$ is global backward in time and satisfies $\norm{u_\eps(t)}_{L^2}=O((\log |t|)^{-1/3})$ as $t\to-\I$.
It is not known whether the asymptotic profile of the solution is like the one given \eqref{eq:d-profile}.
Nevertheless, one can find a dissipative effect in the decay of the mass.

\subsection{Example 1}

The first example is the system which admits solution of the following three types; modified scattering, nonlinear dissipation, nonlinear amplification.
Let us consider
\begin{equation}
\label{eq:2NLS}
\left\{
\begin{aligned}
i\partial_t u_1 + \frac12\partial_x^2 u_1 &= -2(|u_1|^2-|u_2|^2)u_1 + \ol{u_1}u_2^2, 
&&t \in \R,\ x \in \R,\\
i\partial_t u_2 + \frac12\partial_x^2 u_2 &= -2(|u_1|^2-|u_2|^2)u_2 - u_1^2\ol{u_2},
&&t \in \R,\ x \in \R.
\end{aligned}
\right.
\end{equation}
In the terminology introduced in Appendix A, the system belongs to $\mathrm{CNS}_A$, that is,
the corresponding matrix-kernel representation $(C,p,q,r)$ given in Theorem \ref{thm:mkrep} satisfies
 $\tr C=0$ and $p=q=r=0$ and 
\[
	C = A = \begin{pmatrix}
	0 & 1& 0 \\
	1 & 0 & 1 \\
	0 & 1& 0
	\end{pmatrix}.
\]

One easily sees that a pair $(u_1(t),0)$ is a solution to the system if and only if $u_1(t)$ solves
\begin{equation*}
i\partial_t u_1 + \frac12\partial_x^2 u_1= -2 |u_1|^2u_1. %\label{NLS11}
\end{equation*}
Hence, the asymptotic behavior of a small solution of this form is a type of 
the modified scattering for both time directions.
Similarly, a pair $(0,u_2(t))$ is also a solution with the same kind of behavior. However, the sign of the nonlinearity is opposite, which corresponds to the time direction is reversed.

One also sees that a pair $(u(t),e^{-i\pi/4} u(t))$ becomes a solution if and only if $u(t)$ solves
\begin{equation*}
i\partial_t u + \frac12\partial_x^2 u= -i |u|^2u. %\label{NLS11}
\end{equation*}
Hence, the solution to \eqref{eq:2NLS} of this form enjoys the dissipative decay forward in time.
If it is sufficiently small (in $H^1 \cap H^{0,1}$) at some time then the asymptotic behavior is as in \eqref{eq:d-profile}
as $t\to\I$.
Further, there exists an arbitrarily small (in $L^2$) initial data of the form $(u(0),e^{-i\pi/4} u(0) )$ such that the corresponding 
solution blows up for the negative time direction in a finite time.

Finally, one verifies that a pair $(u(t),e^{i\pi/4} u(t))$ becomes a solution if and only if $u(t)$ solves
\begin{equation*}
i\partial_t u + \frac12\partial_x^2 u= i |u|^2u. %\label{NLS11}.
\end{equation*}
Hence, the solution to \eqref{eq:2NLS} of this form enjoys the same property as above with the time directions switched.

\subsection{Example 2}
The second example is
\begin{equation}
\label{eq:3NLS}
\left\{
\begin{aligned}
i\partial_t u_1 + \frac12\partial_x^2 u_1 &= 2|u_1|^2 u_2 - u_1^2 \ol{u_2} +  |u_2|^2u_2, \\
i\partial_t u_2 + \frac12\partial_x^2 u_2 &= - |u_1|^2u_1 - 2u_1|u_2|^2 +  \ol{u_1} u_2^2.
\end{aligned}
\right.
\end{equation}
The matrix-kernel representation of the system is
\[
	\( \begin{pmatrix}
3 & 0 & 1 \\
0 & 2 & 0 \\
1 & 0 & 3
\end{pmatrix} , 0, 0,0 \) \in M_3 (\R) \times \R^3.
\]
We apply the following linear transform of unknowns with \emph{complex coefficients}:
\[
	\begin{pmatrix} v_1 \\ v_2 \end{pmatrix} 
	=  \begin{pmatrix} 1 &- i \\ -1 & -i \end{pmatrix}  \begin{pmatrix} u_1 \\ u_2 \end{pmatrix}. 
\]
Then, one sees that the system for $(v_1,v_2)$ becomes
\[
\left\{
\begin{aligned}
i\partial_t v_1 + \frac12\partial_x^2 v_1 &= i |v_1|^2 v_1, \\
i\partial_t v_2 + \frac12\partial_x^2 v_2 &= -i |v_2|^2v_2.
\end{aligned}
\right.
\]
This is a decoupled system. It is easy to see that there exist a solution of the form $(v_1(t),0)$ which 
blows up forward in time and decays dissipatively backward in time
and a solution of the form $(0,v_2(t))$ which 
blows up backward in time and decays dissipatively forward in time.
Remark that a sum of these two solutions are also a solution to the system. The solution blows up for the both time directions.

\begin{remark}
The second example suggests that the property $M \in GL_2(\R)$ in the linear transformation \eqref{E:linearchange}
is restrictive for revealing the dissipation/amplification phenomenon of the system \eqref{eq:NLS},
or that,
since the matrix-kernel representation (Theorem \ref{thm:mkrep}) is based on the connection to the conserved
quantities of the systems, it is not well-suited to for the purpose.
%この例は, Appendix A で述べた分類において設けている制限である変数変換の係数が実数であるという制約が,
%dissipation/amplification の構造を見るには十分でないことを示唆する.
%これは, matrix-kernel representation が方程式の保存則と関係がある表示であることに起因すると考えられる.
%複素数の係数に関する変数変換が作る同値関係に対してもいくつかの行列表示で
% \eqref{eq:Cprime} のように同値関係を明示するものは存在することは確かめられる.
% しかし, それらが微分方程式の特徴的な性質とも結びついているものは, 現在は見つかっていない.
 It is not hard to construct variants of the matrix-kernel representation for which the change of (possibly complex-valued) coefficients caused by a linear transformation of unknowns with $M \in GL_2(\C)$ is formulated as a similar matrix manipulation.
 However, as far as the authors know, a more useful one is not found.
\end{remark}

%We consider the case where 
%$\la_1=-2/3, \la_2=0, \la_3=1, \la_4=\la_5=0, \la_6=-1, \la_{7}=0, 
%\la_{8}=2/3$ in (\ref{eq:NLS}), i.e., 
%Although, $\rank A =2$ for (\ref{eq:2NLS}), we are not able to apply \label{T:main3} to (\ref{eq:2NLS}). 
%However we see 
% \begin{equation}
%\rank A = 
%\rank \begin{pmatrix}
% 0 & -\frac32 & 0 \\
% \frac32 & 0 & \frac32 \\
% 0 & \frac32 & 0
% \end{pmatrix}=2. 
%\end{equation}
%In this case, we see that there are many types of solutions even for small initial data 
%from the following observations. 

%\vskip2mm
%\noindent
%{\bf Example 2.} $\la_5=\la_9=\la\neq0$, その他の係数がすべて$0$のとき, つまり
%\begin{equation}
%\label{eq:3NLS}
%\left\{
%\begin{aligned}
%i\partial_t u_1 + \frac12\partial_x^2 u_1 &= \la \overline{u_1}u_2^2, 
%&&t \in \R,\ x \in \R,\\
%i\partial_t u_2 + \frac12\partial_x^2 u_2 &= \la u_1^2\overline{u_2},
%&&t \in \R,\ x \in \R.
%\end{aligned}
%\right.
%\end{equation}
%Okamoto-Uriya \cite{OU} constructed a solution to (\ref{eq:3NLS}) which 
%behaves like free solution with nonlinear dissipation as $t\to\infty$.   

\subsection*{Acknowledgements} 
%The authors express their gratitude to Professors Kenji Nakanishi and Hideaki Sunagawa for valuable
%comments on a preliminary version of the manuscript.
S.M. was supported by JSPS KAKENHI Grant Numbers JP17K14219, JP17H02854, JP17H02851, and JP18KK0386. J.S was supported by JSPS KAKENHI Grant Numbers 
JP17H02851, JP19H05597, JP20H00118, JP21H00993, and  JP21K18588.  
K.U was supported by JSPS KAKENHI Grant Number JP19K14578.

 %%%%%%%%%%%%%%%%%%%%%%%%%%%%%%%%%%%%%%%%%%%%%%%%%%%%
%
%  References
%
%%%%%%%%%%%%%%%%%%%%%%%%%%%%%%%%%%%%%%%%%%%%%%%%%%%%

\end{document}